\documentclass[11pt]{article}
\usepackage{amssymb}
\usepackage{amsbsy}
\usepackage[latin1]{inputenc}
\usepackage{amsthm}
\usepackage[dvips]{graphicx}
\usepackage{graphicx} 
\usepackage{subfigure}
\usepackage{pst-eucl}
\usepackage[latin1]{inputenc}
\usepackage[english]{babel}
\usepackage{amsmath,amssymb,graphics,mathrsfs}
\usepackage{amsmath,amssymb,latexsym}
\usepackage{graphicx,color}
\usepackage[T1]{fontenc}
\usepackage[active]{srcltx}
\usepackage{multicol}
\usepackage[latin1]{inputenc}
\usepackage{pst-all}
\usepackage{enumerate}
\usepackage{pstricks}
\usepackage{pstricks-add}
\usepackage{setspace}
\usepackage{soul}
\usepackage{cancel}
\usepackage{nonfloat}
\usepackage[margin=10pt,font=footnotesize,labelfont=bf,labelsep=endash]{caption}
\usepackage[left=4cm,top=3cm,right=2.4cm,bottom=3.2cm]{geometry}
\usepackage[colorlinks=true,citecolor=red,linkcolor=blue,urlcolor=RubineRed,pdfpagetransition=Blinds,pdftoolbar=false,pdfmenubar=false]{hyperref}

\newcommand{\R}{\hbox{\rm I \kern -5pt R}}     
\newcommand{\p} {\hbox{\rm I \kern -5pt P}}

\def\x        {\textit{\textbf{x}}}

\def\H        {{\boldsymbol H}}
\def\L       {{\boldsymbol L}}

\newtheorem{prop}{Proposition}[section]
\newtheorem{defi}[prop]{Definition}
\newtheorem{tma}[prop]{Theorem}
\newtheorem{cor}[prop]{Corollary}
\newtheorem{obs}[prop]{Remark}
\newtheorem{lem}[prop]{Lemma}

\allowdisplaybreaks

\begin{document}

\title{Analysis of a chemo-repulsion model with nonlinear production: The continuous problem and unconditionally energy stable fully discrete schemes}
\author{F.~Guillén-González\thanks{Dpto. Ecuaciones Diferenciales y An\'alisis Num\'erico and IMUS, 
Universidad de Sevilla, Facultad de Matemáticas, C/ Tarfia, S/N, 41012 Sevilla (SPAIN). Email: guillen@us.es, angeles@us.es},
M.~A.~Rodríguez-Bellido$^*$~and
 D.~A.~Rueda-Gómez$^*$\thanks{Escuela de Matemáticas, Universidad Industrial de Santander, A.A. 678, Bucaramanga (COLOMBIA). Email:  diaruego@uis.edu.co}}

\date{}
\maketitle

\begin{abstract}
We consider the following repulsive-productive chemotaxis model: Let $p\in (1,2)$, find $u \geq 0$, the cell density, and $v \geq 0$, the chemical concentration, satisfying
\begin{equation}\label{C5:Am}
\left\{
\begin{array}
[c]{lll}%
\partial_t u - \Delta u - \nabla\cdot (u\nabla v)=0 \ \ \mbox{ in}\ \Omega,\ t>0,\\
\partial_t v - \Delta v + v =  u^p \ \ \mbox{ in}\ \Omega,\ t>0,
\end{array}
\right.
\end{equation}
in a bounded domain $\Omega\subseteq \mathbb{R}^d$, $d=2,3$. By using a regula\-ri\-za\-tion technique, we prove the existence of solutions for problem (\ref{C5:Am}). Moreover, we propose three fully discrete Finite Element (FE) nonlinear approximations of problem (\ref{C5:Am}), where the first one is defined in the variables $(u,v)$, and the second and third ones introduce ${\boldsymbol\sigma}=\nabla v$ as auxiliary variable. We prove some unconditional properties such as mass-conservation, energy-stability and solvability of the schemes. Finally, we compare the behavior of these schemes throughout several numerical simulations and give some conclusions.
\end{abstract}

\noindent{\bf 2010 Mathematics Subject Classification.}  35K51, 35Q92, 65M12, 65M60, 92C17. 

\noindent{\bf Keywords: } Chemorepulsion-production model, finite element approximation, unconditional energy-stability, nonlinear production. 

\section{Introduction}
Chemotaxis is the biological process of the movement of living organisms in res\-pon\-se to a chemical stimulus, which can be given towards a higher (chemo-attraction) or lower (chemo-repulsion) concentration of a chemical substance. At the same time, the presence of living organisms can produce or consume chemical substance. A repulsive-productive chemotaxis model can be given by the following parabolic PDE's system:
$$
\left\{
\begin{array}
[c]{lll}%
\partial_t u - \Delta u = \nabla\cdot (u\nabla v)\quad \mbox{ in}\ \Omega,\ t>0,\\
\partial_t v - \Delta v + v =  f(u)
 \quad \mbox{ in}\ \Omega,\ t>0,
\end{array}
\right.$$
where $\Omega\subseteq \mathbb{R}^d$, $d=2,3$, is a bounded domain with boundary $\partial \Omega$. The unknowns for this model are $u(\x , t) \geq 0$, the cell density, and $v(\x , t) \geq 0$, the chemical concentration. Moreover, $f(u)\geq 0$ (if $u\geq 0$) is the production term. In this paper, we consider the particular case in which $f(u)=u^p$, with $1<p<2$, and then we focus on the following initial-boundary value problem:
\begin{equation}
\left\{
\begin{array}
[c]{lll}%
\partial_t u - \Delta u = \nabla\cdot (u\nabla v)\ \ \mbox{ in}\ \Omega,\ t>0,\\
\partial_t v - \Delta v + v =  u^{p} \ \ \mbox{ in}\ \ \Omega,\ t>0,\\
\displaystyle\frac{\partial u}{\partial \mathbf{n}}=\frac{\partial v}{\partial \mathbf{n}}=0\ \ \mbox{on}\ \partial\Omega,\ t>0,\\
u(\x,0)=u_0(\x)>0,\ v(\x,0)=v_0(
\x)>0\ \ \mbox{in}\ \Omega.
\end{array}
\right.  \label{C5:modelf00}
\end{equation}
In the case of linear ($p=1$) and quadratic ($p=2$) production terms, the problem (\ref{C5:modelf00}) is well-posed (see \cite{C5:Cristian,C5:FMD} respectively) in the following sense: there exist global in time weak solutions (based on an energy inequality) and, for $2D$ domains, there exists a unique global in time strong solution. However, as far as we know,  there are not works studying problem  
(\ref{C5:modelf00}) with production $u^p$, with $1<p<2$.\\
Problem (\ref{C5:modelf00}) is conservative in $u$, because the total mass $\int_\Omega u(\cdot,t)$ remains constant in time, as we can check  integrating equation (\ref{C5:modelf00})$_1$ in $\Omega$, 
\begin{equation}\label{C5:consucont}
\frac{d}{dt}\left(\int_\Omega u(\cdot, t)\right)=0, \ \ \mbox{ i.e. } \ 
\int_\Omega u(\cdot,t)=\int_\Omega u_0 := m_0, \ \ \forall t>0.
\end{equation}
The first aim of this work is to study the existence of weak-strong solutions for problem (\ref{C5:modelf00}) (in the sense of Definition \ref{C5:ws00p} below), satisfying in particular the energy inequality (\ref{C5:wsd}) below. The second aim of this work is to design numerical methods for model (\ref{C5:modelf00}) conserving, at the discrete level, the mass-conservation and energy-stability properties of the continuous model (see (\ref{C5:consucont}) and (\ref{C5:wsd}), respectively).\\

There are only a few works about numerical analysis for  chemotaxis models. For instance, for the Keller-Segel system (i.e.~with chemo-attraction and linear production),  in  \cite{C5:Filbet} Filbet proved the existence of discrete solutions and the convergence of a finite volume scheme. Saito, in \cite{C5:Saito1,C5:Saito2}, studied error estimates for a conservative Finite Element (FE) approximation. In \cite{C5:Eps}, some error estimates are proved for a fully discrete discontinuous FE method, and a mixed FE approximation is studied in \cite{C5:Marrocco}.\\

Energy stable numerical schemes have also been studied in the chemotaxis framework. An energy-stable finite volume scheme for a Keller-Segel model with an additional cross-diffusion term has been studied in \cite{C5:BJ}. In \cite{C5:FMD,C5:FMD2}, unconditionally energy stable time-discrete numerical schemes and fully discrete FE schemes for a chemo-repulsion model with quadratic production have been analyzed. In \cite{C5:FMD4}, the authors studied unconditionally energy stable fully discrete FE schemes for a chemo-repulsion model with linear production.  However, as far as we know, for the chemo-repulsion model with production term $u^p$ (\ref{C5:modelf00}) there are not works studying energy-stable numerical schemes.\\

The outline of this paper is as follows: In Section \ref{C5:S2C5:NPR}, we give the notation and some preliminary results that will be used throughout the paper. In Section \ref{C5:S3C5:CM}, we prove the existence of weak-strong solutions of model (\ref{C5:modelf00}) (in the sense of Definition \ref{C5:ws00p} below) by using a regularization technique. In Section \ref{C5:S4C5:NS}, we propose three fully discrete FE nonlinear approximations of problem (\ref{C5:modelf00}), where the first one is defined in the variables $(u,v)$, and the second and third ones introduce ${\boldsymbol\sigma}=\nabla v$ as an auxiliary variable. We prove some unconditional properties such as mass-conservation, energy-stability and solvability of the schemes. In Section \ref{C5:S5C5:NSi}, we compare the behavior of the schemes throughout several numerical simulations; and in Section \ref{C5:S6C5:C}, the main conclusions obtained in this paper are sumarized.

\section{Notation and preliminary results}\label{C5:S2C5:NPR}
We recall some functional spaces which will be used throughout this paper. We will consider the usual  Lebesgue spaces $L^q(\Omega),$
$1\leq q\leq \infty,$ with norm $\Vert\cdot \Vert_{L^q}$. In particular,  the $L^2(\Omega)$-norm will be denoted by $\Vert
\cdot\Vert_0$. From now on, $(\cdot,\cdot)$ will denote the standard $L^2$-inner product over $\Omega$. We also consider the usual Sobolev
spaces $W^{m,p}(\Omega) = \{u \in  L^p(\Omega) : \Vert \partial^\alpha u\Vert_{L^p} < +\infty, \ \forall\vert \alpha \vert \le m\}$, for a multi-index $\alpha$ and $m \in \mathbb{N}$,
with norm denoted by $\Vert \cdot\Vert_{W^{m,p}}$. In the case when $p=2$, we denote $H^m(\Omega):= W^{m,2}(\Omega)$, with respective  norm $\Vert\cdot\Vert_{m}$. Moreover, we denote by 
$$W^{m,p}_{\mathbf{n}}(\Omega) := \left\{u \in W^{m,p}(\Omega) : \frac{\partial{u}}{\partial \mathbf{n}}=0 \ \mbox{ on } \ \partial\Omega\right\},$$
$$\H^{1}_{\sigma}(\Omega):=\{\mathbf{\boldsymbol\sigma}\in \H^{1}(\Omega): \mathbf{\boldsymbol\sigma}\cdot \mathbf{n}=0 \mbox{ on } \partial\Omega\},$$ 
and we will use the following equivalent norms in $H^1(\Omega)$  and ${\H}_{\sigma}^1(\Omega)$, respectively (see \cite{C5:necas} and \cite[Corollary 3.5]{C5:Nour}, respectively):
\begin{equation*}
\Vert u \Vert_{1}^2=\Vert \nabla u\Vert_{0}^2 + \left( \int_\Omega u\right)^2, \ \ \forall u\in H^1(\Omega),
\end{equation*}
\begin{equation}\label{C5:H1div}
\Vert {\boldsymbol\sigma} \Vert_{1}^2=\Vert {\boldsymbol\sigma}\Vert_{0}^2 + \Vert \mbox{rot }{\boldsymbol\sigma}\Vert_0^2 + \Vert \nabla \cdot {\boldsymbol\sigma}\Vert_0^2, \ \ \forall {\boldsymbol\sigma}\in \H^{1}_{\sigma}(\Omega),
\end{equation}
where rot ${\boldsymbol\sigma}$ denotes the well-known rotational operator (also called curl) which is scalar for 2D domains and vectorial for 3D ones. In particular, (\ref{C5:H1div}) implies that, for all $\boldsymbol\sigma=\nabla v\in \H^{1}_{\sigma}(\Omega)$,
\begin{equation}\label{C5:H1divGrad}
\Vert\nabla v\Vert_{1}^2=\Vert \nabla v\Vert_{0}^2  + \Vert \Delta v\Vert_0^2.
\end{equation}
If $Z$ is a
general Banach space, its topological dual space will be denoted by $Z'$.
Moreover, the letters $C,K$ will denote different positive
constants which may change from line to line. \\
We will use the following results:
\begin{tma} (\cite{C5:Fe})\label{C5:nn}
Let $1<q<+\infty$ and suppose that $f\in L^q(0,T; L^q(\Omega))$, $u_0\in \widehat{W}^{2-\frac{2}{q},q}(\Omega)$, where
$$
\widehat{W}^{2-\frac{2}{q},q}(\Omega)\quad := \quad 
\left\{\begin{array}{l}
{W}^{2-\frac{2}{q},q}(\Omega) \ \ \mbox{ if }\ \ q<3 ,\\
{W}_{\mathbf{n}}^{2-\frac{2}{q},q}(\Omega) \ \ \mbox{ if }\ \ q>3.\\
\end{array}\right.
$$
Then, the problem
$$
\left\{
\begin{array}
[c]{lll}%
\partial_t u - \Delta u = f\ \ \mbox{ in}\ \Omega,\ t>0,\\
\displaystyle\frac{\partial u}{\partial \mathbf{n}}=0\ \ \mbox{on}\ \partial\Omega,\ t>0,\\
u({\x},0)=u_0({\x})\ \ \mbox{in}\ \Omega,
\end{array}
\right. 
$$
admits a unique solution $u$ in the class
$$
u\in L^q(0,T; W^{2,q}(\Omega)) \cap C([0,T];\widehat{W}^{2-\frac{2}{q},q}(\Omega)), \ \ \partial_t u \in L^q(0,T; L^q(\Omega)).
$$
Moreover, there exists a positive constant $C=C(q,\Omega,T)$ such that
$$
\Vert u\Vert_{C([0,T];\widehat{W}^{2-\frac{2}{q},q}(\Omega))} + \Vert \partial_t u\Vert_{L^q(0,T; L^q(\Omega))} + \Vert u\Vert_{L^q(0,T; W^{2,q}(\Omega))} \leq C(\Vert f\Vert_{L^q(0,T; L^q(\Omega))} + \Vert u_0\Vert_{\widehat{W}^{2-\frac{2}{q},q}(\Omega)}).
$$
\end{tma}

\begin{prop}(\cite{C5:GA})
Let $X$ be a Banach space, $\Omega\subseteq X$ an open subset,
$U\subseteq \Omega$ a nonempty convex subset and $J:
\Omega\rightarrow \mathbb{R}$ a functional. Suppose that $J$ is
$G-$differentiable in $\Omega$. Then, $J$ is convex over $U$ if and only if the following relation holds
\begin{equation}\label{C5:Gdif2}
J(x_1)-J(x_2)\leq \delta J(x_1,x_1-x_2), \ \forall x_1,x_2\in U, \
x_1\neq x_2.
\end{equation}
\end{prop}

Finally, we will use the following result to get large time estimates \cite{C5:He}:
\begin{lem} \label{C5:tmaD}
Assume that $\delta,\beta,k>0$ and $d^n\geq 0$ satisfy
\begin{equation*}
 (1+\delta k)d^{n+1} \leq d^n + k\beta, \ \ \ \forall n\geq 0.
\end{equation*}
Then, for any $n_0\geq 0$,
\begin{equation*}
d^n \leq (1+\delta k)^{-(n-n_0)} d^{n_0} + \delta^{-1} \beta, \ \ \  \forall n\geq n_0.
\end{equation*}
\end{lem}

\section{Analysis of the continuous model}\label{C5:S3C5:CM}
In this section, we will prove the existence of weak-strong solutions of problem (\ref{C5:modelf00}) in the sense of the following definition.
\begin{defi} \label{C5:ws00p}{\bf (Weak-strong solutions of (\ref{C5:modelf00}))} 
Let $1<p<2$. Given $(u_0, v_0)\in L^p(\Omega)\times H^1(\Omega)$ with $u_0\geq 0$, $v_0\geq 0$ a.e. \hspace{-0.3 cm} in $\Omega$, a pair $(u,v)$ is called weak-strong solution of problem (\ref{C5:modelf00}) in $(0,+\infty)$, if $u\geq 0$, $v\geq 0$ a.e. \hspace{-0.3 cm} in $(0,+\infty)\times \Omega$,
\begin{equation*}\label{C5:wsa}
\begin{array}{ccc}
u \in L^{\infty}(0,+\infty;L^p(\Omega))  \cap L^{\frac{5p}{p+3}}(0,T;W^{1,\frac{5p}{p+3}}(\Omega)),  \ \ \forall T>0,
\\
v \in L^{\infty}(0,+\infty;H^1(\Omega))  \cap L^{2}(0,T; H^2(\Omega)),  \ \ \forall T>0,
\\
\partial_t u \in L^{\frac{10p}{3p+6}}(0,T;W^{1,\frac{10p}{7p-6}}(\Omega)'), \ \ \partial_t v \in L^{\frac{5}{3}}(0,T;L^{\frac{5}{3}}(\Omega)), \ \ \forall T>0,
\end{array}
\end{equation*}
the following variational formulation for the $u$-equation holds
\begin{equation}\label{C5:wf01}
\int_0^T \langle \partial_t u,\bar{u}\rangle + \int_0^T (\nabla u,  \nabla \bar{u}) +\int_0^T (u\nabla v,\nabla \bar{u})=0, \ \ \forall \bar{u}\in L^{\frac{10p}{7p-6}}(0,T;W^{1,\frac{10p}{7p-6}}(\Omega)), \ \ \forall T>0,
\end{equation}
the $v$-equation holds pointwisely
\begin{equation}\label{C5:wf02}
\partial_t v -\Delta v + v=u^p \ \ \mbox{ a.e. } (t,\x)\in (0,+\infty)\times\Omega,
\end{equation}
the boundary condition $\displaystyle\frac{\partial v}{\partial \mathbf{n}}=0$ and the initial conditions $(\ref{C5:modelf00})_4$ are satisfied, and the following energy inequality (in integral version) holds  for a.e.~$t_0,t_1$ with $t_1\geq t_0\geq 0$:
\begin{equation}\label{C5:wsd}
\mathcal{E}(u(t_1),v(t_1)) - \mathcal{E}(u(t_0),v(t_0))
 + \int_{t_0}^{t_1} \left( \frac{4}{p}\Vert \nabla (u^{p/2}(s))  \Vert_{0}^2 + \Vert \nabla v(s) \Vert_{1}^2
 \right)\ ds \leq0,
\end{equation}
 where
\begin{equation}\label{C5:eneruva}
\mathcal{E}(u,v)=\displaystyle
\frac{1}{p-1}\Vert u\Vert_{p}^p + \frac{1}{2}\Vert \nabla v\Vert_{0}^{2}.
\end{equation}
\end{defi}

Observe that any weak-strong solution of (\ref{C5:modelf00}) is conservative in $u$ (see (\ref{C5:consucont})). In addition, integrating (\ref{C5:modelf00})$_2$ in $\Omega$, we deduce 
\begin{equation}\label{C5:nuevo-1}
\frac{d}{dt}\left(\int_\Omega v\right)  + \int_\Omega
v=\int_\Omega u^p.
\end{equation}

\subsection{Regularized problem}
In order to prove the existence of weak-strong solution of problem (\ref{C5:modelf00}) in the sense of Definition \ref{C5:ws00p}, we introduce the following regularized problem associated to model (\ref{C5:modelf00}): Let $\varepsilon\in (0,1)$, find $(u^\varepsilon, z^\varepsilon)$, with $u^\varepsilon\geq 0$ a.e. in $(0,+\infty)\times \Omega$, such that, for all $T>0$,
\begin{equation}\label{C5:wsaREG}
u^\varepsilon, z^\varepsilon \in \widetilde{\mathcal{X}}:= \{w\in L^\infty(0,T; W^{\frac{4}{5},\frac{5}{3}}(\Omega))\cap L^{\frac{5}{3}}(0,T; W^{2,\frac{5}{3}}(\Omega)): \partial_t w\in L^{\frac{5}{3}}(0,T; L^{\frac{5}{3}}(\Omega)) \},
\end{equation}
and
\begin{equation}
\left\{
\begin{array}
[c]{lll}%
\partial_t u^\varepsilon - \Delta u^\varepsilon = \nabla\cdot (u^\varepsilon\nabla v(z^\varepsilon))\ \ \mbox{in}\ \Omega,\ t>0,\\
\partial_t z^\varepsilon - \Delta z^\varepsilon + z^\varepsilon =  (u^\varepsilon)^{p} \ \mbox{in}\ \ \Omega,\ t>0,\\
\displaystyle\frac{\partial u^\varepsilon}{\partial \mathbf{n}}=\frac{\partial z^\varepsilon}{\partial \mathbf{n}}=0\ \ \mbox{on}\ \partial\Omega,\ t>0,\\
u^\varepsilon(\x,0)=u^\varepsilon_0(\x)\geq 0,\ z^\varepsilon(\x,0)=v^\varepsilon_0(
\x)- \varepsilon\Delta v^\varepsilon_0(\x)\ \ \mbox{in}\ \Omega,
\end{array}
\right.  \label{C5:modelf00reg}
\end{equation}
where $v^\varepsilon=v(z^\varepsilon)$ is the unique solution of the elliptic-Newman problem 
\begin{equation}\label{C5:CVar1}
\left\{
\begin{array}
[c]{lll}
v^\varepsilon - \varepsilon\Delta v^\varepsilon=z^\varepsilon \ \ \mbox{in } \Omega,\\
\displaystyle\frac{\partial v^\varepsilon}{\partial \mathbf{n}}=0\ \ \mbox{on}\ \partial\Omega,
\end{array}
\right. 
\end{equation}
and $(u^\varepsilon_0,z^\varepsilon_0) \in W^{\frac{4}{5},\frac{5}{3}}(\Omega)^2$  with
\begin{equation}\label{C5:ccp}
(u^\varepsilon_0,z^\varepsilon_0) \rightarrow (u_0,z_0) \ \ \mbox{ in } L^2(\Omega)\times L^2(\Omega), \ \mbox{ as } \varepsilon\rightarrow 0. 
\end{equation}
Taking into account (\ref{C5:wsaREG}), system (\ref{C5:modelf00reg}) is satisfied a.e. \hspace{-0.3 cm} in $(0,+\infty)\times \Omega$. From now on in this section, we will denote $v^\varepsilon(z^\varepsilon)$ solution of (\ref{C5:CVar1}) only by $v^\varepsilon$. Observe that if $(u^\varepsilon,z^\varepsilon)$ is any solution of (\ref{C5:modelf00reg}), then (\ref{C5:consucont}) and (\ref{C5:nuevo-1}) are satisfied for $(u,v)=(u^\varepsilon,v^\varepsilon)$.

\begin{tma}
Let $\varepsilon\in (0,1)$. Then, there exists at least one solution of problem (\ref{C5:wsaREG})-(\ref{C5:modelf00reg}).
\end{tma}
\begin{proof}
We will use the Leray-Schauder fixed point theorem. With this aim, we denote 
$$\mathcal{X}:= L^\infty(0,T; L^2(\Omega))\cap L^2(0,T; H^1(\Omega)),$$ 
and we define the operator $R:\mathcal{X}\times \mathcal{X}\rightarrow \widetilde{\mathcal{X}}\times \widetilde{\mathcal{X}}\hookrightarrow \mathcal{X}\times \mathcal{X}$ by 
$R(\widetilde{u}^\varepsilon,\widetilde{z}^\varepsilon)=(u^\varepsilon,z^\varepsilon)$, such that $(u^\varepsilon,z^\varepsilon)$ solves the following linear decoupled problem
\begin{equation}
\left\{
\begin{array}
[c]{lll}%
\partial_t u^\varepsilon - \Delta u^\varepsilon = \nabla\cdot (\widetilde{u}^\varepsilon_+\nabla\widetilde{v}^\varepsilon)\ \ \mbox{in}\ \Omega,\ t>0,\\
\partial_t z^\varepsilon - \Delta z^\varepsilon =  (\widetilde{u}^\varepsilon)^{p} -  \widetilde{z}^\varepsilon \ \mbox{in}\ \ \Omega,\ t>0,\\
\displaystyle\frac{\partial u^\varepsilon}{\partial \mathbf{n}}=\frac{\partial z^\varepsilon}{\partial \mathbf{n}}=0\ \ \mbox{on}\ \partial\Omega,\ t>0,\\
u^\varepsilon(\x,0)=u^\varepsilon_0(\x)\geq 0,\ z^\varepsilon(\x,0)=v^\varepsilon_0(
\x)- \varepsilon\Delta v^\varepsilon_0(\x)\ \ \mbox{in}\ \Omega,
\end{array}
\right.  \label{C5:modelfexist0a}
\end{equation}
where  $\widetilde{v}^\varepsilon={v}(\widetilde{z}^\varepsilon)$ and, in general, we denote $a_+:= \max\{a,0 \}$. Then, $(u^\varepsilon,z^\varepsilon)$ is a solution of 
(\ref{C5:modelf00reg}) iff $(u^\varepsilon,z^\varepsilon)$ is a fixed point of  the operator $R$ defined in (\ref{C5:modelfexist0a}). Let us check every hypotheses of Leray-Schauder Theorem:

\begin{enumerate}
\item{$R$ is well defined}. Observe that if $\widetilde{z}_\varepsilon \in \mathcal{X}$, from the $H^2$ and $H^3$-regularity of problem (\ref{C5:CVar1}) (see \cite[Theorems 2.4.2.7 and 2.5.1.1]{C5:Gris} respectively), we have that  
\begin{equation}\label{nuevo_1}
\widetilde{v}^\varepsilon\in L^\infty(0,T;H^{2}(\Omega))\cap L^2(0,T;H^{3}(\Omega)).
\end{equation}
Thus, we deduce that $\nabla \widetilde{v}^\varepsilon\in L^\infty(0,T;H^{1}(\Omega))\cap L^2(0,T;H^{2}(\Omega))\hookrightarrow L^{10}(0,T; L^{10}(\Omega))$. Then, using this fact and taking into account that $(\widetilde{u}^\varepsilon,\widetilde{z}^\varepsilon)\in \mathcal{X}\times \mathcal{X} \hookrightarrow L^{10/3}(0,T;L^{10/3}(\Omega))^2$, we obtain that $\nabla \cdot (\widetilde{u}^\varepsilon_+\nabla \widetilde{v}^\varepsilon)=\nabla \widetilde{u}^\varepsilon_+\nabla \widetilde{v}^\varepsilon + \widetilde{u}^\varepsilon_+\Delta \widetilde{v}^\varepsilon \in L^{\frac{5}{3}}(0,T; L^{\frac{5}{3}}(\Omega))$ and $(\widetilde{u}^\varepsilon)^p+\widetilde{z}^\varepsilon \in L^{\frac{5}{3}}(0,T; L^{\frac{5}{3}}(\Omega))$ for any $p\in (1,2)$ (using that $\widetilde{u}^\varepsilon_+,\Delta \widetilde{v}^\varepsilon \in L^{\frac{10}{3}}(0,T; L^{\frac{10}{3}}(\Omega))$). Thus, applying Theorem \ref{C5:nn} to (\ref{C5:modelfexist0a}), we deduce that there exists a unique solution $(u^\varepsilon,z^\varepsilon)$ of (\ref{C5:modelfexist0a}), $(u^\varepsilon,z^\varepsilon)\in \widetilde{\mathcal{X}}\times \widetilde{\mathcal{X}}$ (where $\widetilde{\mathcal{X}}$ is defined in (\ref{C5:wsaREG})).
\item{All possible fixed points of $\lambda R$ (with $\lambda \in (0,1]$) are bounded in $\mathcal{X}\times \mathcal{X}$ and $u^\varepsilon\geq 0$.}
In fact, observe that if $({u}^\varepsilon,{z}^\varepsilon)$ is a fixed point of $\lambda R$, then $({u}^\varepsilon,{z}^\varepsilon)$ satisfies
\begin{equation}
\left\{
\begin{array}
[c]{lll}%
\partial_t u^\varepsilon - \Delta u^\varepsilon = \lambda \nabla\cdot ({u}^\varepsilon_+\nabla{v}^\varepsilon)\ \ \mbox{in}\ \Omega,\ t>0,\\
\partial_t z^\varepsilon - \Delta z^\varepsilon   = \lambda ({u}^\varepsilon)^{p} - \lambda z^\varepsilon\ \mbox{in}\ \ \Omega,\ t>0,\\
\displaystyle\frac{\partial u^\varepsilon}{\partial \mathbf{n}}=\frac{\partial z^\varepsilon}{\partial \mathbf{n}}=0\ \ \mbox{on}\ \partial\Omega,\ t>0,\\
u^\varepsilon(\x,0)=u^\varepsilon_0(\x)\geq 0,\ z^\varepsilon(\x,0)=v^\varepsilon_0(
\x)- \varepsilon\Delta v^\varepsilon_0(\x)\ \ \mbox{in}\ \Omega,
\end{array}
\right.  \label{C5:modelfexist0b}
\end{equation}
Multiplying (\ref{C5:modelfexist0b})$_1$ by $u^\varepsilon_-:= \min\{u^\varepsilon,0 \}$ and integrating in $\Omega$,  we have 
\begin{equation*}
\displaystyle\frac{1}{2}\frac{d}{dt}\Vert u^\varepsilon_-\Vert_0^2+ \Vert \nabla u^\varepsilon_- \Vert_{0}^2 = \lambda ({u}^\varepsilon_+\nabla{v}^\varepsilon, \nabla u^\varepsilon_-) = 0,
\end{equation*}
which, taking into account that $u^\varepsilon_0(\x)\geq 0$ a.e. in $\Omega$, implies that $u^\varepsilon\geq 0$ a.e.\hspace{-0.05 cm} in $(0,+\infty)\times \Omega$. Thus, $u^\varepsilon_+= u^\varepsilon$. Now, 
we test (\ref{C5:modelfexist0b})$_1$ and (\ref{C5:modelfexist0b})$_2$ by $\displaystyle\frac{p}{p-1}(u^\varepsilon)^{p-1}$ and $-\Delta v^\varepsilon$ respectively, and adding both equations, the terms $-\lambda \displaystyle\frac{p}{p-1} ({u}^\varepsilon\nabla{v}^\varepsilon,\nabla (u^\varepsilon)^{p-1})$ and $ \lambda (\nabla ({u}^\varepsilon)^{p},\nabla v^\varepsilon)$ cancel, and taking into account (\ref{C5:CVar1}), we obtain 
\begin{eqnarray}\label{C5:pf0a}
&\displaystyle\frac{d}{dt}&\!\!\!\mathcal{E}_\varepsilon(u^\varepsilon,v^\varepsilon)+ \frac{4}{p}\int_\Omega \vert
\nabla ((u^\varepsilon)^{p/2})\vert^{2} \nonumber\\
&&+\varepsilon \Vert \nabla(\Delta v^\varepsilon) \Vert_0^{2} + \Vert \Delta v^\varepsilon
\Vert_0^{2}= -\lambda \Vert \nabla v^\varepsilon \Vert_0^2 -\lambda\varepsilon \Vert \Delta v^\varepsilon \Vert_0^2\leq  0,
\end{eqnarray}
where 
$$\mathcal{E}_\varepsilon(u^\varepsilon,v^\varepsilon):=\displaystyle\frac{1}{p-1} \Vert u^\varepsilon\Vert_{L^p}^{p}   +
\frac{1}{2} \Vert \nabla v^\varepsilon\Vert_0^{2} +
\frac{\varepsilon}{2} \Vert \Delta v^\varepsilon\Vert_0^{2}.$$
Moreover, we observe that the function $y^\varepsilon(t)=\Big(\displaystyle\int_{\Omega} v^\varepsilon(\x ,t) \, d\x\Big)^2$ satisfies  $(y^\varepsilon)'(t)+y^\varepsilon(t)\leq w^\varepsilon(t)$, with $w^\varepsilon(t) = \Vert u^\varepsilon(t) \Vert_{L^p}^{2p}$. In fact, it follows by multiplying (\ref{C5:nuevo-1}) (for $(u,v)=(u^\varepsilon,v^\varepsilon)$) by $\displaystyle\int_{\Omega} v^\varepsilon(\x ,t) \, d\x$ and using the Young inequality.
Therefore, $y^\varepsilon(t)=y^\varepsilon(0) \, e^{-t} + \displaystyle\int_0^t e^{-(t-s)} \, w^\varepsilon(s) \, ds$, which implies that
\begin{equation}\label{C5:e6b}
\Big(\displaystyle\int_{\Omega} v^\varepsilon(\x ,t) \, d\x\Big)^2 \le   \Big(\displaystyle\int_{\Omega} v^\varepsilon_0(\x) \, d\x\Big)^2 + \Vert u^\varepsilon\Vert_{ L^\infty(0,+\infty; L^p)}^{2p}, \ \ \forall t\geq 0.
\end{equation}
Then, from (\ref{C5:pf0a})-(\ref{C5:e6b}) and using  (\ref{C5:H1divGrad}), we deduce  the following estimates with respect to $\lambda$:
\begin{equation}\label{C5:e2}
\left\{\begin{array}{l}
(u^\varepsilon, v^\varepsilon)\  \mbox{ is bounded in }\ L^{\infty}(0,+\infty;L^p(\Omega)\times \H^2(\Omega)) , \\
(u^\varepsilon)^{\frac{p}{2}}\  \mbox{ is bounded in }\ L^{\infty}(0,+\infty;L^2(\Omega))\cap L^{2}(0,T;H^1(\Omega))\hookrightarrow L^{\frac{10}{3}}(0,T; L^{\frac{10}{3}}(\Omega)), \\
u^\varepsilon \ \mbox{ is bounded in }\ L^{p}(0,T;L^{3p}(\Omega)) \ \ \mbox{ and } \  v^\varepsilon \  \mbox{ is bounded in }\ L^{2}(0,T; \H^3(\Omega)).
\end{array}\right.
\end{equation}
Then, from (\ref{C5:e2}) we conclude that $z^\varepsilon$ is bounded in $\mathcal{X}$. Moreover, testing (\ref{C5:modelfexist0b})$_1$ by $u^\varepsilon$, we have
$$
 \displaystyle\frac{1}{2}\frac{d}{dt} \Vert u^\varepsilon\Vert_{0}^{2} + \Vert u^\varepsilon\Vert_1^2 = -\lambda ({u}^\varepsilon\nabla{v}^\varepsilon,\nabla u^\varepsilon) + \Vert u^\varepsilon\Vert_0^2 \leq \frac{1}{2} \Vert u^\varepsilon\Vert_1^2 +  C\Big(\Vert \nabla{v}^\varepsilon\Vert_1^4 + 1\Big) \Vert u^\varepsilon\Vert_0^2,
$$ 
from which, taking into account (\ref{C5:e2}) and using the Gronwall Lemma, we deduce that $u^\varepsilon$ is bounded in $\mathcal{X}$.

\item{$R$ is compact.}\label{C5:3} Let $\{(\widetilde{u}^\varepsilon_n,\widetilde{z}^\varepsilon_n)\}_{n\in\mathbb{N}}$ be a bounded sequence in $\mathcal{X}\times \mathcal{X}$. Then $(u^\varepsilon_n,z^\varepsilon_n)=R(\widetilde{u}^\varepsilon_n,\widetilde{z}^\varepsilon_n)$ solves (\ref{C5:modelfexist0a}) (with $(\widetilde{u}_n^\varepsilon,\widetilde{z}^\varepsilon_n)$ and $(u^\varepsilon_n, z^\varepsilon_n)$ instead of $(\widetilde{u}^\varepsilon,\widetilde{z}^\varepsilon)$ and $(u^\varepsilon,z^\varepsilon)$ respectively). Therefore, analogously as in item 1, we obtain that  $\nabla \cdot (\widetilde{u}^\varepsilon_{n+}\nabla \widetilde{v}^\varepsilon_n)$ and $(\widetilde{u}^\varepsilon_{n})^p + \widetilde{z}^\varepsilon_n$ are bounded in $L^{\frac{5}{3}}(0,T; L^{\frac{5}{3}}(\Omega))$; and therefore, from Theorem \ref{C5:nn} we conclude that $\{R(\widetilde{u}^\varepsilon_n,\widetilde{z}^\varepsilon_n)\}_{n\in\mathbb{N}}$ is bounded in $\widetilde{\mathcal{X}}\times \widetilde{\mathcal{X}}$ which is compactly embedded in $\mathcal{X}\times \mathcal{X}$, and thus $R$ is compact. Observe that the compactness embedding comes from the continuous embedding (using embeddings $W^{k,p}(\Omega) \hookrightarrow H^s(\Omega)$, see \cite[Theorem 9.6]{C5:LM}):
$$
\widetilde{\mathcal{X}}  \hookrightarrow
L^{\infty}(0,T;H^{1/2}(\Omega)) 
\cap
L^{5/3}(0,T;H^{17/10}(\Omega)) \hookrightarrow
L^2(0,T;H^{3/2}(\Omega)).
$$
Then
$u^\varepsilon, \, z^\varepsilon \in  L^{\infty}(0,T;H^{1/2}(\Omega)) 
\cap
L^{2}(0,T;H^{3/2}(\Omega))$ and 
$
\partial_t u^\varepsilon, \, \partial_tz^\varepsilon \in L^{5/3}(0,T;L^{5/3}(\Omega))$, hence the compactness holds by applying the Aubin-Lions Lemma (see \cite{C5:simon}).

\item{$R$ is continuous from $\mathcal{X}\times \mathcal{X}$ into $\mathcal{X}\times \mathcal{X}$.} Let $\{(\widetilde{u}^\varepsilon_n,\widetilde{z}^\varepsilon_n)\}_{n\in\mathbb{N}}\subset \mathcal{X}\times \mathcal{X}$ be a sequence such that 
\begin{equation}\label{C5:c001cont}
(\widetilde{u}^\varepsilon_n,\widetilde{z}^\varepsilon_n)\rightarrow (\widetilde{u}^\varepsilon,\widetilde{z}^\varepsilon) \ \mbox{ in }  \mathcal{X}\times \mathcal{X},
\quad \hbox{as $n\to +\infty$}.
\end{equation}
Therefore, $\{(\widetilde{u}^\varepsilon_n,\widetilde{z}^\varepsilon_n)\}_{n\in\mathbb{N}}$ is bounded in $\mathcal{X}\times \mathcal{X}$, and from item \ref{C5:3} we deduce that $\{({u}^\varepsilon_n,{z}^\varepsilon_n)=R(\widetilde{u}^\varepsilon_n,\widetilde{z}^\varepsilon_n)\}_{n\in\mathbb{N}}$ is bounded in $\widetilde{\mathcal{X}}\times \widetilde{\mathcal{X}}$. Then, there exist $(\hat{u}^\varepsilon,\hat{z}^\varepsilon)$ and a subsequence of $\{R(\widetilde{u}^\varepsilon_n,\widetilde{z}^\varepsilon_n)\}_{n\in\mathbb{N}}$ still denoted by $\{R(\widetilde{u}^\varepsilon_n,\widetilde{z}^\varepsilon_n)\}_{n\in\mathbb{N}}$ such that 
\begin{equation}\label{C5:c002cont}
R(\widetilde{u}^\varepsilon_n,\widetilde{z}^\varepsilon_n)\rightarrow (\hat{u}^\varepsilon,\hat{z}^\varepsilon) \ \ \mbox{ weakly in } \widetilde{\mathcal{X}}\times \widetilde{\mathcal{X}} \ \mbox{ and strongly  in } {\mathcal{X}}\times{\mathcal{X}}.
\end{equation}
Then, from (\ref{C5:c001cont})-(\ref{C5:c002cont}), a standard procedure allows us to pass to the limit, as $n$ goes to $+\infty$, in (\ref{C5:modelfexist0a}) (with $(\widetilde{u}_n^\varepsilon,\widetilde{z}^\varepsilon_n)$ and $(u^\varepsilon_n, z^\varepsilon_n)$ instead of $(\widetilde{u}^\varepsilon,\widetilde{z}^\varepsilon)$ and $(u^\varepsilon,z^\varepsilon)$ respectively), and we deduce that $R(\widetilde{u}^\varepsilon,\widetilde{z}^\varepsilon)=(\hat{u}^\varepsilon,\hat{z}^\varepsilon)$. Therefore, we have proved that any convergent subsequence of  $\{R(\widetilde{u}^\varepsilon_n,\widetilde{z}^\varepsilon_n)\}_{n\in\mathbb{N}}$ converges to $R(\widetilde{u}^\varepsilon,\widetilde{z}^\varepsilon)$ strong in ${\mathcal{X}}\times{\mathcal{X}}$, and from uniqueness of $R(\widetilde{u}^\varepsilon,\widetilde{z}^\varepsilon)$, we conclude that the whole sequence $R(\widetilde{u}^\varepsilon_n,\widetilde{z}^\varepsilon_n)\rightarrow R(\widetilde{u}^\varepsilon,\widetilde{z}^\varepsilon)$ in ${\mathcal{X}}\times{\mathcal{X}}$. Thus, $R$ is continuous.
\end{enumerate}
Therefore, the hypotheses of the Leray-Schauder fixed point theorem are satisfied and we conclude that the map $R(\widetilde{u}^\varepsilon,\widetilde{z}^\varepsilon)$ has a fixed point $(u^\varepsilon,z^\varepsilon)$, that is, $R({u}^\varepsilon,{z}^\varepsilon)=(u^\varepsilon,z^\varepsilon)$, which is a solution of problem (\ref{C5:wsaREG})-(\ref{C5:modelf00reg}). 
\end{proof}

\subsection{Existence of weak-strong solutions of (\ref{C5:modelf00})}
\begin{tma}
There exists at least one $(u,v)$ weak-strong solution of problem (\ref{C5:modelf00}).
\end{tma}
\begin{proof}
Observe that a variational problem associated to (\ref{C5:modelf00reg}) is:
\begin{equation}
\left\{
\begin{array}
[c]{lll}%
\displaystyle\int_0^T \langle \partial_t u^\varepsilon,\bar{u}\rangle + \int_0^T (\nabla u^\varepsilon,  \nabla \bar{u}) +\int_0^T (u^\varepsilon\nabla v^\varepsilon,\nabla \bar{u})=0, \ \ \forall \bar{u}\in L^{\frac{10p}{7p-6}}(0,T;W^{1,\frac{10p}{7p-6}}(\Omega))\\
\displaystyle\int_0^T \langle \partial_t z^\varepsilon,\bar{z}\rangle + \int_0^T (\nabla z^\varepsilon,  \nabla \bar{z}) + \int^T_0 (z^\varepsilon, \bar{z}) = \displaystyle\int_0^T ((u^\varepsilon)^p,\bar{z}), \ \ \forall  \bar{z}\in L^{\frac{5}{2}}(0,T;H^{1}(\Omega)).
\end{array}
\right.  \label{C5:modelfexist0w}
\end{equation}
Recall that $v^\varepsilon=v(z^\varepsilon)$ is the unique solution of problem (\ref{C5:CVar1}). From (\ref{C5:pf0a}) we have that $(u^\varepsilon,v^\varepsilon)$ satisfies the following energy equality:
\begin{equation}\label{C5:pf0ann}
\displaystyle\frac{d}{dt}\mathcal{E}_\varepsilon(u^\varepsilon,v^\varepsilon)+ \frac{4}{p} \Vert
\nabla ((u^\varepsilon)^{p/2})\Vert_0^{2}+ \varepsilon \Vert \Delta v^\varepsilon \Vert_1^{2}+ \Vert \nabla v^\varepsilon
\Vert_1^{2}=0.
\end{equation}
Then, from (\ref{C5:pf0ann}) and using (\ref{C5:e6b}) we deduce the following estimates (independent of $\varepsilon$)
\begin{equation}\label{C5:e2nn}
\left\{\begin{array}{l}
\{(u^\varepsilon)^{\frac{p}{2}}\}\  \mbox{ is bounded in }\ L^{\infty}(0,+\infty;L^2(\Omega))\cap L^{2}(0,T;H^{1}(\Omega))\hookrightarrow L^{\frac{10}{3}}(0,T;L^{\frac{10}{3}}(\Omega)), \\
\{v^\varepsilon\}\  \mbox{ is bounded in }\ L^{\infty}(0,+\infty;H^1(\Omega))\cap L^{2}(0,T;H^2(\Omega)), \\
\{\sqrt{\varepsilon} \Delta v^\varepsilon\} \ \mbox{ is bounded in }\  L^{\infty}(0,+\infty;L^2(\Omega))\cap L^{2}(0,T;H^1(\Omega)),
\end{array}\right.
\end{equation}
and therefore,
\begin{equation}\label{C5:e2nna}
\left\{\begin{array}{l}
\{u^\varepsilon\}\  \mbox{ is bounded in }\ L^{\infty}(0,+\infty;L^p(\Omega))\cap L^{p}(0,T;L^{3p}(\Omega))\hookrightarrow L^{\frac{5p}{3}}(0,T;L^{\frac{5p}{3}}(\Omega)), \\
\{z^\varepsilon\} \ \mbox{ is bounded in }\  L^{\infty}(0,+\infty;L^2(\Omega))\cap L^{2}(0,T;H^1(\Omega)),\\
\{\partial_t u^\varepsilon\}\  \mbox{ is bounded in }\  [L^{\frac{10p}{7p-6}}(0,T;W^{1,\frac{10p}{7p-6}}(\Omega))]',\\
\{\partial_t z^\varepsilon\}\  \mbox{ is bounded in }\  [L^{\frac{5}{2}}(0,T;H^1(\Omega))]'.
\end{array}\right.
\end{equation}
Moreover, taking into account that from  (\ref{C5:e2nn})$_1$ we have that $\nabla ((u^\varepsilon)^{p/2})$ is bounded in $L^{2}(0,T;L^{2}(\Omega))$ and from  (\ref{C5:e2nna})$_1$  $u^{1-\frac{p}{2}}$ is bounded in $L^{\frac{10p}{6-3p}}(0,T;L^{\frac{10p}{6-3p}}(\Omega))$, we conclude that $\nabla u^\varepsilon= \displaystyle\frac{2}{p} u^{1-\frac{p}{2}} \nabla ((u^\varepsilon)^{p/2})$ is bounded in $L^{\frac{5p}{p+3}}(0,T;L^{\frac{5p}{p+3}}(\Omega))$. Therefore, we deduce that
\begin{equation}\label{C5:ees}
\{u^\varepsilon\}\  \mbox{ is bounded in }\  L^{\frac{5p}{p+3}}(0,T;W^{1,\frac{5p}{p+3}}(\Omega)).
\end{equation}
Notice that from (\ref{C5:CVar1}) and (\ref{C5:e2nn})$_3$, we can deduce that 
\begin{equation}\label{C5:ner2}
\Vert z^\varepsilon - v^\varepsilon\Vert_{L^\infty L^2\cap L^2H^1}\leq \varepsilon\Vert \Delta v^\varepsilon\Vert_{L^\infty L^2\cap L^2H^1} \rightarrow 0\  \mbox{ as } \varepsilon\rightarrow 0.
\end{equation}
Then, from (\ref{C5:e2nn})-(\ref{C5:ner2}), we deduce that there exists $(u,v)$,  with
\begin{equation*}
\left\{\begin{array}{l}
u\in L^{\infty}(0,+\infty;L^p(\Omega)) \cap L^{\frac{5p}{3}}(0,T;L^{\frac{5p}{3}}(\Omega))\cap L^{\frac{5p}{p+3}}(0,T;W^{1,\frac{5p}{p+3}}(\Omega)),\\
v\in L^{\infty}(0,+\infty;H^1(\Omega))\cap L^{2}(0,T;H^2(\Omega)),
\end{array}\right.
\end{equation*}
such that for some subsequence of $\{u^\varepsilon,z^\varepsilon, v^\varepsilon\}$ still denoted by $\{u^\varepsilon,z^\varepsilon, v^\varepsilon\}$, the following weak convergences hold when $\varepsilon\rightarrow 0$,
\begin{equation}\label{C5:e2nnb}
\left\{\begin{array}{l}
u^\varepsilon\rightarrow u \ \ \mbox{ weakly in }\  L^{\frac{5p}{3}}(0,T;L^{\frac{5p}{3}}(\Omega))\cap L^{\frac{5p}{p+3}}(0,T;W^{1,\frac{5p}{p+3}}(\Omega)), \\
v^\varepsilon\rightarrow v \ \  \mbox{ weakly in }\  L^{2}(0,T;H^2(\Omega)),\\
z^\varepsilon\rightarrow v \ \ \mbox{ weakly in }\  L^{2}(0,T;H^1(\Omega)),\\
\partial_t u^\varepsilon\rightarrow \partial_t u \ \  \mbox{ weakly-}\star \mbox{ in }\  [L^{\frac{10p}{7p-6}}(0,T;W^{1,\frac{10p}{7p-6}}(\Omega))]',\\
\partial_t z^\varepsilon\rightarrow \partial_t v  \ \ \mbox{ weakly-}\star \mbox{ in }\  [L^{\frac{5}{2}}(0,T;H^1(\Omega))]'.
\end{array}\right.
\end{equation}
On the other hand, taking into account (\ref{C5:e2nna})$_3$ and (\ref{C5:ees}), the Aubin-Lions Lemma implies that 
\begin{equation}\label{Nf1}
\{u^\varepsilon\} \ \mbox{ is relatively compact in } L^{\frac{5p}{p+3}}(0,T;L^{2}(\Omega))
\end{equation} 
(and also in $L^{r}(0,T;L^{r}(\Omega))$, for all  $r<\frac{5p}{3}$). In particular, since  $u^\varepsilon\geq 0$ then $u\geq 0$ a.e. in $(0,+\infty)\times \Omega$. 
Moreover, since the embedding $L^{\infty}(0,T;L^2(\Omega))\cap L^{2}(0,T;H^1(\Omega)) \hookrightarrow  L^{\frac{10}{3}}(0,T;L^{\frac{10}{3}}(\Omega))$ is continuous, from (\ref{C5:e2nn})$_2$ we deduce that 
\begin{equation}\label{C5:new33}
\nabla v^\varepsilon\rightarrow \nabla v \ \mbox{ weakly in } L^{\frac{10}{3}}(0,T;\L^{\frac{10}{3}}(\Omega)).
\end{equation}
Thus, from (\ref{Nf1})-(\ref{C5:new33}) and using that ${u^\varepsilon\nabla v^\varepsilon}$ is bounded in $L^{\frac{10p}{3p+6}}(0,T;\L^{\frac{10p}{3p+6}}(\Omega))$, we deduce that
\begin{equation}\label{C5:nnne}
u^\varepsilon\nabla v^\varepsilon\rightarrow u\nabla v \ \mbox{ weakly in } L^{\frac{10p}{3p+6}}(0,T;\L^{\frac{10p}{3p+6}}(\Omega)).
\end{equation}
Moreover, since $u^\varepsilon \rightarrow u$ strongly in $L^p(0,T;L^p(\Omega))$, we have that 
\begin{equation}\label{Nf2}
(u^\varepsilon)^p \rightarrow u^p \ \mbox{ strongly in }  L^1(0,T;L^1(\Omega)).
\end{equation}
 Thus, taking to the limit when $\varepsilon\rightarrow 0$ in (\ref{C5:modelfexist0w}), and using (\ref{C5:e2nnb}) and (\ref{C5:nnne})-(\ref{Nf2}), we obtain that $(u,v)$ satisfies 
\begin{equation}\label{C5:CCC1}
\int_0^T \langle \partial_t u,\bar{u}\rangle + \int_0^T (\nabla u,  \nabla \bar{u}) +\int_0^T (u\nabla v,\nabla \bar{u})=0, \ \ \forall \bar{u}\in L^{\frac{10p}{7p-6}}(0,T;W^{1,\frac{10p}{7p-6}}(\Omega)),
\end{equation}
\begin{equation}\label{C5:CCC2}
\displaystyle\int_0^T \langle \partial_t v,\bar{z}\rangle + \int_0^T (\nabla v,  \nabla \bar{z}) + \int^T_0 (v, \bar{z}) = \displaystyle\int_0^T (u^p,\bar{z}), \ \ \forall  \bar{z}\in L^{\frac{5}{2}}(0,T;H^{1}(\Omega)),
\end{equation}
and therefore, integrating by parts in (\ref{C5:CCC2}) and taking into account that $u^p\in L^{\frac{5}{3}}(0,T; L^{\frac{5}{3}}(\Omega))$ and $v\in L^2(0,T;H^2(\Omega))$, we arrive at
\begin{equation}\label{C5:nnn22}
\partial_t v - \Delta v  + v= u^p \  \mbox{ in } L^{\frac{5}{3}}(0,T; L^{\frac{5}{3}}(\Omega)),
\end{equation}
with $\displaystyle\frac{\partial v}{\partial \mathbf{n}}=0$ on $\partial\Omega$. Notice that the limit function $v$ is nonnegative. In fact, it follows by testing (\ref{C5:nnn22}) by $v_{-}$ and taking into account that $v_0\geq 0$.   Finally, we will prove that $(u,v)$ satisfies the energy inequality (\ref{C5:wsd}). Indeed, integrating (\ref{C5:pf0ann}) in time from $t_0$ to $t_1$, with $t_1>t_0\geq 0$, and taking into account that 
$$
\int_{t_0}^{t_1} \displaystyle\frac{d}{dt}\mathcal{E}_\varepsilon(u^\varepsilon,v^\varepsilon)= \mathcal{E}_\varepsilon(u^\varepsilon(t_1),v^\varepsilon(t_1)) - \mathcal{E}_\varepsilon(u^\varepsilon(t_0),v^\varepsilon(t_0)) \quad  \forall t_0<t_1,
$$
since $\mathcal{E}_\varepsilon(u^\varepsilon(t),v^\varepsilon(t))\in W^{1,1}(0,T)$ for all $T>0$, is continuous in time, we deduce
\begin{eqnarray}\label{C5:eq13}
&&\!\!\!\!\!\!\mathcal{E}_\varepsilon(u^\varepsilon(t_1),v^\varepsilon(t_1)) - \mathcal{E}_\varepsilon(u^\varepsilon(t_0),v^\varepsilon(t_0)) \nonumber\\
&&\!\!\! + \int_{t_0}^{t_1} \Big( \frac{4}{p} \Vert
\nabla ((u^\varepsilon(t))^{p/2})\Vert_0^{2}+ \varepsilon \Vert \Delta v^\varepsilon(t) \Vert_1^{2}  +
\Vert \nabla v^\varepsilon(t)
\Vert_1^{2}\Big) dt = 0, \ \ \ \forall t_0<t_1.
\end{eqnarray}
Now, we will prove that 
\begin{equation}\label{C5:FIN}
\mathcal{E}_\varepsilon(u^\varepsilon(t),v^\varepsilon(t))\rightarrow \mathcal{E}({u}(t),v(t)), \ \ \mbox{a.e.} \  t\in[0,+\infty).
\end{equation}
Since $u^\varepsilon$ is relatively compact in $L^{p}(0,T;L^{p}(\Omega))$, we have
\begin{equation}\label{C5:cccc}
u^\varepsilon \rightarrow u \ \mbox{ strongly in } L^{p}(0,T;L^{p}(\Omega)).
\end{equation} 
Moreover, for any $T>0$,
\begin{eqnarray}\label{C5:cc00}
 & \Vert \mathcal{E}_\varepsilon&\!\!\!\!\!(u^\varepsilon(t),v^\varepsilon(t))  - \mathcal{E}({u}(t),v(t)) \Vert_{L^1(0,T)} =\displaystyle\int_0^T  \vert \mathcal{E}_\varepsilon(u^\varepsilon(t),v^\varepsilon(t)) - \mathcal{E}({u}(t),v(t)) \vert dt\nonumber\\
 && \!\!\!\!  \leq  \int_0^T 
 \left\vert  \frac{1}{p-1}\left(\Vert u^\varepsilon(t) \Vert_{L^p}^p- \Vert {u}(t) \Vert_{L^p}^p \right) + 
 \frac{1}{2} \left( \Vert\nabla v^\varepsilon(t)\Vert_0^2 -  \Vert\nabla v(t)\Vert_0^2\right)  + \frac{\varepsilon}{2} \Vert \Delta v^\varepsilon\Vert_0^{2}
  \right\vert dt
  \nonumber\\
 && \!\!\!\!  \leq  C\frac{p}{p-1} \Vert u^\varepsilon  -  {u} \Vert_{L^p(0,T;L^p)} (\Vert u^\varepsilon  \Vert_{L^p(0,T;L^p)} + \Vert {u}\Vert_{L^p(0,T;L^p)})^{p-1}
 \nonumber\\
 && \!\!\!\! + \frac{1}{2}\Vert\nabla v^\varepsilon- \nabla v\Vert_{L^2(0,T;L^2)}  (\Vert\nabla v^\varepsilon\Vert_{L^2(0,T;L^2)} + \Vert\nabla v\Vert_{L^2(0,T;L^2)}) + 
 \frac{\varepsilon}{2} \Vert \Delta v^\varepsilon\Vert_{L^2(0,T;L^2)}^{2}.
\end{eqnarray}
Then, taking into account that $u^\varepsilon \rightarrow u$ strongly in $L^p(0,T;L^p(\Omega))$, $\nabla v^\varepsilon \rightarrow \nabla v$ strongly in $L^2(0,T;L^2(\Omega))$ for any $T>0$, and $\Delta v^\varepsilon$ is bounded in $L^2(0,T;L^2(\Omega))$, from (\ref{C5:cc00}) we conclude that $\mathcal{E}_\varepsilon(u^\varepsilon(t),v^\varepsilon(t)) \rightarrow \mathcal{E}({u}(t),v(t))$ strongly in $L^1(0,T)$ for all $T>0$, which implies in particular (\ref{C5:FIN}). Finally, observe that from (\ref{C5:cccc}) we have that $(u^\varepsilon)^{p/2} \rightarrow u^{p/2}$ strongly in $L^{2}(0,T;L^{2}(\Omega))$; and since $\nabla((u^\varepsilon)^{p/2}$ is bounded in $L^{2}(0,T;L^{2}(\Omega))$ we deduce that 
$$
\nabla((u^\varepsilon)^{p/2})\rightarrow \nabla(u^{p/2}) \ \mbox{ weakly in } L^{2}(0,T;L^{2}(\Omega)).
$$
Then, on the one hand
\begin{eqnarray*}
&&\underset{\varepsilon\rightarrow 0}{\lim \mbox{inf} } \int_{t_0}^{t_1}\Big( \frac{4}{p} \Vert
\nabla ((u^\varepsilon(t))^{p/2})\Vert_0^{2}+ \varepsilon \Vert \Delta v^\varepsilon(t) \Vert_1^{2} +
 \Vert \nabla v^\varepsilon(t)
\Vert_1^{2}\Big)  dt\nonumber\\
&&\hspace{3 cm} \ge \int_{t_0}^{t_1} \Big(   \frac{4}{p}\Vert \nabla (u(t)^{p/2})\Vert_{0}^{2} +
\displaystyle\Vert \nabla v(t)\Vert_{1}^{2}\Big) dt \quad \ \forall t_1\geq t_0\geq 0,
\end{eqnarray*}
and on the other hand, owing to (\ref{C5:FIN}),
\begin{eqnarray*}
\underset{\varepsilon\rightarrow 0}{\lim \mbox{inf} }\ \Big[\mathcal{E}_\varepsilon(u^\varepsilon(t_1),v^\varepsilon(t_1)) - \mathcal{E}_\varepsilon(u^\varepsilon(t_0),v^\varepsilon(t_0))\Big]  = \mathcal{E}({u}(t_1),v(t_1)) - \mathcal{E}({u}(t_0),v(t_0)),
\end{eqnarray*}
for a.e. $t_1,t_0: t_1\geq t_0\geq 0$. Thus, taking  $\liminf$ as $\varepsilon\rightarrow 0$ in inequality (\ref{C5:eq13}), we deduce the energy inequality (\ref{C5:wsd}) for a.e. $t_0,t_1:t_1\geq t_0\geq 0$.

\end{proof}

\section{Fully discrete numerical schemes}\label{C5:S4C5:NS}
In this section we will propose three fully discrete numerical schemes associated to model (\ref{C5:modelf00}). We prove some unconditional properties such as mass-conservation, energy-stability and solvability of the schemes.

\subsection{Scheme UV$\varepsilon$}\label{C5:Suv}
In this section, in order to construct an energy-stable fully discrete scheme for model (\ref{C5:modelf00}), we are going to make a regularization procedure, in which we will adapt the ideas of \cite{C5:BB} (see also \cite{C5:GR}). With this aim, given $\varepsilon\in (0,1)$ we consider a function $F_\varepsilon:\mathbb{R}\rightarrow [0,+\infty)$, approximation of $f(s)=s^p$, such that  $F_\varepsilon\in C^2(\mathbb{R})$ and
\begin{equation} \label{C5:F2pE}
F''_\varepsilon(s) \ \ := \ \ 
\left\{\begin{array}{lcl}
\varepsilon^{p-2} & \mbox{ if } & s\leq \varepsilon,\\
s^{p-2} & \mbox{ if } & \varepsilon\leq s\leq \varepsilon^{-1},\\
\varepsilon^{2-p} & \mbox{ if } & s\geq \varepsilon^{-1}.
\end{array}\right.  
\end{equation}
Then, $F_\varepsilon$ is obtained by
integrating in (\ref{C5:F2pE}) and imposing the conditions $F'_\varepsilon(1)=\frac{1}{p-1}$ and $F_\varepsilon(1)=\frac{1}{p(p-1)} + \frac{p^3 - 4p^2+3p+2}{2p(p-1)^2}\varepsilon^p$ (see Figure \ref{C5:fig:Fe}); and 
\begin{equation} \label{C5:lE}
a_\varepsilon(s) \ := \ (p-1)\frac{F'_\varepsilon(s)}{F''_\varepsilon(s)} \ = \ 
\left\{\begin{array}{lcl}
(p-1) s + (2-p)\varepsilon & \mbox{ if } & s\leq \varepsilon,\\
s & \mbox{ if } & \varepsilon\leq s\leq \varepsilon^{-1},\\
(p-1)s + (2-p)\varepsilon^{-1} & \mbox{ if } & s\geq \varepsilon^{-1}.
\end{array}\right. 
\end{equation}
\begin{figure}[htbp]
\centering 
\subfigure[$F_\varepsilon(s)$ vs  $F(s):= \frac{1}{p(p-1)} s^p + \frac{p^3 - 4p^2+3p+2}{2p(p-1)^2}\varepsilon^p$]{\includegraphics[width=72mm]{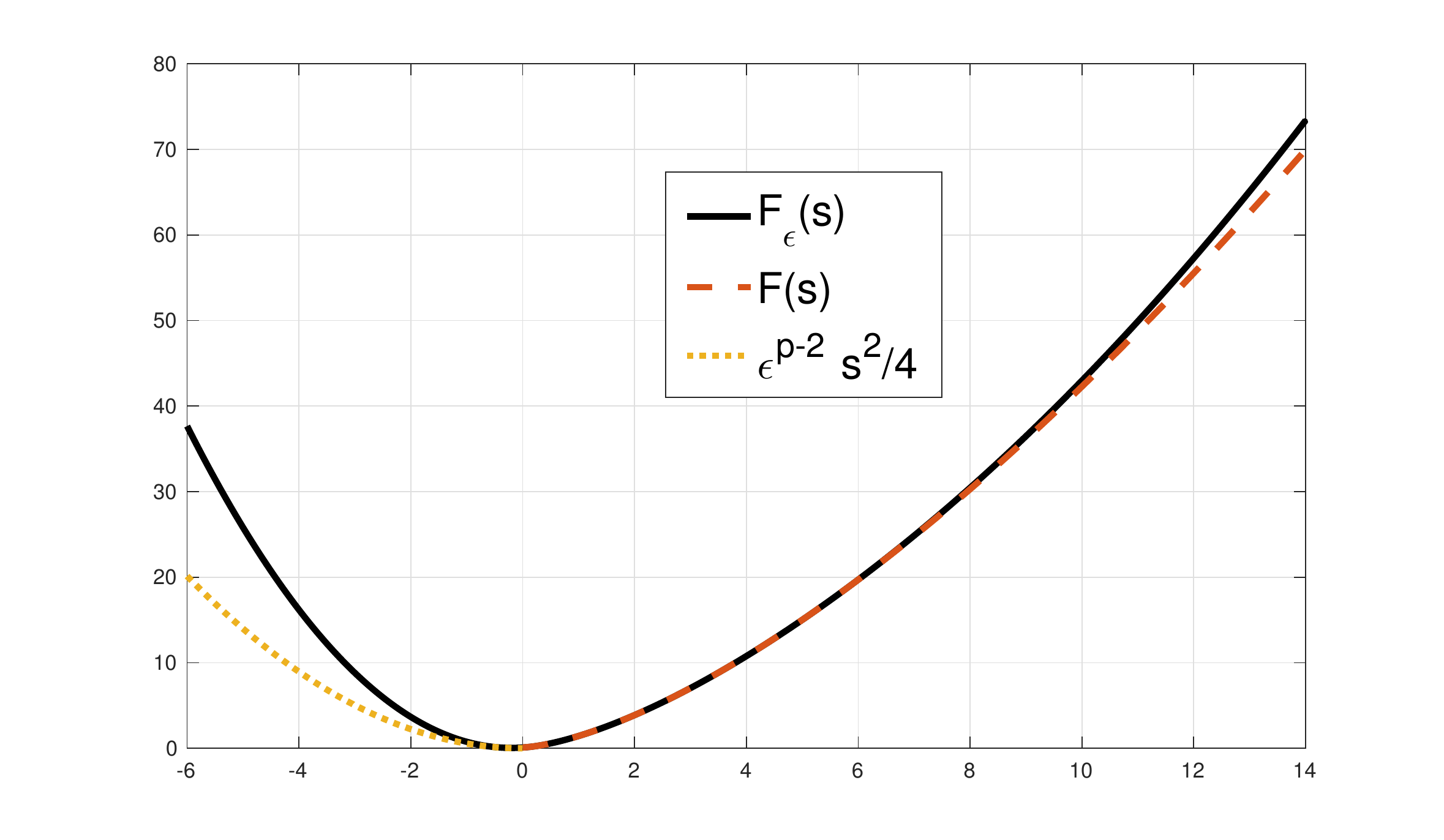}} \hspace{1,2 cm} 
\subfigure[$F'_\varepsilon(s)$ vs  $F'(s):= \frac{1}{p-1} s^{p-1}$]{\includegraphics[width=72mm]{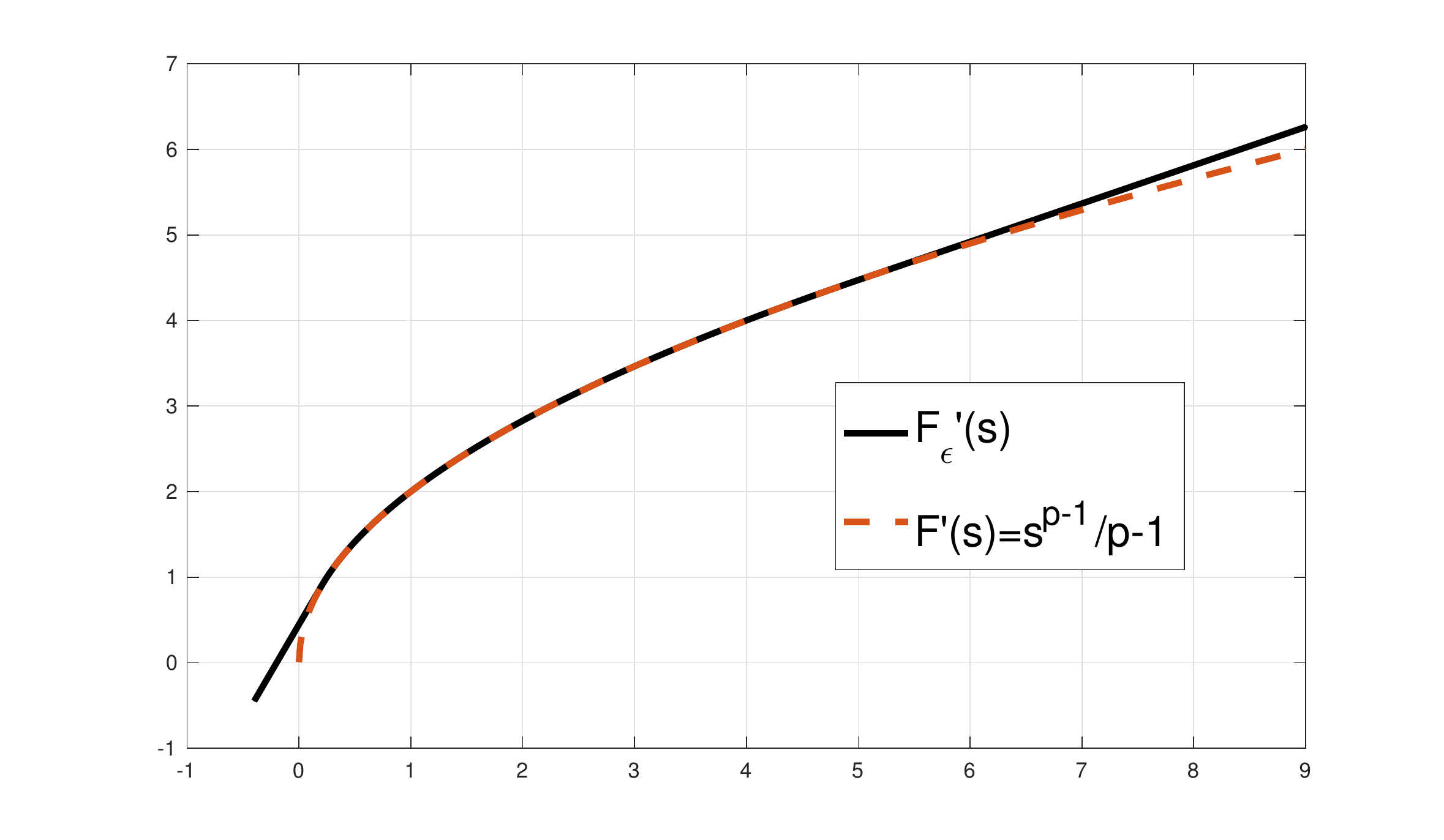}} 
\subfigure[$F''_\varepsilon(s)$ vs  $F''(s):= s^{p-2}$]{\includegraphics[width=72mm]{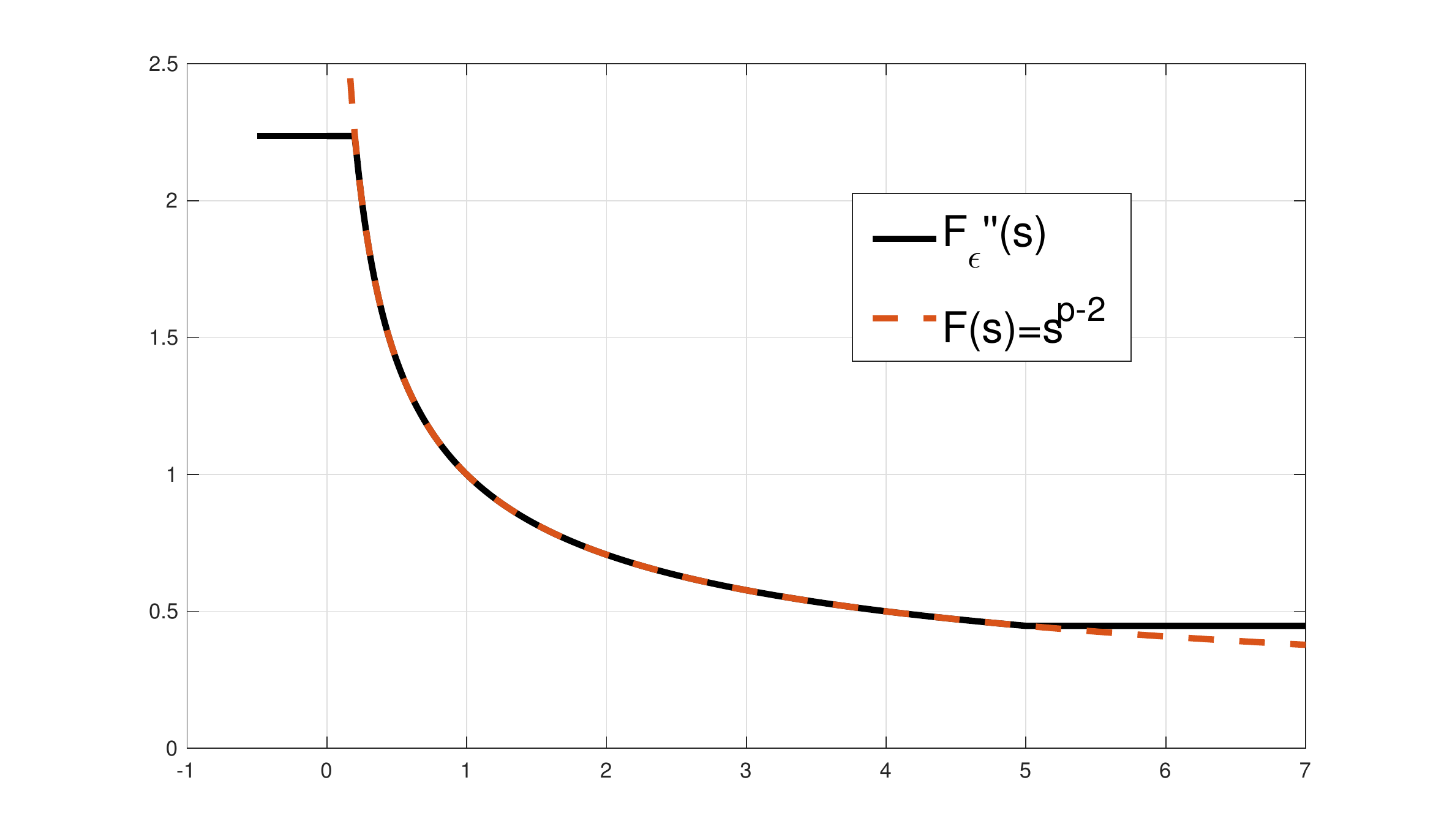}}
\caption{The function $F_\varepsilon$ and its derivatives.} \label{C5:fig:Fe}
\end{figure}
\begin{lem}\label{C5:eFe}
The function $F_\varepsilon$ satisfies 
\begin{equation}\label{C5:es11}
F_\varepsilon(s)\geq \frac{\varepsilon^{p-2}s^2}{4} \qquad \forall s\leq \varepsilon \ \ \mbox{ and } \ \ F_\varepsilon(s)\geq C s^p \qquad \forall s>\varepsilon,
\end{equation}
where the constant $C>0$ is independent of $\varepsilon$.
\end{lem}
\begin{proof}
Since $F_\varepsilon \in  C^{2}(\mathbb{R})$, using the Taylor formula as well as the definition of $F_\varepsilon$ and $F'_\varepsilon$, we have that, for some $s_0\in \mathbb{R}$ between $0$ and $s$,
\begin{equation}\label{C5:Feeq}
F_\varepsilon(s)=F_\varepsilon(0) + F'_\varepsilon(0) s + \frac{1}{2} F''_\varepsilon(s_0) s^2= \Big(\frac{2-p}{p-1} \Big)^2 \varepsilon^{p} + \frac{2-p}{p-1} \varepsilon^{p-1} s + \frac{1}{2} F''_\varepsilon(s_0) s^2.
\end{equation}
Then, taking into account that $F''_\varepsilon(s)=\varepsilon^{p-2}$ for all $s\leq \varepsilon$, from (\ref{C5:Feeq}) we have that: (a) if $s\in [0,\varepsilon]$, $F_\varepsilon(s)\geq \frac{1}{2} \varepsilon^{p-2} s^2$; and (b) if $s<0$, by using the Young inequality, 
\begin{eqnarray*}\label{C5:Feeq1}
F_\varepsilon(s)\geq \Big(\frac{2-p}{p-1} \Big)^2 \varepsilon^{p} - \frac{1}{4} \varepsilon^{p-2} s^2 - \Big(\frac{2-p}{p-1} \Big)^2 \varepsilon^{p} + \frac{1}{2} \varepsilon^{p-2} s^2 = \frac{1}{4} \varepsilon^{p-2} s^2,
\end{eqnarray*}
from which we deduce (\ref{C5:es11})$_1$. Finally, (\ref{C5:es11})$_2$ follows directly from the definition of $F_\varepsilon$ for $s\geq \varepsilon$.
\end{proof}
\begin{obs}\label{C5:OBSr}
Notice that estimates in (\ref{C5:es11}) imply that $\vert s\vert^p\leq K_1 F_\varepsilon(s) + K_2$ for all $s\in\mathbb{R}$, where the constants $K_1,K_2>0$ are independent of $\varepsilon$.
\end{obs}
Then, taking into account the functions $F_\varepsilon$, its derivatives and $a_\varepsilon$, a regularized version of problem (\ref{C5:modelf00}) reads: Find $u_\varepsilon:\Omega\times [0,T]\rightarrow \mathbb{R}$ and $v_\varepsilon:\Omega\times [0,T]\rightarrow \mathbb{R}$, with $u_\varepsilon,v_\varepsilon \geq 0$, such that
\begin{equation}
\left\{
\begin{array}
[c]{lll}%
\partial_t u_\varepsilon -\Delta  u_\varepsilon  -\nabla\cdot(a_\varepsilon(u_\varepsilon) \nabla v_\varepsilon)=0\ \ \mbox{in}\ \Omega,\ t>0,\\
\partial_t {v}_\varepsilon -\Delta v_\varepsilon + {v}_\varepsilon = p (p-1) F_\varepsilon(u_\varepsilon)\ \ \mbox{in}\ \Omega,\ t>0,\\
\displaystyle\frac{\partial u_\varepsilon}{\partial \mathbf{n}}=\frac{\partial v_\varepsilon}{\partial \mathbf{n}}=0\ \ \mbox{on}\ \partial\Omega,\ t>0,\\
u_\varepsilon(\x,0)=u_0(\x)\geq 0,\ v_\varepsilon(\x,0)=v_0(\x)\geq 0\ \ \mbox{in}\ \Omega.
\end{array}
\right.  \label{C5:modelf02acontUVreg}
\end{equation}
\begin{obs}
The idea is to define a fully discrete scheme associated to (\ref{C5:modelf02acontUVreg}), taking in general $\varepsilon=\varepsilon(k,h)$, such that $\varepsilon(k,h)\rightarrow 0$ as $(k,h)\rightarrow 0$, where $k$ is the time step and $h$ the mesh size.
\end{obs}
Observe that (formally) multiplying (\ref{C5:modelf02acontUVreg})$_1$ by $pF'_\varepsilon(u_\varepsilon)$, (\ref{C5:modelf02acontUVreg})$_2$ by $-\Delta v_\varepsilon$, integrating over $\Omega$ and adding, the chemotaxis and production terms cancel and we obtain the following energy law
$$
\frac{d}{dt} \displaystyle \int_\Omega \Big( p F_\varepsilon(u_\varepsilon) + \frac{1}{2} \vert \nabla v_\varepsilon\vert^2\Big) d \x + \int_\Omega p F''_\varepsilon(u_\varepsilon)\vert \nabla u_\varepsilon\vert^2 d\x+\Vert \nabla v_\varepsilon\Vert_1^2 = 0.
$$
In particular, the modified energy 
$$\mathcal{E}_\varepsilon(u,v)= \displaystyle\int_\Omega \Big( p F_\varepsilon(u) + \frac{1}{2} \vert \nabla v\vert^2\Big) d \x$$
is decreasing in time. Thus, we consider a fully discrete approximation of the regularized problem (\ref{C5:modelf02acontUVreg}) using a FE discretization in space and the backward
Euler discretization in time (considered for simplicity on a uniform partition of $[0,T]$ with time step $k=T/N : (t_n = nk)_{n=0}^{n=N}$). Let $\Omega$ be a polygonal domain. We consider a shape-regular and quasi-uniform family of triangulations of $\Omega$, denoted by $\{\mathcal{T}_h\}_{h>0}$, with simplices $K$, $h_K= diam(K)$ and $h:= \max_{K\in \mathcal{T}_h} h_K$, so that $\overline{\Omega}=\cup_{K\in \mathcal{T}_h} \overline{K}$. Further, let $\mathcal{N}_h = \{ \mathbf{a}_{i}\}_{i\in \mathcal{I}}$ denote the set of all the vertices of $\mathcal{T}_h$,  and in this case we will assume the following hypothesis:
\begin{enumerate}
\item[({\bf H})]{The triangulation is structured in the sense that all simplices have a right angle.}
\end{enumerate}
We choose the following continuous FE spaces for $u_\varepsilon$ and $v_\varepsilon$:
$$(U_h, V_h) \subset H^1(\Omega)^2,\quad \hbox{generated by $\mathbb{P}_1,\mathbb{P}_r$ with $r\geq 1$.}
$$
\begin{obs}
The right-angled constraint ({\bf H}) and the approximation of $U_h$ by $\mathbb{P}_1$-continuous FE are necessary to obtain the relations (\ref{C5:PL15uv})-(\ref{C5:PL16uv}) below, which are essential in order to obtain the energy-stability of the scheme \textbf{UV$\varepsilon$} (see Theorem \ref{C5:estinc1uvreg} below).
\end{obs}
We denote the Lagrange interpolation operator by $\Pi^h: C(\overline{\Omega})\rightarrow U_h$, and we introduce the discrete semi-inner product on $C(\overline{\Omega})$ (which is an inner product in $U_h$) and its induced discrete seminorm (norm in $U_h$):
\begin{equation}\label{C5:mlumpreguv}
(u_1,u_2)^h:=\int_\Omega \Pi^h (u_1 u_2), \   \vert u \vert_h=\sqrt{(u,u)^h}.
\end{equation}
\begin{obs}\label{C5:eqh25uv}
In $U_h$, the norms $\vert \cdot\vert_h$ and $\Vert \cdot\Vert_0$ are equivalents uniformly with respect to $h$ (see \cite{C5:PB}).
\end{obs}
We consider also the  $L^2$-projection $Q^h:L^2(\Omega)\rightarrow U_h$ given by
\begin{equation}\label{C5:MLP25uv}
(Q^h u,\bar{u})^h=(u,\bar{u}), \ \ \forall \bar{u}\in U_h,
\end{equation}
and the standard $H^1$-projection ${R}^h:H^1(\Omega)\rightarrow V_h$. Moreover, owing to the right angled constraint ({\bf H}) and the choice of $\mathbb{P}_1$-continuous FE for $U_h$, following the ideas of \cite{C5:BB} (see also \cite{C5:GR}), for each $\varepsilon\in (0,1)$, we can construct two operators $\Lambda^i_\varepsilon: U_h\rightarrow L^\infty(\Omega)^{d\times d}$ ($i=1,2$) such that  $\Lambda_\varepsilon^i u^h$  are symmetric matrices and $\Lambda_\varepsilon^1 u^h$ is positive definite, for all $u^h\in U_h$ and a.e. in $\Omega$, and satisfy
\begin{equation}\label{C5:PL15uv}
(\Lambda^1_\varepsilon u^h) \nabla \Pi^h (F'_\varepsilon(u^h))=\nabla u^h \ \ \mbox{ in } \Omega,
\end{equation}
\begin{equation}\label{C5:PL16uv}
(\Lambda^2_\varepsilon u^h) \nabla \Pi^h (F'_\varepsilon(u^h))=(p-1)\nabla \Pi^h (F_\varepsilon(u^h)) \ \ \mbox{ in } \Omega.
\end{equation}
Basically, $\Lambda^i_\varepsilon u^h$ ($i=1,2$) are constant by elements matrices such that (\ref{C5:PL15uv}) and (\ref{C5:PL16uv}) holds by elements. In the $1$-dimensional case, $\Lambda_\varepsilon^i$ are constructed as follows: For all $u^h\in U_h$ and $K\in \mathcal{T}_h$ with vertices $\mathbf{a}_0^K$ and $\mathbf{a}_1^K$, we set
\begin{equation}\label{C5:ConLa1}
\Lambda_\varepsilon^1(u^h)\vert_K \ := 
\left\{\begin{array}{lcl}
\frac{u^h(\mathbf{a}_1^K) - u^h(\mathbf{a}_0^K)}{F'_\varepsilon(u^h(\mathbf{a}_1^K))- F'_\varepsilon(u^h(\mathbf{a}_0^K))}=\frac{1}{F''_\varepsilon(u^h(\xi))}\ & \mbox{ if } & u^h(\mathbf{a}_0^K) \neq u^h(\mathbf{a}_1^K),\\
\frac{1}{F''_\varepsilon(u^h(\mathbf{a}_0^K))} & \mbox{ if } & u^h(\mathbf{a}_0^K) = u^h(\mathbf{a}_1^K),
\end{array}\right. 
\end{equation}
for some $\xi\in K$, and 
\begin{equation} \label{C5:ConLa2}
\Lambda_\varepsilon^2(u^h)\vert_K \!:=\! 
\left\{\begin{array}{lcl}
(p-1)\frac{F_\varepsilon(u^h(\mathbf{a}_1^K)) - F_\varepsilon(u^h(\mathbf{a}_0^K))}{F'_\varepsilon(u^h(\mathbf{a}_1^K))- F'_\varepsilon(u^h(\mathbf{a}_0^K))}=(p-1)\frac{F'_\varepsilon(u^h(\xi_1))}{F''_\varepsilon(u^h(\xi_2))} \  & \mbox{if} & u^h(\mathbf{a}_0^K) \neq u^h(\mathbf{a}_1^K),\\
(p-1)\frac{F'_\varepsilon(u^h(\mathbf{a}_0^K))}{F''_\varepsilon(u^h(\mathbf{a}_0^K))} \ & \mbox{if} & u^h(\mathbf{a}_0^K)= u^h(\mathbf{a}_1^K),
\end{array}\right. 
\end{equation}
for some $\xi_1,\xi_2\in K$. Following \cite{C5:BB} (see also \cite{C5:GR}), these constructions can be extended to dimensions 2 and 3, and from (\ref{C5:ConLa1}) the following estimate holds:
\begin{equation}\label{C5:D5uv}
\varepsilon^{2-p} \xi^T \xi \leq \xi^T \Lambda^{1}_\varepsilon(u^h)^{-1} \xi \leq \varepsilon^{p-2} \xi^T \xi, \ \ \forall \xi \in \mathbb{R}^d, \ u^h\in U_h.
\end{equation}
Now, we prove the following result which will be used in order 
to prove the well-posedness of the scheme \textbf{UV$\varepsilon$}. 
\begin{lem}\label{C5:lemconv}
Let $\Vert \cdot \Vert$ denote the spectral norm on $\mathbb{R}^{d\times d}$. Then for any given $\varepsilon\in (0,1)$ the function $\Lambda_\varepsilon^2:U_h\rightarrow [L^\infty(\Omega)]^{d\times d}$ satisfies, for all $u^h_1,u^h_2 \in U_h$ and $K\in \mathcal{T}_h$ with vertices $\{\mathbf{a}_l^K\}_{l=0}^d$,
\begin{eqnarray}\label{C5:eqL}
&&\!\!\!\!\!\!\Vert (\Lambda^2_\varepsilon(u^h_1) - \Lambda^2_\varepsilon(u^h_2))\vert_K \Vert\nonumber\\
&&\hspace{0.4 cm}\!\!\!\!\!\! \leq 3 \varepsilon^{2(p-2)}\max\{1,(p-1)\varepsilon^{2(p-2)}\} \max_{l=1,...,d} \{\vert u^h_1( \mathbf{a}_{l}^K) - u^h_2 (\mathbf{a}_{l}^K))\vert + \vert u^h_1( \mathbf{a}_{0}^K) - u^h_2 (\mathbf{a}_{0}^K)\vert  \},
\end{eqnarray}
where $\mathbf{a}^K_0$ is the right-angled vertex.
\end{lem}
\begin{proof}
The proof follows the ideas of \cite[Lemma 2.1]{C5:BN}, with some modifications. For simplicity in the notation, we will prove (\ref{C5:eqL}) in the 1-dimensional case, but this proof can be extended to dimensions 2 and 3 as in \cite[Lemma 2.1]{C5:BN}. Observe that, from (\ref{C5:ConLa2})
\begin{eqnarray}\label{C5:E0a}
&\Vert (\Lambda^2_\varepsilon(u^h_1) - \Lambda^2_\varepsilon(u^h_2))\vert_K \Vert&\!\!\! \leq \vert (\Lambda^2_\varepsilon(u^h_{1}) - \Lambda^2_\varepsilon(u^h_{1,2}))\vert_K \vert  + \vert (\Lambda^2_\varepsilon(u^h_{1,2}) - \Lambda^2_\varepsilon(u^h_{2}))\vert_K \vert\nonumber\\
&&\hspace{-3 cm} = (p-1)\left\vert \frac{F'_\varepsilon(\mu_{11})}{F''_\varepsilon(\mu_{12})} - \frac{F'_\varepsilon(\xi_{1})}{F''_\varepsilon(\xi_{2})}  \right\vert + (p-1)\left\vert   \frac{F'_\varepsilon(\xi_{1})}{F''_\varepsilon(\xi_{2})}  - \frac{F'_\varepsilon(\mu_{21})}{F''_\varepsilon(\mu_{22})} \right\vert,
\end{eqnarray}
where $u^h_{1,2} \in \mathbb{P}_1(K)$ with $u^h_{1,2}(\mathbf{a}_0^K)=u^h_2(\mathbf{a}_0^K)$ and $u^h_{1,2}(\mathbf{a}_1^K)=u^h_1(\mathbf{a}_1^K)$, $\mu_{1i}$ ($i=1,2$) lie between $u^h_1(\mathbf{a}_0^K)$ and $u^h_1(\mathbf{a}_1^K)$, $\mu_{2i}$ ($i=1,2$) lie between $u^h_2(\mathbf{a}_0^K)$ and $u^h_2(\mathbf{a}_1^K)$, and $\xi_{i}$ ($i=1,2$) lie between $u^h_1(\mathbf{a}_1^K)$ and $u^h_2(\mathbf{a}_0^K)$. Then, first we will show that
\begin{equation}\label{C5:E1a}
(p-1)\left\vert \frac{F'_\varepsilon(\mu_{11})}{F''_\varepsilon(\mu_{12})} - \frac{F'_\varepsilon(\xi_{1})}{F''_\varepsilon(\xi_{2})}  \right\vert \leq 3 \varepsilon^{2(p-2)}\max\{1,(p-1)\varepsilon^{2(p-2)}\} \vert u^h_1( \mathbf{a}_{0}^K) - u^h_2 (\mathbf{a}_{0}^K)\vert,
\end{equation}
for $u^h_1( \mathbf{a}_{0}^K) \neq u^h_2 (\mathbf{a}_{0}^K)$, because the case $u^h_1( \mathbf{a}_{0}^K)= u^h_2 (\mathbf{a}_{0}^K)$ is trivially true. With this aim, we consider $\gamma_i$ ($i=1,2$) lying between $u^h_1( \mathbf{a}_{0}^K)$ and $u^h_2 (\mathbf{a}_{0}^K)$ such that
\begin{equation}\label{C5:lambdapf1}
F'_\varepsilon(\gamma_1)=\frac{F_\varepsilon(u^h_2( \mathbf{a}_{0}^K)) - F_\varepsilon(u^h_1 (\mathbf{a}_{0}^K))}{u^h_2( \mathbf{a}_{0}^K) - u^h_1 (\mathbf{a}_{0}^K)} \ \ \mbox{ and } \ \ 
F''_\varepsilon(\gamma_2)=\frac{F'_\varepsilon(u^h_2( \mathbf{a}_{0}^K)) - F'_\varepsilon(u^h_1 (\mathbf{a}_{0}^K))}{u^h_2( \mathbf{a}_{0}^K) - u^h_1 (\mathbf{a}_{0}^K)},
\end{equation}
and therefore, from the definitions of $\xi_i$, $\gamma_i$ and $\mu_{1i}$, $i=1,2$, given after (\ref{C5:E0a}) and (\ref{C5:lambdapf1}), we deduce
\begin{equation}\label{C5:Ac1}
(u^h_2( \mathbf{a}_{0}^K) - u^h_1 (\mathbf{a}_{0}^K))F'_\varepsilon(\gamma_1)= (u^h_2( \mathbf{a}_{0}^K) - u^h_1 (\mathbf{a}_{1}^K))F'_\varepsilon(\xi_1) + (u^h_1( \mathbf{a}_{1}^K) - u^h_1 (\mathbf{a}_{0}^K)) F'_\varepsilon(\mu_{11}),
\end{equation}
\begin{equation}\label{C5:Ac2}
(u^h_2( \mathbf{a}_{0}^K) - u^h_1 (\mathbf{a}_{0}^K))F''_\varepsilon(\gamma_2)= (u^h_2( \mathbf{a}_{0}^K) - u^h_1 (\mathbf{a}_{1}^K))F''_\varepsilon(\xi_2) + (u^h_1( \mathbf{a}_{1}^K) - u^h_1 (\mathbf{a}_{0}^K)) F''_\varepsilon(\mu_{12}).
\end{equation}
Then, for $u^h_2( \mathbf{a}_{0}^K)$, $u^h_1( \mathbf{a}_{0}^K)$ and $u^h_1( \mathbf{a}_{1}^K)$, there are only 3 options: (1) $u^h_1( \mathbf{a}_{1}^K)$ lies between $u^h_2( \mathbf{a}_{0}^K)$ and $u^h_1( \mathbf{a}_{0}^K)$; (ii) $u^h_2( \mathbf{a}_{0}^K)$ lies between $u^h_1( \mathbf{a}_{1}^K)$ and $u^h_1( \mathbf{a}_{0}^K)$; and (iii) $u^h_1( \mathbf{a}_{0}^K)$ lies between $u^h_1( \mathbf{a}_{1}^K)$ and $u^h_2( \mathbf{a}_{0}^K)$. 

Notice that from (\ref{C5:F2pE})-(\ref{C5:lE}), we have that $F'_\varepsilon$ and $(p-1)\frac{F'_\varepsilon}{F''_\varepsilon}$ are globally Lipschitz functions with constants $\varepsilon^{p-2}$ and $1$ respectively, and $\frac{1}{\vert F''_\varepsilon \vert}\leq \varepsilon^{p-2}$.  Then, in case (i), taking into account that all intermediate values $\mu_{1i}, \gamma_i,\xi_i$ ($i=1,2$) lie between $u^h_2( \mathbf{a}_{0}^K)$ and $u^h_1( \mathbf{a}_{0}^K)$, we have
\begin{eqnarray}\label{C5:C1}
&(p-1)&\!\!\!\left\vert \frac{F'_\varepsilon(\mu_{11})}{F''_\varepsilon(\mu_{12})} - \frac{F'_\varepsilon(\xi_{1})}{F''_\varepsilon(\xi_{2})}  \right\vert\leq (p-1)\left\vert \frac{F'_\varepsilon(\mu_{11})- F'_\varepsilon(\mu_{12})}{F''_\varepsilon(\mu_{12})} \right\vert\nonumber\\
&&\!\!\! + (p-1)\left\vert \frac{F'_\varepsilon(\mu_{12})}{F''_\varepsilon(\mu_{12})}  - \frac{F'_\varepsilon(\xi_{2})}{F''_\varepsilon(\xi_{2})} \right\vert +(p-1)\left\vert \frac{F'_\varepsilon(\xi_{1})-F'_\varepsilon(\xi_{2})}{F''_\varepsilon(\xi_{2})}  \right\vert\nonumber\\
&&\!\!\!\leq (p-1)\varepsilon^{2(p-2)}\vert \mu_{11} -\mu_{12}\vert + \vert \mu_{12}- \xi_2 \vert + (p-1)\varepsilon^{2(p-2)}\vert \xi_{1} -\xi_{2}\vert\nonumber\\
&&\!\!\!\leq 3\max\{1,(p-1)\varepsilon^{2(p-2)}\} \vert u^h_1( \mathbf{a}_{0}^K) - u^h_2 (\mathbf{a}_{0}^K)\vert.
\end{eqnarray}

In case (ii), all intermediate values $\mu_{1i}, \gamma_i,\xi_i$ ($i=1,2$) lie between $u^h_1( \mathbf{a}_{1}^K)$ and $u^h_1( \mathbf{a}_{0}^K)$, and from (\ref{C5:Ac1})-(\ref{C5:Ac2}) by eliminating the term $(u^h_2( \mathbf{a}_{0}^K) - u^h_1 (\mathbf{a}_{1}^K))$, we have the equality
\begin{eqnarray*}
(u^h_1( \mathbf{a}_{1}^K) - u^h_1 (\mathbf{a}_{0}^K)) \left[ \frac{F'_\varepsilon(\xi_{1})}{F''_\varepsilon(\xi_{2})} - \frac{F'_\varepsilon(\mu_{11})}{F''_\varepsilon(\mu_{12})} \right]= (u^h_2( \mathbf{a}_{0}^K) - u^h_1 (\mathbf{a}_{0}^K)) \frac{F''_\varepsilon(\gamma_{2})}{F''_\varepsilon(\mu_{12})} \left[ \frac{F'_\varepsilon(\xi_{1})}{F''_\varepsilon(\xi_{2})} - \frac{F'_\varepsilon(\gamma_{1})}{F''_\varepsilon(\gamma_{2})} \right],
\end{eqnarray*}
from which, bounding the term $\left\vert \frac{F'_\varepsilon(\xi_{1})}{F''_\varepsilon(\xi_{2})} - \frac{F'_\varepsilon(\gamma_{1})}{F''_\varepsilon(\gamma_{2})} \right\vert$ as in (\ref{C5:C1}),  we obtain
\begin{eqnarray*}
&(p-1)&\!\!\!\vert u^h_1( \mathbf{a}_{1}^K) - u^h_1 (\mathbf{a}_{0}^K))\vert \left\vert \frac{F'_\varepsilon(\mu_{11})}{F''_\varepsilon(\mu_{12})} - \frac{F'_\varepsilon(\xi_{1})}{F''_\varepsilon(\xi_{2})}  \right\vert \\
&&\!\!\!\!\!\leq \varepsilon^{2(p-2)}3\max\{1,(p-1)\varepsilon^{2(p-2)}\} \vert u^h_1( \mathbf{a}_{0}^K) - u^h_2 (\mathbf{a}_{0}^K)\vert \vert u^h_1( \mathbf{a}_{1}^K) - u^h_1 (\mathbf{a}_{0}^K))\vert,
\end{eqnarray*}
and therefore, dividing by $\vert u^h_1( \mathbf{a}_{1}^K) - u^h_1 (\mathbf{a}_{0}^K))\vert$ we arrive at
\begin{equation}\label{C5:C2}
(p-1)\left\vert \frac{F'_\varepsilon(\mu_{11})}{F''_\varepsilon(\mu_{12})} - \frac{F'_\varepsilon(\xi_{1})}{F''_\varepsilon(\xi_{2})}  \right\vert \leq 3\varepsilon^{2(p-2)}\max\{1,(p-1)\varepsilon^{2(p-2)}\} \vert u^h_1( \mathbf{a}_{0}^K) - u^h_2 (\mathbf{a}_{0}^K)\vert.
\end{equation}
In case (iii), by arguing analogously to case (ii), from (\ref{C5:Ac1})-(\ref{C5:Ac2}) we have
\begin{eqnarray*}
(u^h_1( \mathbf{a}_{1}^K) - u^h_2 (\mathbf{a}_{0}^K)) \left[ \frac{F'_\varepsilon(\xi_{1})}{F''_\varepsilon(\xi_{2})} - \frac{F'_\varepsilon(\mu_{11})}{F''_\varepsilon(\mu_{12})} \right]= (u^h_2( \mathbf{a}_{0}^K) - u^h_1 (\mathbf{a}_{0}^K)) \frac{F''_\varepsilon(\gamma_{2})}{F''_\varepsilon(\xi_{2})} \left[ \frac{F'_\varepsilon(\gamma_{1})}{F''_\varepsilon(\gamma_{2})} - \frac{F'_\varepsilon(\mu_{11})}{F''_\varepsilon(\mu_{12})} \right],
\end{eqnarray*}
which implies (\ref{C5:C2}). Therefore, we have proved (\ref{C5:E1a}). Analogously, we can prove that
\begin{equation}\label{C5:E2a}
(p-1)\left\vert   \frac{F'_\varepsilon(\xi_{1})}{F''_\varepsilon(\xi_{2})}  - \frac{F'_\varepsilon(\mu_{21})}{F''_\varepsilon(\mu_{22})} \right\vert \leq 3 \varepsilon^{2(p-2)}\max\{1,(p-1)\varepsilon^{2(p-2)}\} \vert u^h_1( \mathbf{a}_{1}^K) - u^h_2 (\mathbf{a}_{1}^K)\vert.
\end{equation}
Thus, from (\ref{C5:E0a}), (\ref{C5:E1a}) and (\ref{C5:E2a}) we conclude (\ref{C5:eqL}).
\end{proof}

Let $A_h: V_h \rightarrow V_h$ be the linear operator defined as follows
\begin{equation*}
(A_h v^h, \bar{v})=(\nabla v^h,\nabla \bar{v})+( v^h, \bar{v}) , \ \ \forall \bar{v}\in V_h.
\end{equation*}
Then, the following estimate holds (see for instance, \cite[Theorem 3.2]{C5:FMD2}):
\begin{equation}\label{C5:estW16}
\Vert v^h \Vert_{W^{1,6}}\leq C \Vert A_h v^h\Vert_0, \ \ \forall v^h \in V_h.
\end{equation}
 Thus, we consider the following first order in time, nonlinear and coupled scheme: 
\begin{itemize}
\item{\underline{\emph{Scheme \textbf{UV$\varepsilon$}}:}\\
{\bf Initialization}:   Let $(u^{0},v^{0})=(Q^h u_0, R^h v_0)\in  U_h\times V_h$.\\
{\bf Time step} n: Given $(u^{n-1}_\varepsilon,v^{n-1}_\varepsilon)\in  U_h\times V_h$, compute $(u^{n}_\varepsilon,v^{n}_\varepsilon)\in  U_h \times V_h$ solving
\begin{equation}
\left\{
\begin{array}
[c]{lll}%
(\delta_t u^n_\varepsilon,\bar{u})^h + (\nabla u^n_\varepsilon,\nabla \bar{u}) = -(\Lambda_\varepsilon^2 (u^{n}_\varepsilon)\nabla v^n_\varepsilon,\nabla \bar{u}), \ \ \forall \bar{u}\in U_h,\\
(\delta_t {v}^n_\varepsilon,\bar{v}) +(A_h  v^n_\varepsilon, \bar{v}) = p(p-1)( \Pi^h (F_\varepsilon(u^n_\varepsilon)) ,\bar{v}),\ \ \forall
\bar{v}\in V_h,
\end{array}
\right.  \label{C5:modelf02aUVreg}
\end{equation}}
\end{itemize}
where, in general, we denote $\delta_t a^n:= \displaystyle\frac{a^n - a^{n-1}}{k}$.

\begin{obs}{\bf (Positivity of $v^n_\varepsilon$)}\label{C5:posVuv}
By using the mass-lumping technique in all terms of (\ref{C5:modelf02aUVreg})$_2$ excepting the self-diffusion term $(\nabla v^n_\varepsilon,\nabla\bar{v})$, and approximating $V_h$ by $\mathbb{P}_1$-continuous FE, we can prove that if $v^{n-1}_\varepsilon\geq 0$ then $v^n_\varepsilon\geq 0$. In fact, it follows testing (\ref{C5:modelf02aUVreg})$_2$  by $\bar{v}=\Pi^h (v^n_{\varepsilon-})\in V_h$, where $v^n_{\varepsilon-}:=\min\{v^n_\varepsilon, 0 \}$ (see Remark 3.10 in \cite{C5:FMD4}).
\end{obs}

\subsubsection{Mass-conservation, Energy-stability and Solvability}\label{C5:ELusreg}
Since $\bar{u}=1\in U_h$ and $\bar{v}=1 \in V_h$, we deduce that the scheme \textbf{UV$\varepsilon$} is conservative in $u^n_\varepsilon$, that is,
\begin{equation}\label{C5:conu1regUV}
(u_\varepsilon^n,1)=(u^n_\varepsilon,1)^h= (u^{n-1}_\varepsilon,1)^h=\cdot\cdot\cdot= (u^{0},1)^h=(u^0,1)=(Q^h u_0,1)=(u_0,1):=m_0,
\end{equation}
and we have the following behavior for $\int_\Omega v^n_\varepsilon$: 
\begin{equation}\label{C5:conv1UV}
\delta_t \left(\int_\Omega v^n_\varepsilon \right)= p(p-1)\int_\Omega \Pi^h (F_\varepsilon(u^{n}_{\varepsilon}))  - \int_\Omega v^n_\varepsilon.
\end{equation}

\begin{defi}\label{C5:enesfUVreg}
A numerical scheme with solution $(u^n_\varepsilon,v^n_\varepsilon)$ is called energy-stable with respect to the  energy 
\begin{equation}\label{C5:ENuvreg}
\mathcal{E}_\varepsilon^h(u,v)=p(F_\varepsilon(u),1)^h + \frac{1}{2} \Vert \nabla v\Vert_0^2
\end{equation}
if this energy is time decreasing, that is $\mathcal{E}_\varepsilon^h(u^n_\varepsilon,v^n_\varepsilon)\leq \mathcal{E}_\varepsilon^h(u^{n-1}_\varepsilon,v^{n-1}_\varepsilon)$ for all $n\geq 1$.
\end{defi}

\begin{tma} {\bf (Unconditional stability)} \label{C5:estinc1uvreg}
The scheme \textbf{UV$\varepsilon$} is unconditionally
energy stable with respect to $\mathcal{E}_\varepsilon^h(u,v)$. In fact, if $(u^n_\varepsilon,v^n_\varepsilon)$ is a solution of \textbf{UV$\varepsilon$}, then the following discrete energy law holds
\begin{equation}\label{C5:deluvreg}
\delta_t \mathcal{E}_\varepsilon^h(u^n_\varepsilon,v^n_\varepsilon) + \frac{k\varepsilon^{2-p}p}{2}\Vert \delta_t u^n_\varepsilon\Vert_0^2 + \frac{k}{2} \Vert \delta_t \nabla v^n_\varepsilon\Vert_0^2 +  p\varepsilon^{2-p}\Vert \nabla u^n_\varepsilon\Vert_0^2 +\Vert (A_h - I) \nabla v^n_\varepsilon\Vert_0^2 + \Vert \nabla v^n_\varepsilon\Vert_0^2 \leq 0.
\end{equation}
\end{tma}
\begin{proof}
Testing (\ref{C5:modelf02aUVreg})$_1$ by $\bar{u}= p\Pi^h (F'_\varepsilon (u^n_\varepsilon))$ and (\ref{C5:modelf02aUVreg})$_2$ by $\bar{v}=(A_h - I) v^n_\varepsilon$, adding and taking into account that $\Lambda_\varepsilon^i(u^n_\varepsilon)$ are symmetric as well as (\ref{C5:PL15uv})-(\ref{C5:PL16uv}), the terms $-p(\Lambda^2_\varepsilon (u^{n}_\varepsilon) \nabla v^n_\varepsilon, \nabla  \Pi^h(F'_\varepsilon(u^n_\varepsilon)))=-p( \nabla v^n_\varepsilon, \Lambda^2_\varepsilon (u^{n}_\varepsilon)\nabla  \Pi^h(F'_\varepsilon(u^n_\varepsilon)))=- p(p-1)(\nabla v^n_\varepsilon,\nabla  \Pi^h (F_\varepsilon(u^n_\varepsilon))) $ and $p(p-1)(\Pi^h (F_\varepsilon(u^n_\varepsilon)),(A_h-I) v^n_\varepsilon)=p(p-1)(\nabla  \Pi^h (F_\varepsilon(u^n_\varepsilon)) ,\nabla v^n_\varepsilon)$ cancel, and using that $\nabla \Pi^h (F'_\varepsilon (u^n_\varepsilon))=\Lambda_\varepsilon^1(u^{n}_\varepsilon)^{-1} \nabla u^n_\varepsilon$ we obtain
\begin{eqnarray}\label{C5:I01aUVreg}
&p(\delta_t u^n_\varepsilon,F'_\varepsilon (u^n_\varepsilon))^h&\!\!\!\! + p\int_\Omega (\nabla u^n_\varepsilon)^T\!\cdot\!\Lambda_\varepsilon^1(u^{n}_\varepsilon)^{-1}\!\cdot\! \nabla u^n_\varepsilon d\x  \nonumber\\
&&\!\!\!\! + \delta_t \Big(\frac{1}{2} \Vert \nabla v^n_\varepsilon\Vert_0^2\Big) + \frac{k}{2} \Vert \delta_t \nabla v^n_\varepsilon\Vert_0^2 +\Vert  (A_h - I) v^n_\varepsilon\Vert_0^2+ \Vert \nabla v^n_\varepsilon\Vert_0^2  = 0.
\end{eqnarray}
Moreover, observe that from the Taylor formula we have
\begin{equation*}
F_\varepsilon(u^{n-1}_\varepsilon)=F_\varepsilon(u^n_\varepsilon)+F'_\varepsilon(u^n_\varepsilon)(u^{n-1}_\varepsilon-u^n_\varepsilon)+\frac{1}{2}F''_\varepsilon(\theta
u^n_\varepsilon+(1-\theta)u^{n-1}_\varepsilon)(u^{n-1}_\varepsilon-u^n_\varepsilon)^2,
\end{equation*}
and therefore,
\begin{equation}\label{C5:I01breguv}
\delta_t u^n_\varepsilon \cdot F'_\varepsilon (u^n_\varepsilon) = \delta_t \Big( F_\varepsilon(u^n_\varepsilon) \Big) + \frac{k}{2} F''_\varepsilon(\theta
u^n_\varepsilon+(1-\theta)u^{n-1}_\varepsilon)(\delta_t u^n_\varepsilon)^2.
\end{equation}
Then, using (\ref{C5:I01breguv}) and taking into account that $\Pi^h$ is linear and $F''_\varepsilon(s)\geq \varepsilon^{2-p}$ for all $s\in \mathbb{R}$, we have
\begin{eqnarray}\label{C5:I01creguv}
&(\delta_t u^n_\varepsilon,F'_\varepsilon (u^n_\varepsilon))^h&\!\!\!= \delta_t \Big( \int_\Omega \Pi^h (F_\varepsilon(u^n_\varepsilon)) \Big) + \frac{k}{2} \int_\Omega \Pi^h(F''_\varepsilon(\theta
u^n_\varepsilon+(1-\theta)u^{n-1}_\varepsilon)(\delta_t u^n_\varepsilon)^2)\nonumber\\
&& \!\!\! \geq \delta_t \Big((F_\varepsilon(u^n_\varepsilon),1)^h \Big) + \frac{k\varepsilon^{2-p}}{2}\vert \delta_t u^n_\varepsilon\vert_h^2.
\end{eqnarray}
Thus, from (\ref{C5:I01aUVreg}),  (\ref{C5:D5uv}), (\ref{C5:I01creguv}) and Remark \ref{C5:eqh25uv}, we arrive at (\ref{C5:deluvreg}).
\end{proof}

\begin{cor} \label{C5:welemUVreg} {\bf(Uniform estimates) }
Assume that $(u_0,v_0)\in L^2(\Omega)\times H^1(\Omega)$. Let $(u^n_\varepsilon,v^n_\varepsilon)$ be a solution of scheme \textbf{UV$\varepsilon$}. Then, it holds 
\begin{equation}\label{C5:weak01uv-aUVreg1}
p(F_\varepsilon(u^n_\varepsilon),1)^h + \frac{1}{2}\Vert v^n_\varepsilon\Vert_{1}^{2}+k \underset{m=1}{\overset{n}{\sum}}\left( p\varepsilon^{2-p}\Vert \nabla u^m_\varepsilon\Vert_0^2 +\Vert  (A_h - I) v^m_\varepsilon\Vert_0^2 + \Vert \nabla v^m_\varepsilon\Vert_0^2\right)\leq \frac{C_0}{(p-1)^2}, \ \ \forall n\geq 1,
\end{equation}
\begin{equation}\label{C5:weak01uv-aUVreg2}
k \underset{m=n_0 + 1}{\overset{n+n_0}{\sum}} \Vert v^m_\varepsilon\Vert_{W^{1,6}}^2 \leq \frac{C_1}{(p-1)^2}(1+kn), \ \ \forall n\geq 1,
\end{equation}
where the integer $n_0\geq 0$ is arbitrary, with the constants $C_0,C_1>0$ depending on the data $(\Omega,u_0, v_0)$, but independent of $k,h,n$ and $\varepsilon$. Moreover,
\begin{equation}\label{C5:pneg1-aUVreg}
\Vert \Pi^h (u^n_{\varepsilon-})\Vert_0^2 \leq \frac{C_0}{(p-1)^2}\varepsilon^{2-p} \ \ \mbox{ and} \ \
\Vert u^n_\varepsilon\Vert_{L^p}^p \leq \frac{C_0K}{(p-1)^2} + K, \ \ \forall n\geq 1,
\end{equation}
where $u^n_{\varepsilon-}:=\min\{u^n_\varepsilon, 0 \}\leq 0$ and the constant $K>0$ is independent of $k,h,n$ and $\varepsilon$.
\end{cor}
\begin{obs}\label{C5:NNuh1uv}{\bf (Approximated positivity of $u^n_\varepsilon$)}
From (\ref{C5:pneg1-aUVreg})$_1$, the following estimate holds 
$$\max_{n\geq 0} \Vert \Pi^h (u^n_{\varepsilon-})\Vert_0^2 \leq \frac{C_0}{(p-1)^2}\varepsilon^{2-p}.$$
 \end{obs}
\begin{proof}
First, taking into account that $(u^{0},{v}^{0})=(Q^h u_0, R^h v_0)$, $u_0\geq 0$ (and therefore, $u^{0}\geq 0$), as well as the definition of $F_\varepsilon$, we have that
\begin{eqnarray}\label{C5:ee55uv}
&\mathcal{E}^h_\varepsilon(u^0,v^0)&\!\!\!= p\int_\Omega \Pi^h (F_\varepsilon(u^0)) +  \frac{1}{2}\Vert \nabla v^0\Vert_{0}^{2} \leq \frac{C}{p-1}\int_\Omega \Pi^h \Big((u^0)^2 + \frac{1}{p-1}\Big) +  \frac{1}{2}\Vert  \nabla v^0\Vert_{0}^{2}\nonumber\\
&&\!\!\!\!\!\!\!\!\!\!\!\!\!\!\!\!\! \leq \frac{C}{p-1}\Big(\Vert u^0\Vert_0^2 + \Vert  \nabla v^0\Vert_{0}^{2} + \frac{1}{p-1}\Big) \leq \frac{C}{p-1}\Big(\Vert u_0\Vert_0^2 + \Vert v_0\Vert_{1}^{2} + \frac{1}{p-1}\Big)\leq \frac{C_0}{(p-1)^2},
\end{eqnarray}
where the constant $C_0>0$ depends on the data $(\Omega, u_0, v_0)$, but is independent of $k,h,n$ and $\varepsilon$. Therefore,  from the discrete energy law (\ref{C5:deluvreg}) and estimate (\ref{C5:ee55uv}), we have
\begin{equation}\label{C5:ni2}
\mathcal{E}^h_\varepsilon(u^n_\varepsilon,v^n_\varepsilon)+k \underset{m=1}{\overset{n}{\sum}}\left( p\varepsilon^{2-p}\Vert\nabla u^m_\varepsilon\Vert_0^2 +\Vert (A_h - I) v^m_\varepsilon\Vert_0^2 + \Vert \nabla v^m_\varepsilon\Vert_0^2\right) \leq  \mathcal{E}^h_\varepsilon(u^0,v^0)\leq \frac{C_0}{(p-1)^2}.
\end{equation}
Moreover, from (\ref{C5:conv1UV}), the definition of $F_\varepsilon$, Remark \ref{C5:OBSr} and (\ref{C5:ni2}), we have 
\begin{eqnarray}\label{C5:esV1}
&(1+k)\displaystyle \left\vert \int_\Omega v^n_\varepsilon \right\vert - \left\vert \int_\Omega v^{n-1}_\varepsilon \right\vert&\!\!\! \leq kp(p-1)\int_\Omega \Pi^h (F_\varepsilon(u^{n}_{\varepsilon}))\leq  k \frac{C}{p-1},
\end{eqnarray}
where the constant $C>0$ is independent of $k,h,n$ and $\varepsilon$. Then, applying Lemma \ref{C5:tmaD} in (\ref{C5:esV1}) (for $\delta=1$ and $\beta=\frac{C}{p-1}$), we arrive at
$$
\left\vert \int_\Omega v^n_\varepsilon \right\vert \leq (1+k)^{-n}  \left\vert \int_\Omega v^0_h \right\vert+ \frac{C}{p-1}= (1+k)^{-n}  \left\vert \int_\Omega R^h v_0 \right\vert+\frac{C}{p-1},
$$
which, together with (\ref{C5:ni2}), imply (\ref{C5:weak01uv-aUVreg1}). Moreover, adding (\ref{C5:deluvreg}) from $m=n_0+1$ to $m=n+n_0$, and using (\ref{C5:estW16}) and (\ref{C5:weak01uv-aUVreg1}), we deduce (\ref{C5:weak01uv-aUVreg2}). On the other hand, from (\ref{C5:es11})$_1$, we have $\frac{1}{4} \varepsilon^{p-2} (u^n_{\varepsilon-}(\x))^2 \leq F_\varepsilon (u^n_\varepsilon(\x))$ for all $u^n_\varepsilon \in U_h$; and therefore, using that $(\Pi^h u)^2\leq \Pi^h (u^2)$ for all $u\in C(\overline{\Omega})$, we have
\begin{equation*}
 \frac{1}{4} \varepsilon^{p-2}\int_\Omega (\Pi^h (u^n_{\varepsilon-}))^2 \leq \frac{1}{4}\varepsilon^{p-2}\int_\Omega \Pi^h ((u^n_{\varepsilon-})^2) \leq  \int_\Omega \Pi^h (F_\varepsilon(u^n_\varepsilon)) \leq \frac{C_0}{(p-1)^2},
\end{equation*}
where estimate (\ref{C5:weak01uv-aUVreg1}) was used in the last inequality. Thus, we obtain (\ref{C5:pneg1-aUVreg})$_1$. Finally, taking into account that $\vert\Pi^h u\vert^p\leq \Pi^h (\vert u\vert ^p)$ for all $u\in C(\overline{\Omega})$, as well as Remark \ref{C5:OBSr} and (\ref{C5:weak01uv-aUVreg1}),  we have
\begin{eqnarray*}
\Vert u^n_\varepsilon\Vert_{L^p}^p = \int_\Omega \vert \Pi^h u^n_\varepsilon\vert^p \leq \int_\Omega \Pi^h (\vert u^n_\varepsilon\vert^p)\leq \int_\Omega \Pi^h (K_1 F_\varepsilon(u^n_\varepsilon)+K_2) \leq \frac{C_0K}{(p-1)^2} + K,
\end{eqnarray*}
arriving at (\ref{C5:pneg1-aUVreg})$_2$.
\end{proof}

\begin{tma} {\bf (Unconditional existence)} \label{uncond}
There exists at least one solution $(u^n_\varepsilon,v^n_\varepsilon)$  of scheme \textbf{UV$\varepsilon$}.
\end{tma}
\begin{proof}
The proof follows by using the Leray-Schauder fixed point theorem. With this aim, given $(u^{n-1}_\varepsilon,v^{n-1}_\varepsilon)\in U_h\times V_h$, we define the operator $R:U_h\times V_h\rightarrow U_h\times V_h$ by  $R(\widetilde{u},\widetilde{v})=(u,{v})$, such that $(u,v)\in U_h\times V_h$ solves the following linear decoupled problems
$$
u\in U_h \ \ \mbox{s.t. } \ \displaystyle\frac{1}{k}(u,\bar{u})^h + (\nabla u, \nabla\bar{u}) =\displaystyle\frac{1}{k}(u^{n-1}_\varepsilon,\bar{u})^h -(\Lambda_\varepsilon^2(\widetilde{u})\nabla \widetilde{v},\nabla \bar{u}), \ \ \forall \bar{u}\in U_h,
$$
$$
v\in V_h \ \ \mbox{s.t. } \ \displaystyle\frac{1}{k}(v,\bar{v}) + (A_h v, \bar{v}) =\displaystyle\frac{1}{k}(v^{n-1}_\varepsilon,\bar{v})
+p(p-1)(\Pi^h(F_\varepsilon(\widetilde{u})), \bar{v}), \ \ \forall \bar{v}\in V_h.
$$
The hypotheses of the Leray-Schauder fixed point theorem are satisfied as in Theorem 3.11 of \cite{C5:FMD4}, but applying in this case Lemma \ref{C5:lemconv} in order to prove the continuity of the operator $R$. Thus, we conclude that the map $R$ has a fixed point $(u,v)$, that is
$R(u,{v})=(u,{v})$, which is a solution of the scheme \textbf{UV$\varepsilon$}.
\end{proof}

\subsection{Scheme US$\varepsilon$}
In this section, in order to construct another energy-stable fully discrete scheme 
{\color{blue} for}
(\ref{C5:modelf00}), we are going to use the regularized functions $F_\varepsilon$, $F'_\varepsilon$ and $F''_\varepsilon$ defined in Section \ref{C5:Suv} and we will consider the auxiliary variable ${\boldsymbol\sigma}=\nabla v$. Then, another regularized version of problem (\ref{C5:modelf00}) reads: Find $u_\varepsilon:\Omega\times [0,T]\rightarrow \mathbb{R}$ and ${\boldsymbol\sigma}_\varepsilon:\Omega\times [0,T]\rightarrow \mathbb{R}^d$, with $u_\varepsilon\geq 0$, such that
\begin{equation}
\left\{
\begin{array}
[c]{lll}%
\partial_t u_\varepsilon -\Delta  u_\varepsilon  -\nabla\cdot( u_\varepsilon {\boldsymbol\sigma}_\varepsilon)=0\ \ \mbox{in}\ \Omega,\ t>0,\\
\partial_t {\boldsymbol \sigma}_\varepsilon + \mbox{rot(rot } {\boldsymbol \sigma}_\varepsilon \mbox{)} -\nabla(\nabla \cdot {\boldsymbol \sigma}_\varepsilon) + {\boldsymbol \sigma}_\varepsilon = p\; u_\varepsilon\nabla (F'_\varepsilon(u_\varepsilon))\ \ \mbox{in}\ \Omega,\ t>0,\\
\displaystyle\frac{\partial u_\varepsilon}{\partial \mathbf{n}}=0,\ \ {\boldsymbol \sigma}_\varepsilon\cdot \mathbf{n}=0, \ \ \left[\mbox{rot }{\boldsymbol \sigma}_\varepsilon \times \mathbf{n}\right]_{tang}=0 \quad \mbox{on}\ \partial\Omega,\ t>0,\\
u_\varepsilon(\x ,0)=u_0(\x )\geq 0,\ {\boldsymbol \sigma}_\varepsilon(\x ,0)=\nabla v_0(
\x ),\quad \mbox{in}\ \Omega.
\end{array}
\right.  \label{C5:modelf02acontUSreg}
\end{equation}
This kind of formulation considering ${\boldsymbol\sigma}=\nabla v$ as auxiliary variable has been used in the construction of numerical schemes for other chemotaxis models (see for instance \cite{C5:FMD2,C5:FMD4,C5:Z1}). Once problem (\ref{C5:modelf02acontUSreg}) is solved, we can recover $v_\varepsilon$ from $u_\varepsilon$  by solving 
\begin{equation*} \left\{
\begin{array}
[c]{lll}%
\partial_t v_\varepsilon -\Delta v_\varepsilon + v_\varepsilon = u^p_\varepsilon  \quad \mbox{in}\ \Omega,\ t>0,\\
\displaystyle
\frac{\partial v_\varepsilon}{\partial \mathbf{n}}=0\quad\mbox{on}\ \partial\Omega,\ t>0,\\
 v_\varepsilon(\x ,0)=v_0(\x )\geq 0\quad \mbox{in}\ \Omega.
\end{array}
\right.  \label{C5:modelf01eqvreg}
\end{equation*}
Observe that (formally) multiplying (\ref{C5:modelf02acontUSreg})$_1$ by $p F'_\varepsilon(u_\varepsilon)$, (\ref{C5:modelf02acontUSreg})$_2$ by ${\boldsymbol\sigma}_\varepsilon$, integrating over $\Omega$ and adding both equations, the terms $p(u_\varepsilon\nabla (F'_\varepsilon(u_\varepsilon)),{\boldsymbol \sigma}_\varepsilon)$ cancel, and we obtain the following energy law
$$
\frac{d}{dt} \displaystyle \int_\Omega \Big( p F_\varepsilon(u_\varepsilon) + \frac{1}{2} \vert {\boldsymbol\sigma}_\varepsilon\vert^2\Big) d \x + \int_\Omega p F''_\varepsilon(u_\varepsilon)\vert \nabla u_\varepsilon\vert^2 d\x+\Vert {\boldsymbol\sigma}_\varepsilon\Vert_1^2 = 0.
$$
In particular, the modified energy $\mathcal{E}_\varepsilon(u,{\boldsymbol\sigma})=\displaystyle \int_\Omega \Big( p F_\varepsilon(u) + \frac{1}{2} \vert {\boldsymbol\sigma}\vert^2\Big) d \x $ is decreasing in time. Then, we consider a fully discrete approximation of the regularized problem (\ref{C5:modelf02acontUSreg}) using a FE discretization in space and the backward
Euler discretization in time (considered for simplicity on a uniform partition of $[0,T]$ with time step $k=T/N : (t_n = nk)_{n=0}^{n=N}$). Concerning the space discretization, we consider the triangulation as in the scheme \textbf{UV$\varepsilon$}, imposing again the constraint ({\bf H}) related with the right angled simplices. We choose the following continuous FE spaces for $u_\varepsilon$, ${\boldsymbol\sigma}_\varepsilon$, and $v_\varepsilon$:
$$(U_h,{\boldsymbol\Sigma}_h, V_h) \subset H^1(\Omega)^3,\quad \hbox{generated by $\mathbb{P}_1,\mathbb{P}_{m},\mathbb{P}_r$ with $m,r\geq 1$.}
$$
\begin{obs}
The right-angled constraint ({\bf H}) and the approximation of $U_h$ by $\mathbb{P}_1$-continuous FE are again necessary in order to obtain the relation (\ref{C5:PL15uv}) and estimate (\ref{C5:D5uv}) for $\Lambda_\varepsilon^1$, which are essential in order to obtain the energy-stability of the scheme \textbf{US$\varepsilon$} (see Theorem \ref{C5:estinc1usreg} below).
\end{obs}
Then, we consider the following first order in time, nonlinear and coupled scheme: 
\begin{itemize}
\item{\underline{\emph{Scheme \textbf{US$\varepsilon$}}:}\\
{\bf Initialization}:   Let $(u^{0},{\boldsymbol \sigma}^{0})=(Q^h u_0, \widetilde{Q}^h (\nabla v_0))\in  U_h\times {\boldsymbol\Sigma}_h$.\\
{\bf Time step} n: Given $(u^{n-1}_\varepsilon,{\boldsymbol \sigma}^{n-1}_\varepsilon)\in  U_h\times {\boldsymbol\Sigma}_h$, compute $(u^{n}_\varepsilon,{\boldsymbol \sigma}^{n}_\varepsilon)\in  U_h \times {\boldsymbol\Sigma}_h$ solving
\begin{equation}
\left\{
\begin{array}
[c]{lll}%
(\delta_t u^n_\varepsilon,\bar{u})^h + (\nabla u^n_\varepsilon,\nabla \bar{u}) = -(u^{n}_\varepsilon {\boldsymbol\sigma}^n_\varepsilon,\nabla \bar{u}), \ \ \forall \bar{u}\in U_h,\\
(\delta_t {\boldsymbol \sigma}^n_\varepsilon,\bar{\boldsymbol \sigma}) + (B_h {\boldsymbol \sigma}^n_\varepsilon,\bar{\boldsymbol \sigma}) =
p(u^{n}_\varepsilon \nabla  \Pi^h(F'_\varepsilon(u^n_\varepsilon)),\bar{\boldsymbol \sigma}),\ \ \forall
\bar{\boldsymbol \sigma}\in \Sigma_h,
\end{array}
\right.  \label{C5:modelf02aUSreg}
\end{equation}}
\end{itemize}
where $Q^h$ is the $L^2$-projection on $U_h$ defined in (\ref{C5:MLP25uv}), $\widetilde{Q}^h$ is the standard $L^2$-projection on ${\boldsymbol\Sigma}_h$, and the operator $B_h$ is defined as
$$(B_h  {\boldsymbol \sigma}^n_\varepsilon,\bar{\boldsymbol \sigma}) = (\mbox{rot } {\boldsymbol\sigma}_\varepsilon^n,\mbox{rot } \bar{\boldsymbol\sigma}) + (\nabla \cdot {\boldsymbol\sigma}_\varepsilon^n, \nabla \cdot \bar{\boldsymbol\sigma}) + ({\boldsymbol\sigma}_\varepsilon^n, \bar{\boldsymbol\sigma}), \ \ \forall \bar{\boldsymbol\sigma}\in {\boldsymbol\Sigma}_h.$$
We recall that $\Pi^h: C(\overline{\Omega})\rightarrow U_h$ is the Lagrange interpolation operator, and the discrete semi-inner product $(\cdot,\cdot)^h$ was defined in (\ref{C5:mlumpreguv}). 
\begin{obs}
Notice that the right-angled constraint ({\bf H}) is necessary in the implementation of the scheme \textbf{UV$\varepsilon$} (in order to construct the matricial function $\Lambda^2_\varepsilon(u^n_\varepsilon)$); but, for the implementation of the scheme \textbf{US$\varepsilon$}, this hypothesis ({\bf H}) is not necessary.
\end{obs}
\begin{obs}
Following the ideas of \cite{C5:FMD4}, we can construct another unconditionally energy-stable nonlinear scheme in the variables $(u^n_\varepsilon,{\boldsymbol\sigma}^n_\varepsilon)$ without imposing the right-angled constraint ({\bf H}), replacing the self-diffusion term $(\nabla u^n_\varepsilon,\nabla \bar{u})$ by $\nabla \cdot (\frac{1}{F''_\varepsilon(u^n_\varepsilon)}\nabla \Pi^h (F'_\varepsilon(u^n_\varepsilon)))$. However, this scheme has convergence problems for the linear iterative method as $p\rightarrow 1$ and $\varepsilon\rightarrow 0$.
\end{obs}

Once the scheme \textbf{US$\varepsilon$} is solved, given $v^{n-1}_\varepsilon\in V_h$, we can recover $v^n_\varepsilon=v^n_\varepsilon(u^n_\varepsilon) \in V_h$ solving: 
\begin{equation}\label{C5:edovfreg}
(\delta_t v^n_\varepsilon,\bar{v})  +(\nabla v^n_\varepsilon,\nabla \bar{v}) + (v^n_\varepsilon,\bar{v}) =p(p-1)(F_\varepsilon(u^n_\varepsilon),\bar{v}), \ \ \forall
\bar{v}\in V_h.
\end{equation}

Given $u^n_\varepsilon\in U_h$ and $v^{n-1}_\varepsilon\in V_h$, Lax-Milgram theorem implies that there exists a unique $v^n_\varepsilon \in V_h$ solution of (\ref{C5:edovfreg}). Moreover, notice that the result concerning to the positivity of $v^n_\varepsilon$ solution of scheme \textbf{UV$\varepsilon$} established in Remark \ref{C5:posVuv} remains true for $v^n_\varepsilon$ in the scheme \textbf{US$\varepsilon$}.

\subsubsection{Mass-conservation and Energy-stability}\label{C5:ELusreg}
Observe that the scheme \textbf{US$\varepsilon$}  is also conservative in $u$ (satisfying (\ref{C5:conu1regUV})), and we have the following behavior for $\int_\Omega v^n_\varepsilon$: 
$$
\delta_t \left(\int_\Omega v^n_\varepsilon \right)= p(p-1) \int_\Omega F_\varepsilon(u^n_\varepsilon)  - \int_\Omega v^n_\varepsilon.
$$

\begin{defi}\label{C5:enesfUSreg}
A numerical scheme with solution $(u^n_\varepsilon,{\boldsymbol\sigma}^n_\varepsilon)$ is called energy-stable with respect to the  energy 
\begin{equation}\label{C5:ENusreg}
\mathcal{E}_\varepsilon^h(u,{\boldsymbol\sigma})=p(F_\varepsilon(u),1)^h + \frac{1}{2} \Vert {\boldsymbol\sigma}\Vert_0^2
\end{equation}
if this energy is time decreasing, that is $\mathcal{E}_\varepsilon^h(u^n_\varepsilon,{\boldsymbol\sigma}^n_\varepsilon)\leq \mathcal{E}_\varepsilon^h(u^{n-1}_\varepsilon,{\boldsymbol\sigma}^{n-1}_\varepsilon)$ for all $n\geq 1$.
\end{defi}

\begin{tma} {\bf (Unconditional stability)} \label{C5:estinc1usreg}
The scheme \textbf{US$\varepsilon$} is unconditionally energy stable with respect to $\mathcal{E}_\varepsilon^h(u,{\boldsymbol\sigma})$. In fact, if $(u^n_\varepsilon,{\boldsymbol\sigma}^n_\varepsilon)$ is a solution of \textbf{US$\varepsilon$}, then the following discrete energy law holds
\begin{equation}\label{C5:delusreg}
\delta_t \mathcal{E}_\varepsilon^h(u^n_\varepsilon,{\boldsymbol\sigma}^n_\varepsilon) + \frac{k \varepsilon^{2-p}p}{2}\Vert \delta_t u^n_\varepsilon \Vert_0^2 + \frac{k}{2} \Vert \delta_t  {\boldsymbol\sigma}^n_\varepsilon\Vert_0^2 +  p\varepsilon^{2-p}\Vert \nabla u^n_\varepsilon\Vert_0^2 +\Vert  {\boldsymbol\sigma}^n_\varepsilon\Vert_1^2 \leq 0.
\end{equation}
\end{tma}
\begin{proof}
Testing (\ref{C5:modelf02aUSreg})$_1$ by $\bar{u}= p\Pi^h (F'_\varepsilon (u^n_\varepsilon))$, (\ref{C5:modelf02aUSreg})$_2$ by $\bar{\boldsymbol\sigma}={\boldsymbol\sigma}^n_\varepsilon$ and adding, the terms\\ $p(u^{n}_\varepsilon \nabla  \Pi^h(F'_\varepsilon(u_\varepsilon)),{\boldsymbol \sigma}^n_\varepsilon)$ cancel, and using that $\nabla \Pi^h (F'_\varepsilon (u^n_\varepsilon))=\Lambda_\varepsilon^1 (u^{n}_\varepsilon)^{-1} \nabla u^n_\varepsilon$, we arrive at
\begin{equation*}
p(\delta_t u^n_\varepsilon, F'_\varepsilon (u^n_\varepsilon))^h + p\int_\Omega (\nabla u^n_\varepsilon)^T\!\cdot\!\Lambda_\varepsilon^1 (u^{n}_\varepsilon)^{-1}\!\cdot\! \nabla u^n_\varepsilon d\x  +
\delta_t \Big(\frac{1}{2} \Vert {\boldsymbol\sigma}^n_\varepsilon\Vert_0^2\Big) + \frac{k}{2} \Vert \delta_t  {\boldsymbol\sigma}^n_\varepsilon\Vert_0^2 +\Vert  {\boldsymbol\sigma}^n_\varepsilon\Vert_1^2 = 0,
\end{equation*}
which, proceeding as in (\ref{C5:I01breguv})-(\ref{C5:I01creguv}) and using Remark \ref{C5:eqh25uv} and estimate (\ref{C5:D5uv}), implies (\ref{C5:delusreg}).
\end{proof}

\begin{cor} \label{C5:welemUSreg} {\bf(Uniform estimates) }
Assume that $(u_0,v_0)\in L^2(\Omega)\times H^1(\Omega)$. Let $(u^n_\varepsilon,{\boldsymbol\sigma}^n_\varepsilon)$ be a solution of scheme \textbf{US$\varepsilon$}. Then, it holds 
\begin{equation}\label{C5:weak01uv-aUSreg}
p(F_\varepsilon(u^n_\varepsilon),1)^h + \frac{1}{2}\Vert {\boldsymbol\sigma}^n_\varepsilon\Vert_{0}^{2}+k \underset{m=1}{\overset{n}{\sum}}\left( p\varepsilon^{2-p}\Vert \nabla u^m_\varepsilon\Vert_0^2 +\Vert  {\boldsymbol\sigma}^m_\varepsilon\Vert_1^2\right)\leq \frac{C_0}{(p-1)^2}, \ \ \forall n\geq 1,
\end{equation}
with the constant $C_0>0$ depending on the data $(\Omega, u_0, v_0)$, but independent of $k,h,n$ and $\varepsilon$; and the estimates given in (\ref{C5:pneg1-aUVreg}) also hold.
\end{cor}
\begin{obs}\label{C5:NNuh1}{\bf (Approximated positivity of $u^n_\varepsilon$)}
The approximated positivity result for $u^n_\varepsilon$ established in Remark \ref{C5:NNuh1uv} remains true for the scheme \textbf{US$\varepsilon$}.
\end{obs}
\begin{proof}
Proceeding as in (\ref{C5:ee55uv}) (using the fact that $(u^{0},{\boldsymbol \sigma}^{0})=(Q^h u_0, \widetilde{Q}^h (\nabla v_0))$), we can deduce that
\begin{equation}\label{C5:ee55}
p\int_\Omega \Pi^h (F_\varepsilon(u^0)) +  \frac{1}{2}\Vert  {\boldsymbol\sigma}^0\Vert_{0}^{2} \leq  \frac{C_0}{(p-1)^2},
\end{equation}
where the constant $C_0>0$ depends on the data $(\Omega, u_0, v_0)$, but is independent of $k,h,n$ and $\varepsilon$. Therefore,  from the discrete energy law (\ref{C5:delusreg}) and estimate (\ref{C5:ee55}), we have
$$
\mathcal{E}_\varepsilon^h(u^n_\varepsilon, {\boldsymbol\sigma}^n_\varepsilon)+k \underset{m=1}{\overset{n}{\sum}}\left( p\varepsilon^{2-p}\Vert\nabla u^m_\varepsilon\Vert_0^2 +\Vert {\boldsymbol\sigma}^m_\varepsilon\Vert_1^2\right) \leq  \mathcal{E}_\varepsilon^h(u^0, {\boldsymbol\sigma}^0)\leq \frac{C_0}{(p-1)^2}, \ \ \ \ \
$$
which implies (\ref{C5:weak01uv-aUSreg}). Finally,  the estimates given in (\ref{C5:pneg1-aUVreg}) are proved as in Corollary \ref{C5:welemUVreg}.
\end{proof}

\subsubsection{Well-posedness}
The following two results are concerning to the well-posedness of the scheme \textbf{US$\varepsilon$}.
\begin{tma} {\bf (Unconditional existence)} 
There exists at least one solution $(u^n_\varepsilon,{\boldsymbol\sigma}^n_\varepsilon)$  of scheme \textbf{US$\varepsilon$}.
\end{tma}
\begin{proof}
The proof follows as in Theorem 4.5 of \cite{C5:FMD4}, by using the Leray-Schauder fixed point theorem. 
\end{proof}

\begin{lem}{\bf (Conditional uniqueness)}
If $k\, f(h,\varepsilon)<1$ (where $f(h,\varepsilon)\uparrow +\infty$ when $h\downarrow 0$ or $\varepsilon\downarrow 0$), then the solution $(u^n_\varepsilon,{\boldsymbol \sigma}_\varepsilon^n)$ of the scheme \textbf{US$\varepsilon$} is unique.
\end{lem}
\begin{proof}
The proof follows as in Lemma 4.6 of \cite{C5:FMD4}.
\end{proof}

\subsection{Scheme US0}
In this section, we are going to study another unconditionally energy-stable fully discrete scheme associated to model (\ref{C5:modelf00}). With this aim, we consider the following reformulation of problem (\ref{C5:modelf00}): Find $u:\Omega\times [0,T]\rightarrow \mathbb{R}$ and ${\boldsymbol\sigma}:\Omega\times [0,T]\rightarrow \mathbb{R}^d$, with $u\geq 0$, such that
\begin{equation} \left\{
\begin{array}
[c]{lll}%
\partial_t u - \Delta u - \nabla\cdot (u{\boldsymbol \sigma})=0\ \ \mbox{in}\ \Omega,\ t>0,\\
\partial_t {\boldsymbol \sigma} + \mbox{rot(rot } {\boldsymbol \sigma} \mbox{)} -\nabla(\nabla \cdot {\boldsymbol \sigma}) + {\boldsymbol \sigma} = \nabla(u^p) \ \mbox{in}\ \ \Omega,\ t>0,\\
\displaystyle\frac{\partial u}{\partial \mathbf{n}}=0,\ \ {\boldsymbol \sigma}\cdot \mathbf{n}=0, \ \ \left[\mbox{rot }{\boldsymbol \sigma} \times \mathbf{n}\right]_{tang}=0 \quad \mbox{on}\ \partial\Omega,\ t>0,\\
u(\x ,0)=u_0(\x )\geq 0,\ {\boldsymbol \sigma}(\x ,0)=\nabla v_0(
\x ),\quad \mbox{in}\ \Omega.
\end{array}
\right.  \label{C5:modelf01}
\end{equation}
Once system (\ref{C5:modelf01}) is solved, we can recover $v$ from $u$ by solving
\begin{equation} \left\{
\begin{array}
[c]{lll}%
\partial_t v -\Delta v + v = u^{p}  \ \mbox{in}\ \ \Omega,\ t>0,\\
\frac{\partial v}{\partial \mathbf{n}}=0\ \ \mbox{on}\ \partial\Omega,\ t>0,\\
 v(\x,0)=v_0(\x)>0\ \ \mbox{in}\ \Omega.
\end{array}
\right.  \label{C5:modelf01eqv}
\end{equation}
Observe that (formally) multiplying (\ref{C5:modelf01})$_1$ by $\frac{p}{p-1} u^{p-1}$, (\ref{C5:modelf01})$_2$ by ${\boldsymbol\sigma}$, integrating over $\Omega$ and adding both equations, the terms $\frac{p}{p-1} (u{\boldsymbol \sigma},\nabla(u^{p-1}))$ and $(\nabla(u^p),{\boldsymbol \sigma})$ vanish, we obtain the following energy law
$$
\frac{d}{dt} \displaystyle \int_\Omega \Big( \frac{1}{p-1} \vert u\vert^p + \frac{1}{2} \vert {\boldsymbol\sigma}\vert^2\Big) d \x + \int_\Omega \frac{4}{p} \vert \nabla (u^{p/2}) \vert^2 d\x+\Vert {\boldsymbol\sigma}\Vert_1^2 = 0.
$$
In particular, the modified energy $\mathcal{E}(u,{\boldsymbol\sigma})=\displaystyle \int_\Omega \Big( \frac{1}{p-1} \vert u\vert^p + \frac{1}{2} \vert {\boldsymbol\sigma}\vert^2\Big) d \x$ is decreasing in time. Then, taking into account the reformulation (\ref{C5:modelf01})-(\ref{C5:modelf01eqv}), we consider a fully discrete approximation using a FE discretization in space and the backward
Euler discretization in time (considered for simplicity on a uniform partition of $[0,T]$ with time step
$k=T/N : (t_n = nk)_{n=0}^{n=N}$). Concerning the space discretization, we consider the triangulation as in the scheme \textbf{UV$\varepsilon$}, but in this case without imposing the constraint ({\bf H}) related with the right-angles simplices. We choose the following continuous FE spaces for $u$, ${\boldsymbol\sigma}$ and $v$:
$$(U_h, {\boldsymbol\Sigma}_h, V_h) \subset H^1(\Omega)^3,\quad \hbox{generated by $\mathbb{P}_1,\mathbb{P}_m,\mathbb{P}_{r}$ with $m,r\geq 1$.}
$$
Then, we consider the following first order in time, nonlinear and coupled scheme: 
\begin{itemize}
\item{\underline{\emph{Scheme \textbf{US0}}:}\\
{\bf Initialization}:  Let $(u^0,{\boldsymbol \sigma}^{0})=(Q^h u_0, \widetilde{Q}^h (\nabla v_0))\in  U_h\times {\boldsymbol\Sigma}_h$.\\
{\bf Time step} n: Given $(u^{n-1},{\boldsymbol \sigma}^{n-1})\in  U_h\times {\boldsymbol\Sigma}_h$, compute $(u^{n},{\boldsymbol \sigma}^n)\in  U_h \times {\boldsymbol\Sigma}_h$ solving
\begin{equation}
\left\{
\begin{array}
[c]{lll}%
(\delta_t u^n,\bar{u})^h +\frac{1}{p-1}
((u^{n}_+)^{2-p}\nabla(\Pi^h ((u^n_+)^{p-1})),\nabla\bar{u}) = -(u^{n}{\boldsymbol \sigma}^n,\nabla \bar{u}), \ \ \forall \bar{u}\in U_h,\\
(\delta_t {\boldsymbol \sigma}^n,\bar{\boldsymbol \sigma}) +
( B_h {\boldsymbol \sigma}^n,\bar{\boldsymbol
\sigma}) =
\frac{p}{p-1}(u^{n}\nabla (\Pi^h ((u^n_+)^{p-1})),\bar{\boldsymbol \sigma}),\ \ \forall
\bar{\boldsymbol \sigma}\in \Sigma_h,
\end{array}
\right.  \label{C5:modelf02aUS}
\end{equation}}
\end{itemize}
where $u^n_+:= \max\{ u^n, 0\}\geq 0$. Recall that $Q^h$ is the $L^2$-projection on $U_h$ defined in (\ref{C5:MLP25uv}), $\widetilde{Q}^h$ is the standard $L^2$-projection on ${\boldsymbol\Sigma}_h$, $\Pi^h: C(\overline{\Omega})\rightarrow U_h$ is the Lagrange interpolation operator, $(B_h {\boldsymbol \sigma}^n,\bar{\boldsymbol \sigma}) = (\mbox{rot } {\boldsymbol\sigma}^n,\mbox{rot } \bar{\boldsymbol\sigma}) + (\nabla \cdot {\boldsymbol\sigma}^n, \nabla \cdot \bar{\boldsymbol\sigma}) + ({\boldsymbol\sigma}^n, \bar{\boldsymbol\sigma})$ and the discrete semi-inner product $(\cdot,\cdot)^h$ was defined in (\ref{C5:mlumpreguv}). Once the scheme \textbf{US0} is solved, given $v^{n-1}\in V_h$, we can recover $v^n=v^n(u^n) \in V_h$ solving: 
\begin{equation}\label{C5:edovfus}
(\delta_t v^n,\bar{v})  +( \nabla v^n,\nabla\bar{v}) + (v^n,\bar{v})  =((u^n_+)^{p},\bar{v}), \ \ \forall
\bar{v}\in V_h.
\end{equation}

Given $u^n\in U_h$ and $v^{n-1}\in V_h$, Lax-Milgram theorem implies that there exists a unique $v^n \in V_h$ solution of (\ref{C5:edovfus}).

\begin{obs}{\bf (Positivity of $v^n$)} 
Imposing the geometrical property of the triangulation where the interior angles of the triangles or tetrahedra must be at most $\pi/2$, the result concerning to the positivity of $v^n$ stablished in Remark \ref{C5:posVuv} remains true for the scheme \textbf{US0}.
\end{obs}

\subsubsection{Mass-conservation, Energy-stability and Solvability}\label{C5:ELus}
Since $\bar{u}=1\in U_h$ and $\bar{v}=1 \in V_h$, we deduce that the scheme \textbf{US0} is conservative in $u^n$, that is,
\begin{equation}\label{C5:conu1us}
(u^n,1)=(u^n,1)^h= (u^{n-1},1)^h=\cdot\cdot\cdot= (u^{0},1)^h=(u_0,1)=m_0,
\end{equation}
and we have the following behavior for $\int_\Omega v^n$: 
$$
\delta_t \left(\int_\Omega v^n \right)= \int_\Omega (u^n_+)^{p}  - \int_\Omega v^n.
$$

\begin{defi}\label{C5:enesfUS}
A numerical scheme with solution $(u^n,{\boldsymbol\sigma}^n)$ is called energy-stable with respect to the  energy 
\begin{equation}\label{C5:ENus}
\mathcal{E}^h(u,{\boldsymbol\sigma})= \frac{1}{p-1}( (u_+)^p,1)^h+ \frac{1}{2} \Vert {\boldsymbol\sigma}\Vert_0^2,
\end{equation}
if this energy is time decreasing, that is $\mathcal{E}^h(u^n,{\boldsymbol\sigma}^n)\leq \mathcal{E}^h(u^{n-1},{\boldsymbol\sigma}^{n-1})$, for all $n\geq 1$.
\end{defi}

\begin{tma} {\bf (Unconditional stability)} \label{C5:estinc1us}
The scheme \textbf{US0} is unconditionally energy stable with respect to $\mathcal{E}^h(u,{\boldsymbol\sigma})$. In fact, if $(u^n,{\boldsymbol\sigma}^n)$ is a solution of \textbf{US0}, then the following discrete energy law holds
\begin{equation}\label{C5:delus}
\delta_t \mathcal{E}^h(u^n,{\boldsymbol\sigma}^n)  + \frac{k}{2} \Vert \delta_t  {\boldsymbol\sigma}^n\Vert_0^2 +\frac{p}{(p-1)^2}\int_\Omega (u^{n}_+)^{2-p}\vert \nabla (\Pi^h ((u^n_+)^{p-1})) \vert^{2} d\x +\Vert  {\boldsymbol\sigma}^n\Vert_1^2 \leq 0.
\end{equation}
\end{tma}
\begin{proof}
Testing (\ref{C5:modelf02aUS})$_1$ by $\bar{u}=\frac{p}{p-1} \Pi^h ((u^n_+)^{p-1})$,  (\ref{C5:modelf02aUS})$_2$ by $\bar{\boldsymbol\sigma}={\boldsymbol\sigma}^n$ and  adding, the terms\\
 $\frac{p}{p-1}(u^{n}\nabla (\Pi^h ((u^n_+)^{p-1})),{\boldsymbol \sigma}^n)$ cancel, and we obtain
\begin{eqnarray}\label{C5:modelele01us}
&\displaystyle \frac{p}{p-1}\int_\Omega  \Pi^h(\delta_t u^n \cdot (u^n_+)^{p-1}) d\x &\!\!\!\! + \frac{1}{2}\delta_t \Vert {\boldsymbol\sigma}^n \Vert_0^2+ \frac{k}{2}\Vert \delta_t {\boldsymbol\sigma}^n\Vert_0^2 \nonumber\\
&&\!\!\!\!\!\!\!+ \frac{p}{(p-1)^2}\int_\Omega
(u^{n}_+)^{2-p}\vert \nabla (\Pi^h ((u^n_+)^{p-1})) \vert^{2} d\x+ \Vert{\boldsymbol\sigma}^n \Vert_1^2 =0. 
\end{eqnarray}
Denoting by $F(u^n)=\displaystyle\frac{1}{p}(u^n_+)^{p}$, we have that $F$ is a differentiable and convex function, and then, from (\ref{C5:Gdif2}) we have that
\begin{equation*}
\delta_t u^n \cdot (u^n_+)^{p-1} =
\frac{1}{k} F'(u^n) (
u^n - u^{n-1})  \geq  \frac{1}{k} (F(u^n)-F(u^{n-1})) =\delta_t
 F(u^n) ,
\end{equation*}
and therefore, 
\begin{equation}\label{C5:modelele02us}
\displaystyle\int_\Omega \Pi^h(\delta_t u^n \cdot (u^n_+)^{p-1}) \geq  \delta_t
\left(\int_\Omega \Pi^h F(u^n) \right) = \frac{1}{p}\delta_t
\left(\int_\Omega \Pi^h ((u^n_+)^{p}) \right).
\end{equation}
Therefore, from (\ref{C5:modelele01us}) and (\ref{C5:modelele02us}) we deduce (\ref{C5:delus}).
\end{proof}

\begin{cor} \label{C5:welemUS} {\bf(Uniform estimates) }
Let $(u^n,{\boldsymbol\sigma}^n)$ be a solution of scheme \textbf{US0}. Then, it holds for all $n\geq 1$,
\begin{equation}\label{C5:weak01uv-aUS}
\frac{1}{p-1}( (u^n_+)^p,1)^h+ \frac{1}{2} \Vert {\boldsymbol\sigma}^n\Vert_0^2+k \underset{m=1}{\overset{n}{\sum}}\left(\frac{p}{(p-1)^2}\int_\Omega
(u^{m}_+)^{2-p}\vert \nabla (\Pi^h ((u^m_+)^{p-1})) \vert^{2} d\x +\Vert  {\boldsymbol\sigma}^m\Vert_1^2\right)\leq \frac{C_0}{p-1}, 
\end{equation}
\begin{equation}\label{C5:esul1}
\int_\Omega \vert u^n\vert \leq C_1,  
\end{equation}
with the constants $C_0,C_1>0$ depending on the data $(\Omega,u_0, v_0)$, but independent of $(k,h)$ and $n$. 
\end{cor}
\begin{proof}
In order to obtain (\ref{C5:weak01uv-aUS}), by multiplying (\ref{C5:delus}) by $k$ and adding from $m=1$ to $m=n$, it suffices to bound the initial energy $\mathcal{E}^h(u^0,{\boldsymbol\sigma}^0)$. Taking into account that $(u^0,{\boldsymbol \sigma}^{0})=(Q^h u_0, \widetilde{Q}^h (\nabla v_0))$ and $u_0\geq 0$ (and therefore, $u^0\geq 0$), we have
\begin{equation*}
\mathcal{E}^h(u^0,{\boldsymbol\sigma}^0) \leq \frac{C}{p-1}\int_\Omega \Pi^h ((u^0)^2 + 1)+ \frac{1}{2} \Vert v_0\Vert_1^2 \leq \frac{C}{p-1} (\Vert u_0\Vert_0^2+ \Vert v_0\Vert_1^2 + 1).
\end{equation*}
On the other hand, by considering $u^n_-=\min\{u^n, 0 \}\geq 0$, taking into account that $\vert u^n\vert = 2 u^n_{+} - u^n$, using the H\"older and Young inequalities as well as (\ref{C5:conu1us}), we have
\begin{eqnarray}\label{C5:esup}
&\displaystyle\int_\Omega \vert u^n\vert&\!\!\!\! \leq \int_\Omega \Pi^h \vert u^n\vert = 2 \int_\Omega \Pi^h (u^n_+) -  \int_\Omega u^n\nonumber\\
&& \leq  C\Big(\int_\Omega (\Pi^h (u^n_+))^p +1\Big)\leq  C\Big(\int_\Omega \Pi^h ((u^n_+)^p) + 1\Big).
\end{eqnarray}
Therefore, from (\ref{C5:weak01uv-aUS}) and (\ref{C5:esup}), we deduce (\ref{C5:esul1}).
\end{proof}

\begin{tma} {\bf (Unconditional existence)} 
There exists at least one solution $(u^n,{\boldsymbol\sigma}^n)$  of scheme \textbf{US0}.
\end{tma}
\begin{proof}
The proof follows as in Theorem 4.5 of \cite{C5:FMD4}, by using the Leray-Schauder fixed point theorem. 
\end{proof}

\section{Numerical simulations}\label{C5:S5C5:NSi}
In this section, we will compare the results of several numerical simulations using the schemes derived through the paper. We have chosen the 2D domain $[0,2]^2$ using a structured mesh (then the right-angled constraint ({\bf H}) holds and the scheme \textbf{UV$\varepsilon$} can be defined), the spaces for $u$ and ${\boldsymbol\sigma}$ have been generated by $\mathbb{P}_1$-continuous FE, and all the simulations have been carried out using \textbf{FreeFem++} software. We will also compare with the usual Backward Euler scheme for problem (\ref{C5:modelf00}), which is given for the following first order in time, nonlinear and coupled scheme: 
\begin{itemize}
\item{\underline{\emph{Scheme \textbf{UV}}:}\\
{\bf Initialization}: Let $(u^{0},{v}^{0})\in  U_h\times V_h$ an approximation of $(u_0,v_0)$ as $h\rightarrow 0$. \\
{\bf Time step} n: Given $(u^{n-1},{v}^{n-1})\in  U_h\times {V}_h$, compute $(u^n,{v}^{n})\in U_h \times {V}_h$  by solving
\begin{equation*}
\left\{
\begin{array}
[c]{lll}%
(\delta_t u^n,\bar{u}) + (\nabla u^n,\nabla \bar{u}) = -(u^{n}\nabla {v}^n,\nabla \bar{u}), \ \ \forall \bar{u}\in U_h,\\
(\delta_t {v}^n,\bar{v}) +(\nabla  v^n, \nabla\bar{v}) + (v^n,\bar{v})  = ((u^n_+)^{p} ,\bar{v}),\ \ \forall
\bar{v}\in V_h.
\end{array}
\right.  
\end{equation*}}
\end{itemize}
\begin{obs}\label{C5:RBE5}
The scheme \textbf{UV} has not been analyzed in the previous sections because it is not clear how to prove its energy-stability. In fact, observe that the scheme \textbf{UV$\varepsilon$} (which is the ``closest'' approximation to the scheme \textbf{UV} considered in this paper) differs from the scheme \textbf{UV} in the use of the regularized functions $F_\varepsilon$ and its derivatives (see Figure \ref{C5:fig:Fe}) and in the approximation of the cross-diffusion and production terms, $(u\nabla {v},\nabla \bar{u})$ and $(u^p,\bar{v})$ respectively, which are crucial for the proof of the energy-stability of the scheme \textbf{UV$\varepsilon$}. 
\end{obs}
The linear iterative methods used to approach the solutions of the nonlinear schemes \textbf{UV$\varepsilon$}, \textbf{US$\varepsilon$}, \textbf{US0} and \textbf{UV} are the following Picard methods:
\begin{enumerate}
\item[(i)]{Picard method to approach a solution $(u^{n}_\varepsilon,{v}^{n}_\varepsilon)$ of the scheme \textbf{UV$\varepsilon$}}:\\
{\bf Initialization ($l=0$):} Set $(u^{0}_\varepsilon,{v}^{0}_\varepsilon)=(u^{n-1}_\varepsilon,{v}^{n-1}_\varepsilon)\in  U_h\times V_h$.\\
{\bf Algorithm:} Given $(u^{l}_\varepsilon,{v}^{l}_\varepsilon)\in  U_h\times V_h$, compute $(u^{l+1}_\varepsilon,{v}^{l+1}_\varepsilon)\in  U_h\times V_h$ such that
$$
\left\{
\begin{array}
[c]{lll}%
\frac{1}{k}(u^{l+1}_\varepsilon,\bar{u})^h + (\nabla u^{l+1}_\varepsilon,\nabla \bar{u}) = \frac{1}{k}(u^{n-1}_\varepsilon,\bar{u})^h -(\Lambda_\varepsilon^2 (u^{l}_\varepsilon)\nabla {v}^{l}_\varepsilon,\nabla \bar{u}), \ \ \forall \bar{u}\in U_h,\\
\frac{1}{k}({v}^{l+1}_\varepsilon,\bar{v}) +(A_h  v^{l+1}_\varepsilon, \bar{v})  = \frac{1}{k}({v}^{n-1}_\varepsilon,\bar{v}) 
+p(p-1)(\Pi^h F_\varepsilon(u^{l+1}_\varepsilon),\bar{v}),\ \ \forall
\bar{v}\in V_h,
\end{array}
\right.  
$$
choosing the stopping criteria $\max\left\{\displaystyle\frac{\Vert
u^{l+1}_\varepsilon - u^{l}_\varepsilon\Vert_{0}}{\Vert u^{l}_\varepsilon\Vert_{0}},\displaystyle\frac{\Vert
v^{l+1}_\varepsilon - v^{l}_\varepsilon\Vert_{0}}{\Vert
v^{l}_\varepsilon\Vert_{0}}\right\}\leq tol$. 

\item[(ii)]{Picard method to approach a solution $(u^n_\varepsilon, {\boldsymbol\sigma}^n_\varepsilon)$ of the scheme \textbf{US$\varepsilon$}:}\\
{\bf Initialization ($l=0$):} Set $(u^{0}_\varepsilon,{\boldsymbol \sigma}_\varepsilon^{0})=(u^{n-1}_\varepsilon,{\boldsymbol \sigma}_\varepsilon^{n-1})\in  U_h\times {\boldsymbol \Sigma}_h$.\\
{\bf Algorithm:} Given $(u^{l}_\varepsilon,{\boldsymbol \sigma}_\varepsilon^{l})\in  U_h\times {\boldsymbol \Sigma}_h$, compute $(u^{l+1}_\varepsilon,{\boldsymbol \sigma}_\varepsilon^{l+1})\in  U_h\times {\boldsymbol \Sigma}_h$ such that
\begin{equation*}
\left\{
\begin{array}
[c]{lll}%
\frac{1}{k}(u^{l+1}_\varepsilon ,\bar{u})^h + (\nabla u^{l+1}_\varepsilon, \nabla \bar{u}) +(u^{l+1}_\varepsilon {\boldsymbol\sigma}^{l}_\varepsilon,\nabla \bar{u})  = \frac{1}{k}(u^{n-1}_\varepsilon ,\bar{u})^h, \ \ \forall \bar{u}\in U_h,\\
\frac{1}{k}({\boldsymbol \sigma}^{l+1}_\varepsilon,\bar{\boldsymbol \sigma}) + (B_h {\boldsymbol \sigma}^{l+1}_\varepsilon,\bar{\boldsymbol \sigma}) =\frac{1}{k}({\boldsymbol \sigma}^{n-1}_\varepsilon,\bar{\boldsymbol \sigma})+
p(u^{l+1}_\varepsilon \nabla  \Pi^h(F'_\varepsilon(u^{l+1}_\varepsilon)),\bar{\boldsymbol \sigma}),\ \ \forall
\bar{\boldsymbol \sigma}\in \Sigma_h,
\end{array}
\right. 
\end{equation*}
choosing the stopping criteria $\max\left\{\displaystyle\frac{\Vert
u^{l+1}_\varepsilon - u^{l}_\varepsilon\Vert_{0}}{\Vert u^{l}_\varepsilon\Vert_{0}},\displaystyle\frac{\Vert
{\boldsymbol \sigma}^{l+1}_\varepsilon - {\boldsymbol \sigma}^{l}_\varepsilon\Vert_{0}}{\Vert
{\boldsymbol \sigma}^{l}_\varepsilon\Vert_{0}}\right\}\leq tol$. 

\item[(iii)]{Picard method to approach a solution $(u^n, {\boldsymbol\sigma}^n)$ the scheme \textbf{US0}:}\\
{\bf Initialization ($l=0$):} Set $(u^{0},{\boldsymbol \sigma}^{0})=(u^{n-1},{\boldsymbol \sigma}^{n-1})\in  U_h\times {\boldsymbol \Sigma}_h$.\\
{\bf Algorithm:} Given $(u^{l},{\boldsymbol \sigma}^{l})\in  U_h\times {\boldsymbol \Sigma}_h$, compute $(u^{l+1},{\boldsymbol \sigma}^{l+1})\in  U_h\times {\boldsymbol \Sigma}_h$ such that
\begin{equation*}
\left\{
\begin{array}
[c]{lll}%
\frac{1}{k}(u^{l+1},\bar{u})^h  + (\nabla u^{l+1}, \nabla \bar{u}) -  (\nabla u^{l}, \nabla \bar{u})+(u^{l+1}{\boldsymbol \sigma}^{l},\nabla \bar{u}) \\
\hspace{2 cm} =\frac{1}{k}(u^{n-1},\bar{u})^h   -\frac{1}{p-1}
((u^{l}_+)^{2-p}\nabla(\Pi^h (u^{l}_+)^{p-1}),\nabla\bar{u}), \ \ \forall \bar{u}\in U_h,\\
\frac{1}{k}({\boldsymbol \sigma}^{l+1},\bar{\boldsymbol \sigma}) +
(B_h {\boldsymbol \sigma}^{l+1},\bar{\boldsymbol
\sigma}) = \frac{1}{k}({\boldsymbol \sigma}^{n-1},\bar{\boldsymbol \sigma}) + 
\frac{p}{p-1}(u^{l+1}\nabla (\Pi^h (u^{l+1}_+)^{p-1}),\bar{\boldsymbol \sigma}),\ \ \forall
\bar{\boldsymbol \sigma}\in \Sigma_h,
\end{array}
\right. 
\end{equation*}
choosing the stopping criteria $\max\left\{\displaystyle\frac{\Vert
u^{l+1} - u^{l}\Vert_{0}}{\Vert u^{l}\Vert_{0}},\displaystyle\frac{\Vert
{\boldsymbol \sigma}^{l+1} - {\boldsymbol \sigma}^{l}\Vert_{0}}{\Vert
{\boldsymbol \sigma}^{l}\Vert_{0}}\right\}\leq tol$. Observe that the resi\-dual term $(\nabla (u^{l+1} - u^{l}), \nabla \bar{u})$ is considered.

\item[(iv)]{Picard method to approach a solution $(u^n,v^n)$ of the scheme \textbf{UV}:}\\
{\bf Initialization ($l=0$):} Set $(u^{0},{v}^{0})=(u^{n-1},{v}^{n-1})\in  U_h\times V_h$.\\
{\bf Algorithm:} Given $(u^{l},{v}^{l})\in  U_h\times V_h$, compute $(u^{l+1},{v}^{l+1})\in  U_h\times V_h$ such that
\begin{equation*}
\left\{
\begin{array}
[c]{lll}%
\frac{1}{k}(u^{l+1},\bar{u}) + (\nabla u^{l+1},\nabla \bar{u}) + (u^{l+1}\nabla {v}^{l},\nabla \bar{u}) = \frac{1}{k}(u^{n-1},\bar{u}), \ \ \forall \bar{u}\in U_h,\\
\frac{1}{k}({v}^{l+1},\bar{v}) +(\nabla  v^{l+1}, \nabla\bar{v}) + (v^{l+1}, \bar{v})  = \frac{1}{k}({v}^{n-1},\bar{v}) + 
((u^{l+1}_+)^p ,\bar{v}),\ \ \forall
\bar{v}\in V_h,
\end{array}
\right.  
\end{equation*}
choosing the stopping criteria $\max\left\{\displaystyle\frac{\Vert
u^{l+1} - u^{l}\Vert_{0}}{\Vert u^{l}\Vert_{0}},\displaystyle\frac{\Vert
v^{l+1} - v^{l}\Vert_{0}}{\Vert
v^{l}\Vert_{0}}\right\}\leq tol$.

\end{enumerate}
\begin{obs}
In all cases, first we compute $u^{l+1}$ solving the $u$-equation, and then, inserting $u^{l+1}$ in the $v$-equation (resp. ${\boldsymbol\sigma}$-system), we compute $v^{l+1}$ (resp. ${\boldsymbol\sigma}^{l+1}$).
\end{obs}

\subsection{Positivity of $u^n$}
In this subsection, we compare the positivity of the variable $u^n$ in the four schemes. Here, we choose the space for $v$  generated by $\mathbb{P}_2$-continuous FE. We recall that for the three schemes studied in this paper, namely schemes \textbf{UV$\varepsilon$}, \textbf{US$\varepsilon$} and \textbf{US0},  the positivity of the variable $u^n$ is not clear. Moreover, for the schemes \textbf{UV$\varepsilon$} and \textbf{US$\varepsilon$}, it was proved that $\Pi^h(u^n_{\varepsilon-})\rightarrow 0$ as $\varepsilon\rightarrow 0$ (see Remarks \ref{C5:NNuh1uv} and \ref{C5:NNuh1}). For this reason, in Figures \ref{C5:fig:MU111}-\ref{C5:fig:MU192} we compare the positivity of the variable $u^n_\varepsilon$ in the schemes, for different values of $p$, $1<p<2$, and taking $\varepsilon=10^{-3}$, $\varepsilon=10^{-5}$ and $\varepsilon=10^{-8}$ in the schemes \textbf{UV$\varepsilon$} and \textbf{US$\varepsilon$}. We consider $k=10^{-5}$, $h=\frac{1}{40}$, the tolerance parameter $tol=10^{-3}$ and the initial conditions (see Figure \ref{C5:fig:initcond5})
$$u_0\!\!=\!\!-10xy(2-x)(2-y)exp(-10(y-1)^2-10(x-1)^2)+10.0001,$$
$$v_0\!\!=\!\!100xy(2-x)(2-y)exp(-30(y-1)^2-30(x-1)^2)+0.0001.$$
\begin{figure}[htbp]
\centering 
\subfigure[Initial cell density $u_0$]{\includegraphics[width=65mm]{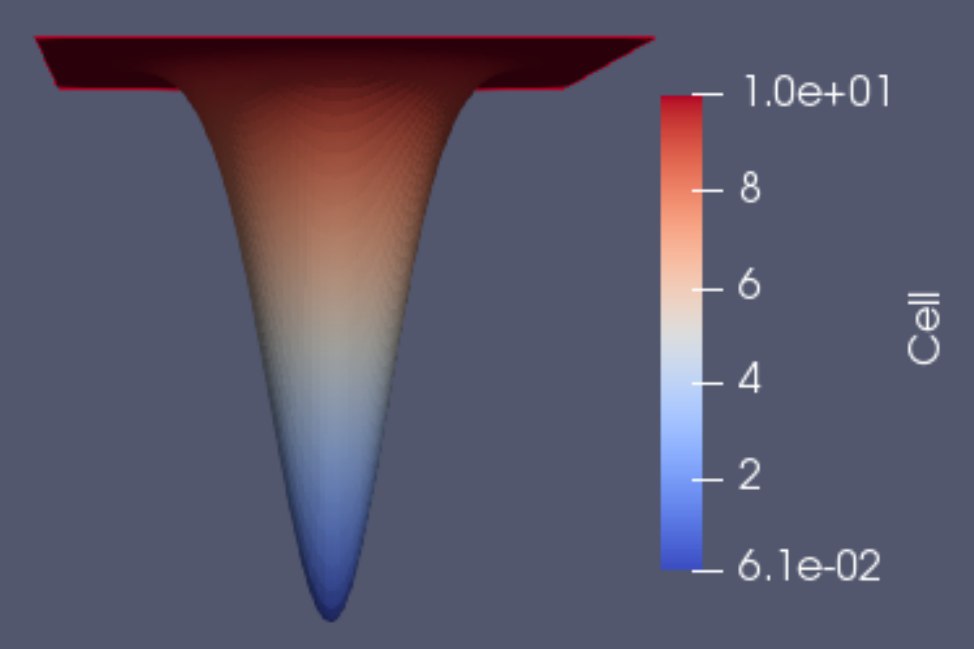}} \hspace{1,2 cm} 
\subfigure[Initial chemical concentration $v_0$]{\includegraphics[width=65mm]{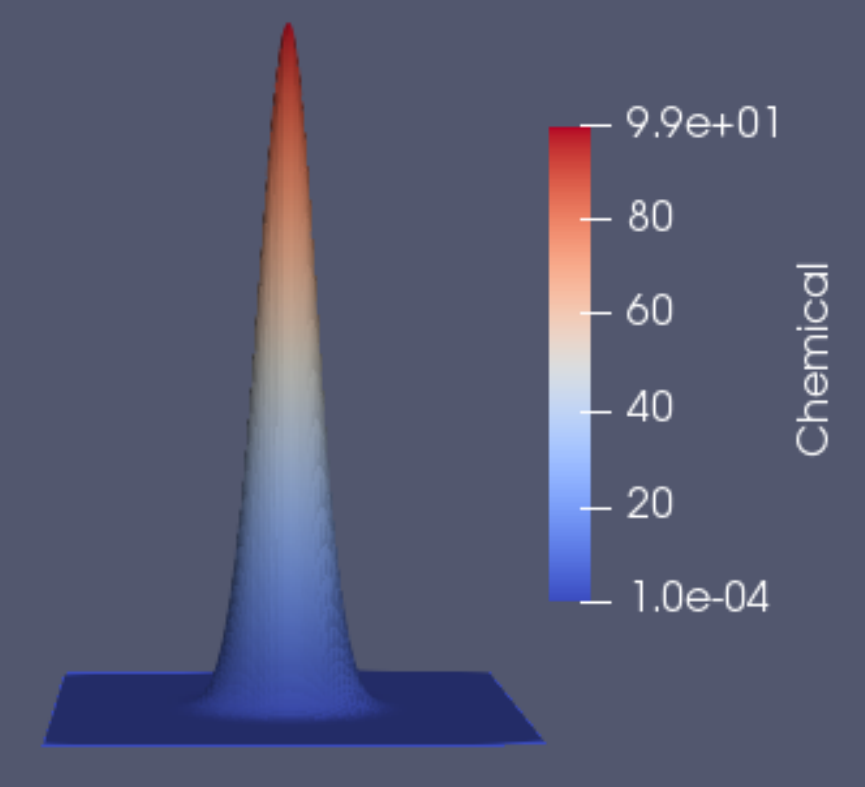}} 
\caption{Initial conditions.} \label{C5:fig:initcond5}
\end{figure}
Note that $u_0,v_0>0$ in $\Omega$, $\min(u_0)=u_0(1,1)=0.0001$ and $\max(v_0)=v_0(1,1)=100.0001$. 
We obtain that:
\begin{enumerate}
\item[(i)] All the schemes take negative values for the minimum of $u^n$ in different times $t_n\geq 0$, for the different values considered for $p$ and $\varepsilon$. However, in the case of the schemes \textbf{UV$\varepsilon$} and \textbf{US$\varepsilon$}, it is observed that these values are closer to $0$ as $\varepsilon\rightarrow 0$ (see Figures \ref{C5:fig:MU111}-\ref{C5:fig:MU192}).
\item[(ii)] In all cases, the scheme \textbf{UV$\varepsilon$} ``preserves'' better the positivity than the schemes \textbf{UV}, \textbf{US$\varepsilon$} and \textbf{US0} (see Figures \ref{C5:fig:MU111}-\ref{C5:fig:MU192}). 
\end{enumerate}
 \begin{figure}[h]
  \begin{center}
 {\includegraphics[height=0.5\linewidth]{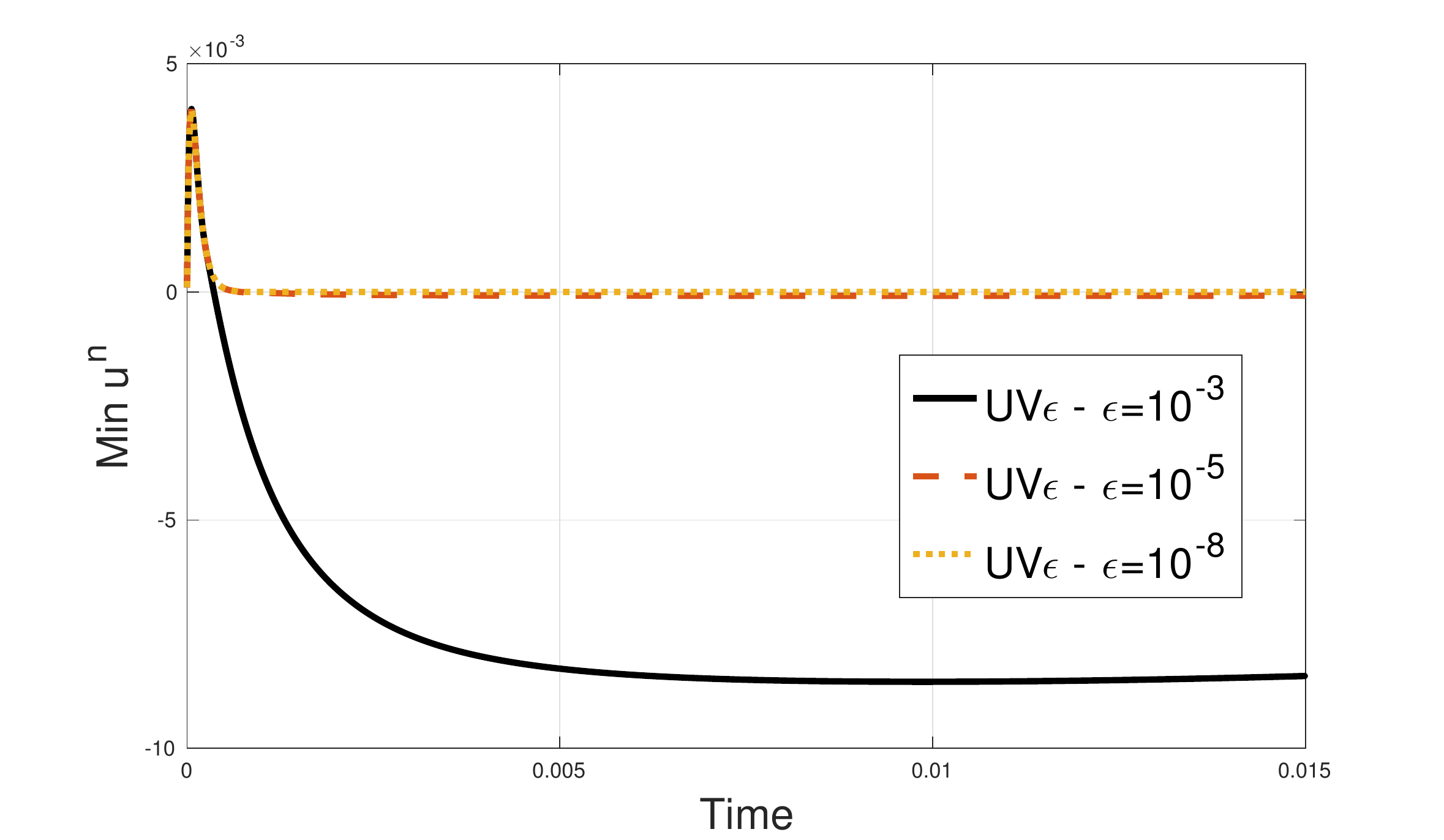}}
 \caption{Minimum values of $u^n_\varepsilon$ for $p=1.1$, computed using the scheme \textbf{UV$\varepsilon$}. We also obtain negative values for $\varepsilon=10^{-8}$ of order $10^{-8}$.
  \label{C5:fig:MU111}}
  \end{center}
\end{figure}
 \begin{figure}[h]
  \begin{center}
 {\includegraphics[height=0.5\linewidth]{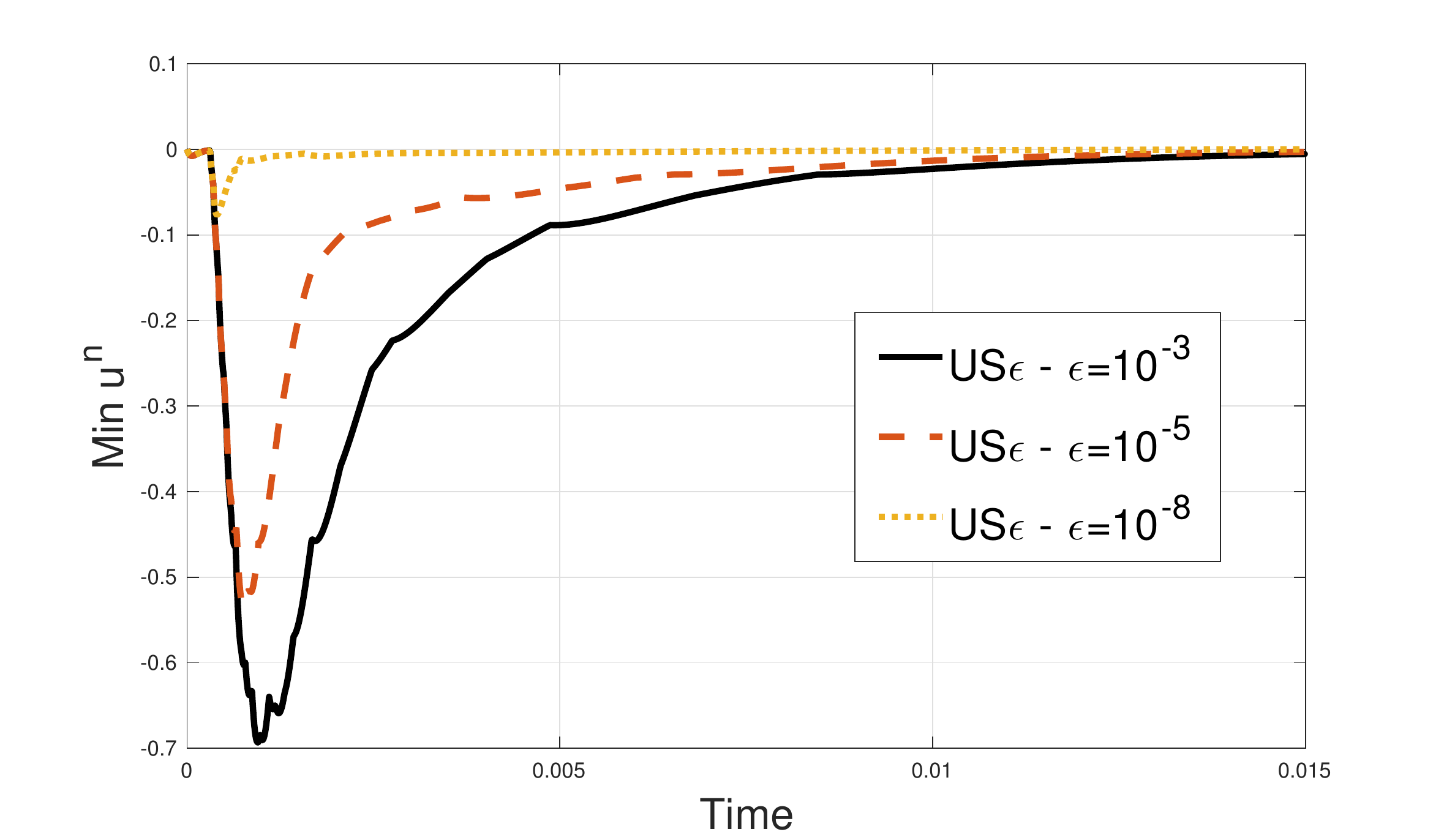}}
 \caption{Minimum values of $u^n_\varepsilon$ for $p=1.1$, computed using the scheme \textbf{US$\varepsilon$}.
  \label{C5:fig:MU112}}
  \end{center}
\end{figure}
 \begin{figure}[h]
  \begin{center}
 {\includegraphics[height=0.5\linewidth]{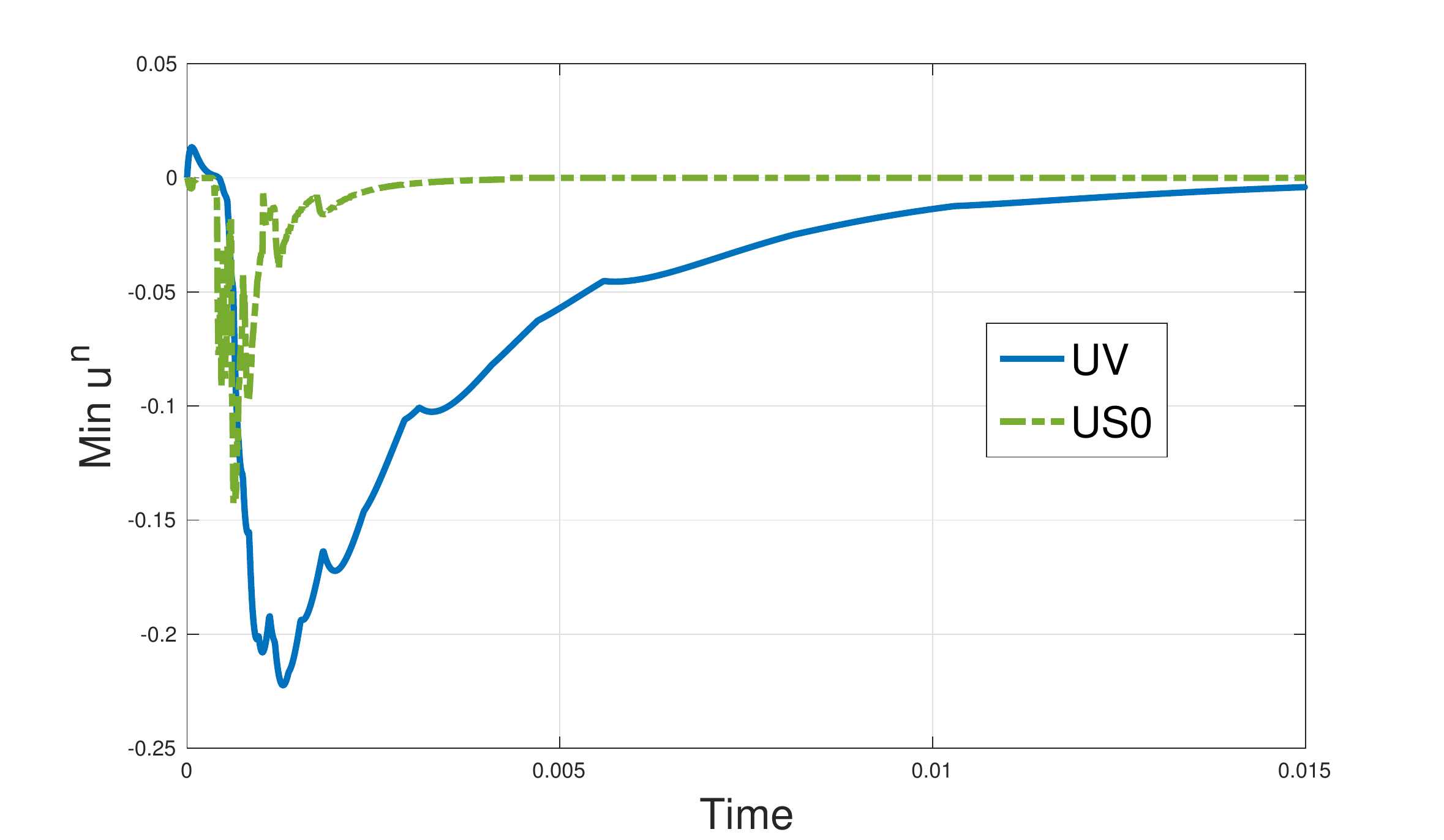}}
 \caption{Minimum values of $u^n_\varepsilon$ for $p=1.1$, computed using the schemes \textbf{UV} and \textbf{US0}.
  \label{C5:fig:MU113}}
  \end{center}
\end{figure}
 \begin{figure}[h]
  \begin{center}
  {\includegraphics[height=0.5\linewidth]{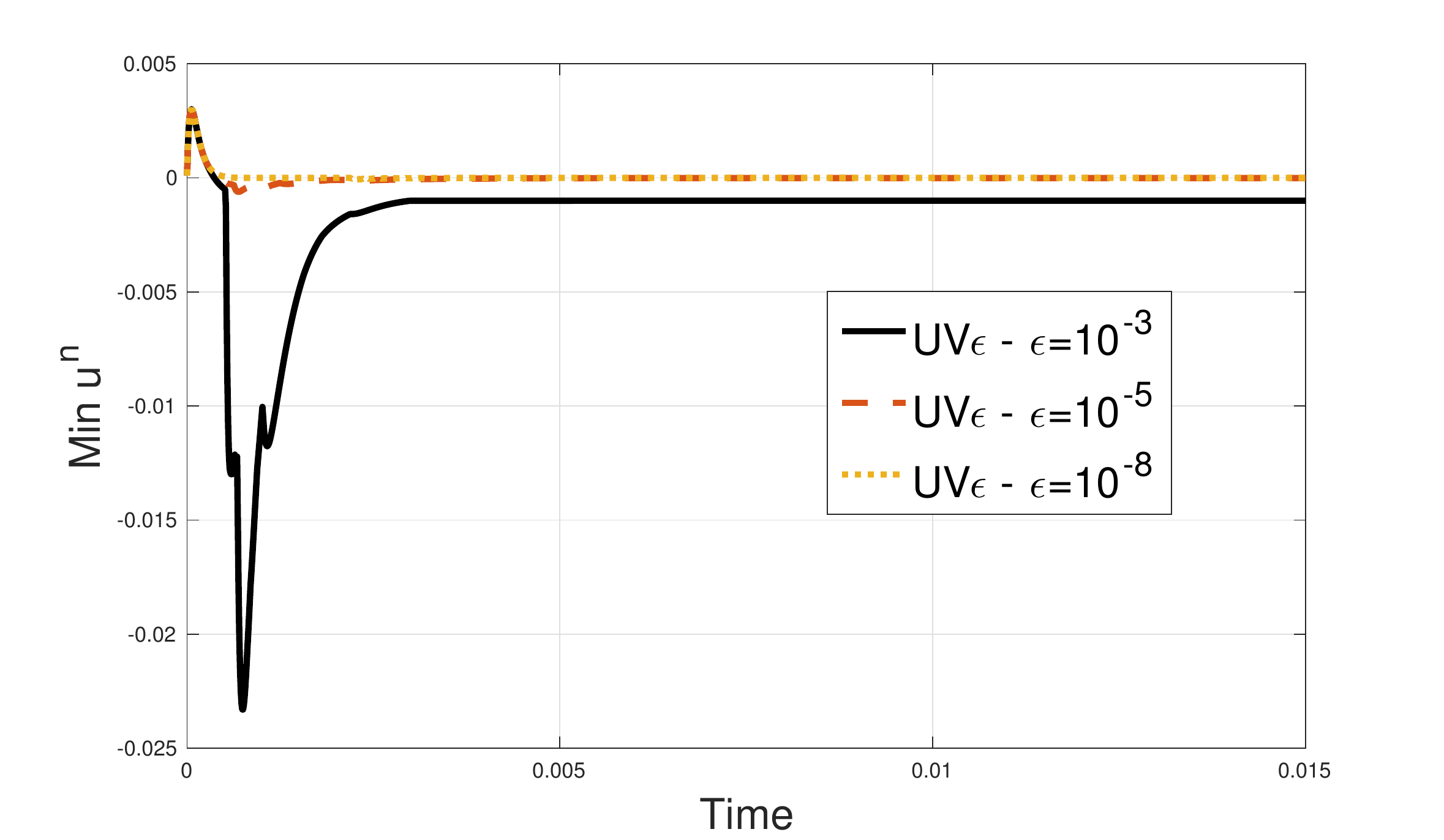}}
 \caption{Minimum values of $u^n_\varepsilon$ for $p=1.5$, computed using the scheme \textbf{UV$\varepsilon$}. We also obtain negative values for $\varepsilon=10^{-8}$ of order $10^{-5}$.
  \label{C5:fig:MU151}}
  \end{center}
\end{figure}
 \begin{figure}[h]
  \begin{center}
  {\includegraphics[height=0.5\linewidth]{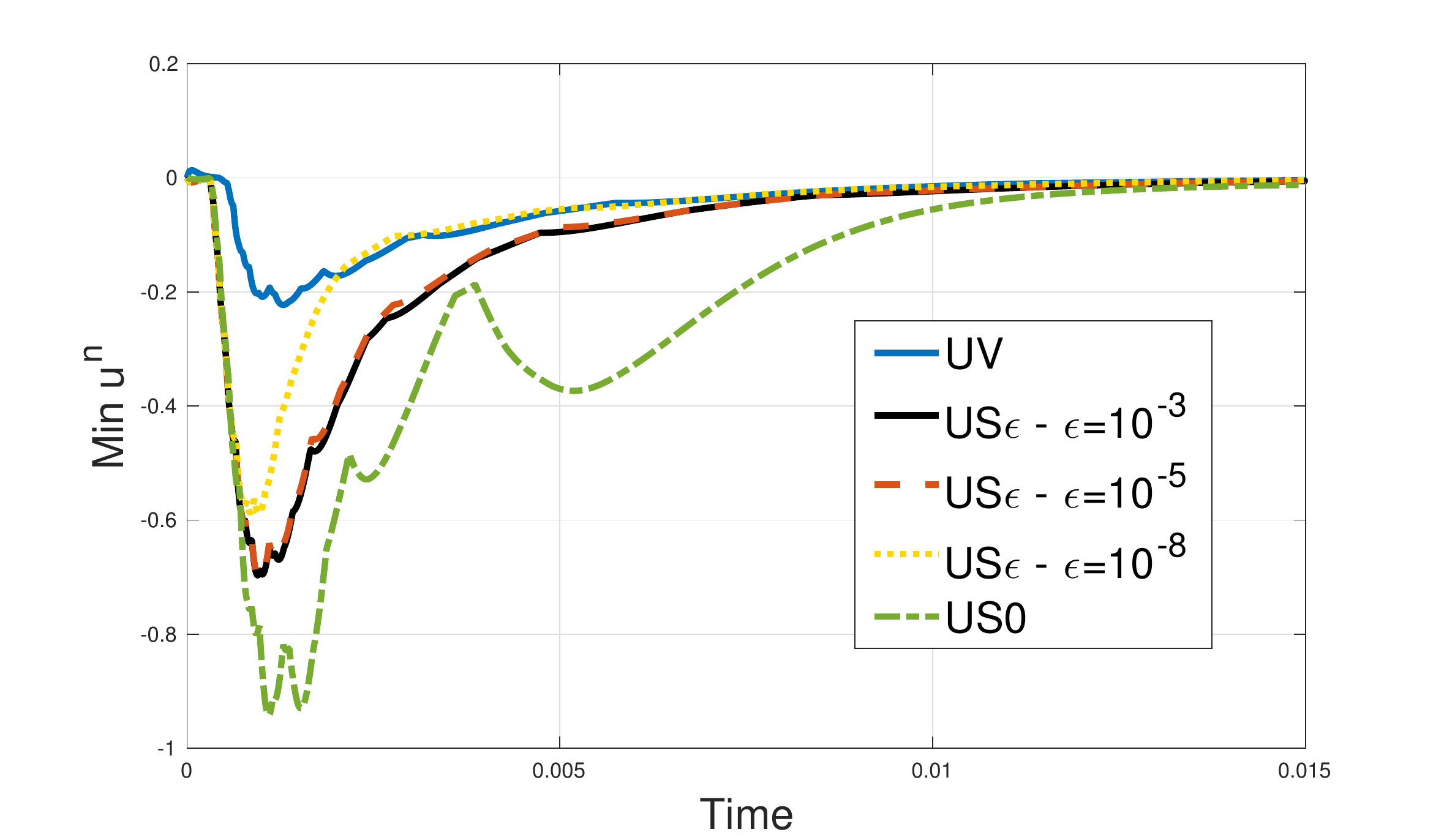}}
  \caption{Minimum values of $u^n$ for $p=1.5$, computed using the schemes \textbf{UV}, \textbf{US$\varepsilon$} and \textbf{US0}.
  \label{C5:fig:MU152}}
  \end{center}
\end{figure}
 \begin{figure}[h]
  \begin{center}
  {\includegraphics[height=0.5\linewidth]{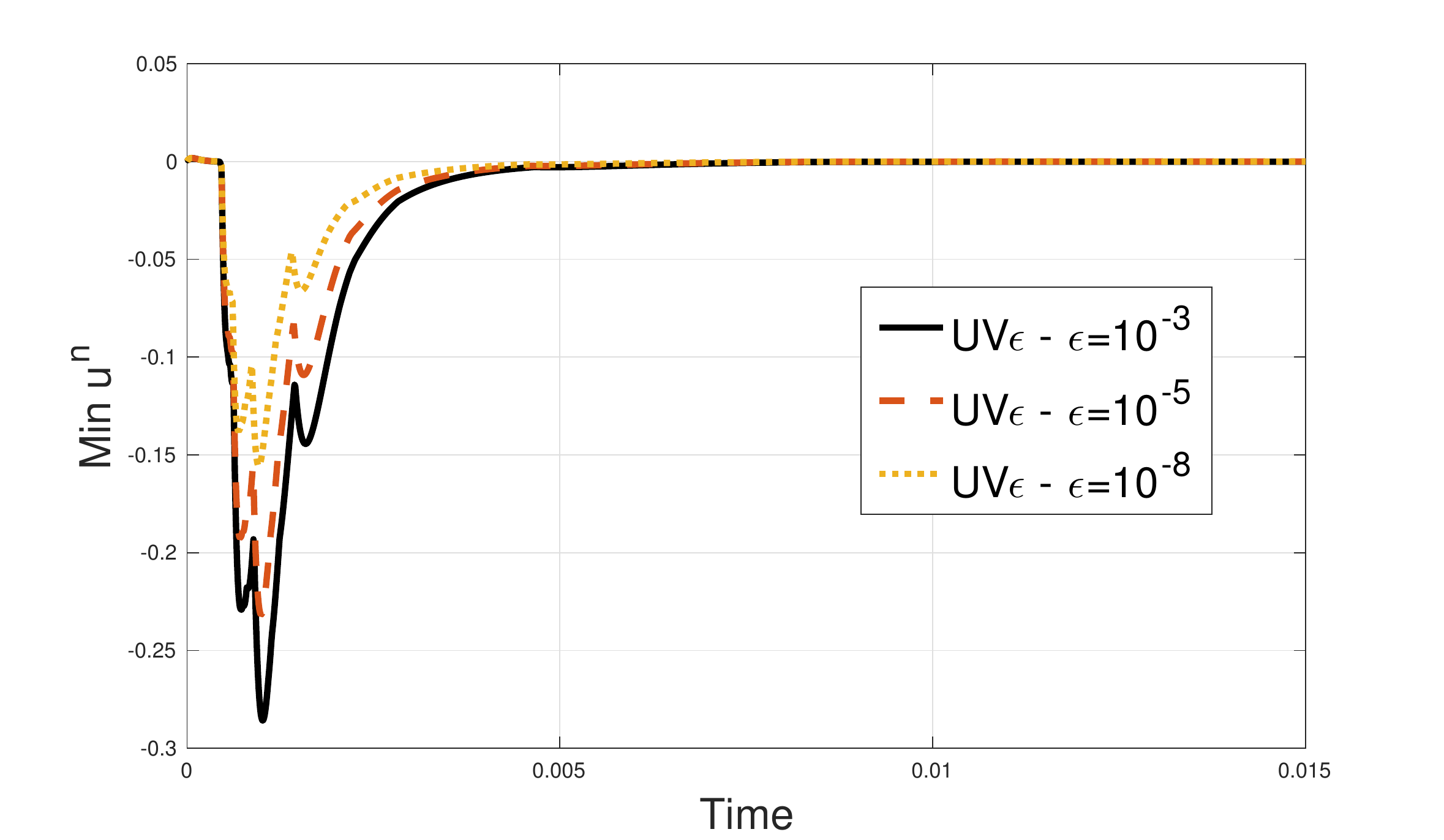}}
 \caption{Minimum values of $u^n_\varepsilon$ for $p=1.9$, computed using the scheme \textbf{UV$\varepsilon$}.
  \label{C5:fig:MU191}}
  \end{center}
\end{figure}
 \begin{figure}[h]
  \begin{center}
  {\includegraphics[height=0.5\linewidth]{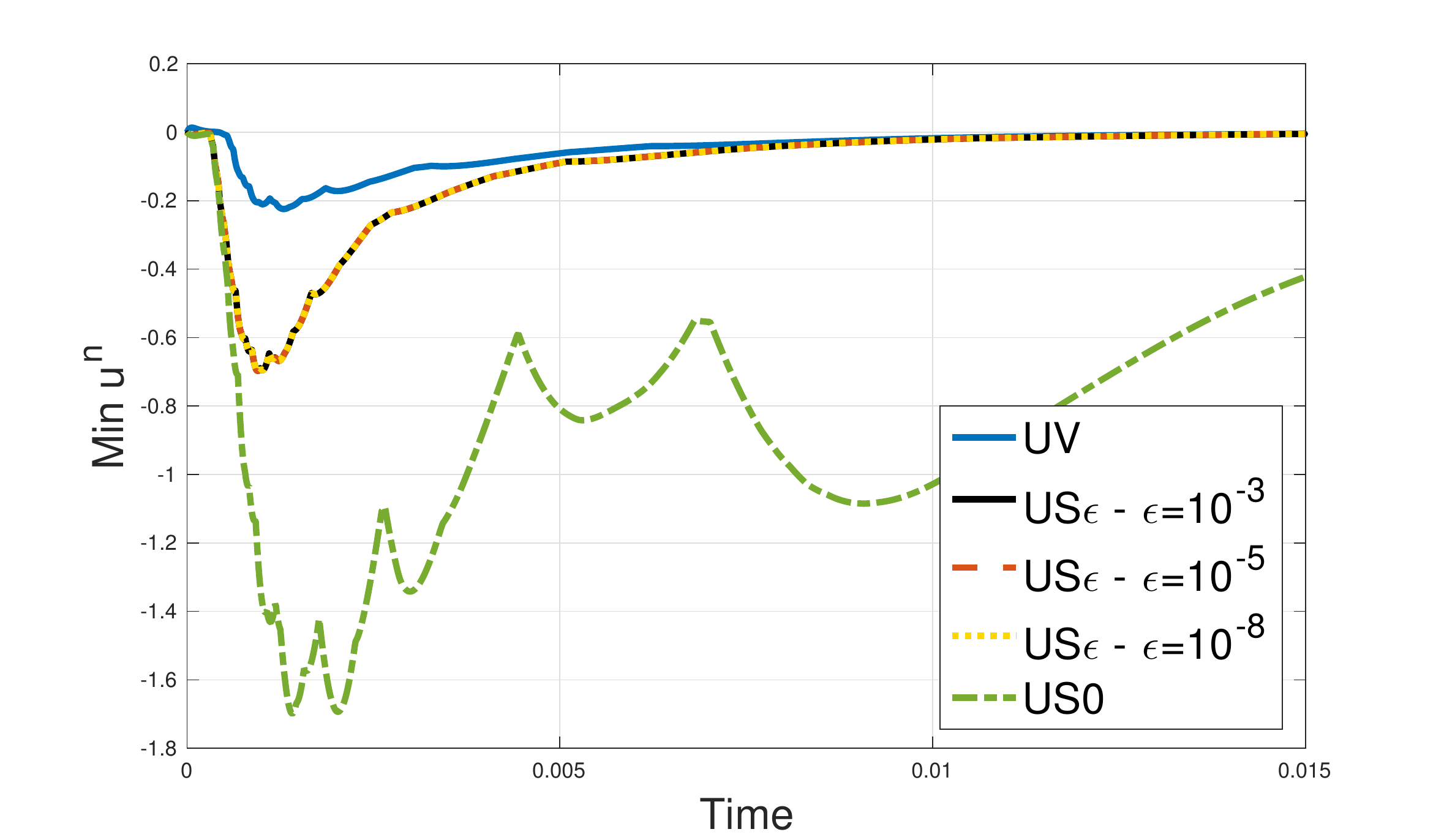}}
  \caption{Minimum values of $u^n$ for $p=1.9$, computed using the schemes \textbf{UV}, \textbf{US$\varepsilon$} and \textbf{US0}.
  \label{C5:fig:MU192}}
  \end{center}
\end{figure}

\subsection{Energy stability}
In this subsection, we compare numerically the stability of the schemes \textbf{UV$\varepsilon$}, \textbf{US$\varepsilon$}, $\textbf{US0}$ and $\textbf{UV}$ with respect to the ``exact'' energy
\begin{equation}\label{C5:EComun5}
 \mathcal{E}_e(u,v)= \displaystyle \int_\Omega  \frac{1}{p-1} (u_+)^p d\x + \frac{1}{2} \Vert \nabla {v}\Vert_0^2.
 \end{equation} 
It was proved that the schemes \textbf{UV$\varepsilon$}, \textbf{US$\varepsilon$} and $\textbf{US0}$ are unconditionally energy-stables with respect to modified energies defined in terms of the variables of each scheme, and some energy inequalities are satisfied (see Theorems  \ref{C5:estinc1uvreg}, \ref{C5:estinc1usreg} and \ref{C5:estinc1us}). However, it is not clear how to prove the energy-stability of these schemes with respect to the ``exact'' energy $ \mathcal{E}_e(u,v)$ given in (\ref{C5:EComun5}), which comes from the continuous problem (\ref{C5:modelf00}) (see (\ref{C5:wsd})-(\ref{C5:eneruva})). Therefore, it is interesting to compare numerically the schemes with respect to this energy $\mathcal{E}_e(u,v)$, and to study the behavior of the corresponding discrete energy law residual  
 \begin{equation}\label{C5:ns01-bb1}
RE_e(u^n,v^n):=\delta_t \mathcal{E}_e(u^n,v^n)+ \frac{4}{p}\int_\Omega \vert \nabla ((u^n_+)^{p/2})\vert^2 d\x
+\Vert \Delta_h v^{n}\Vert_{0}^{2}+\Vert \nabla v^{n}\Vert_{0}^{2}.
\end{equation}
We consider $k=10^{-5}$, $h=\frac{1}{25}$, $p=1.4$, $tol=10^{-3}$ and the initial conditions (see Figure \ref{C5:fig:initcond52})
$$u_0=14 cos(2\pi x) cos(2\pi y)+14.0001 \ \ \mbox{and} \ \ v_0=-14 cos(2\pi x) cos(2\pi y)+14.0001.$$ 
\begin{figure}[htbp]
\centering 
\subfigure[Initial cell density $u_0$]{\includegraphics[width=70mm]{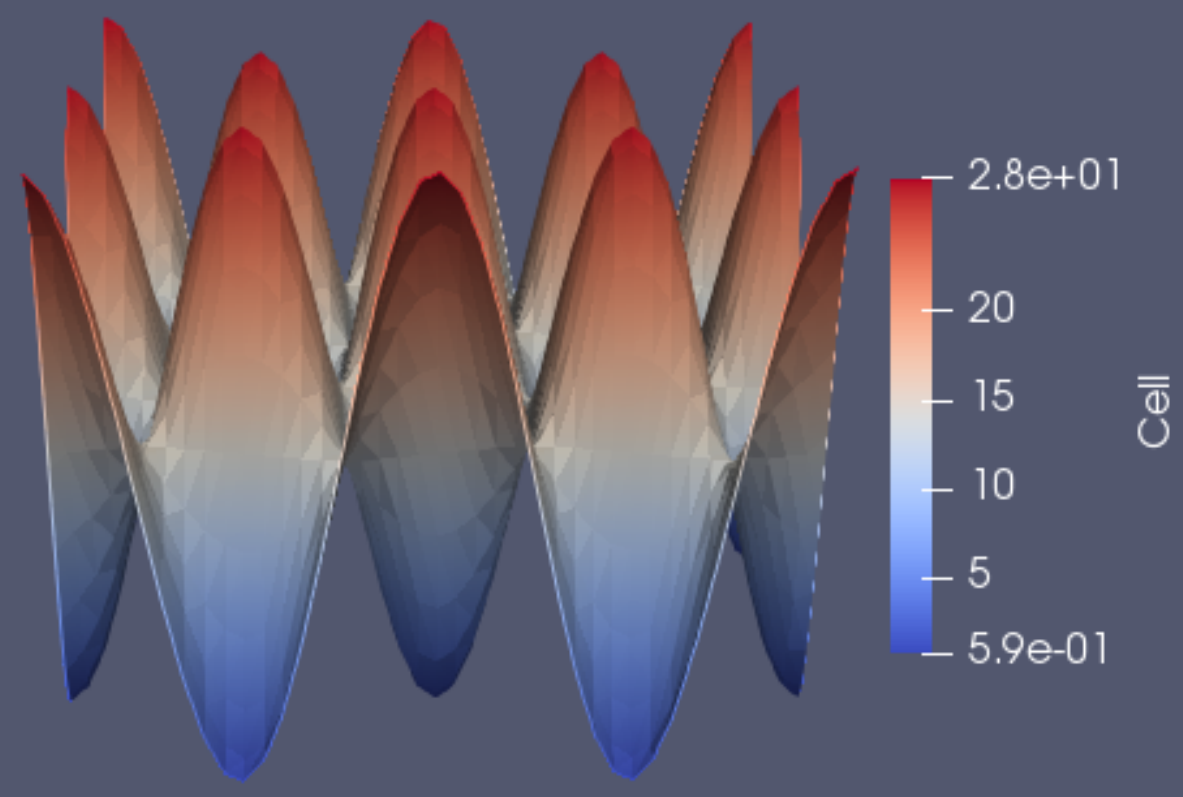}} \hspace{1,2 cm} 
\subfigure[Initial chemical concentration $v_0$]{\includegraphics[width=70mm]{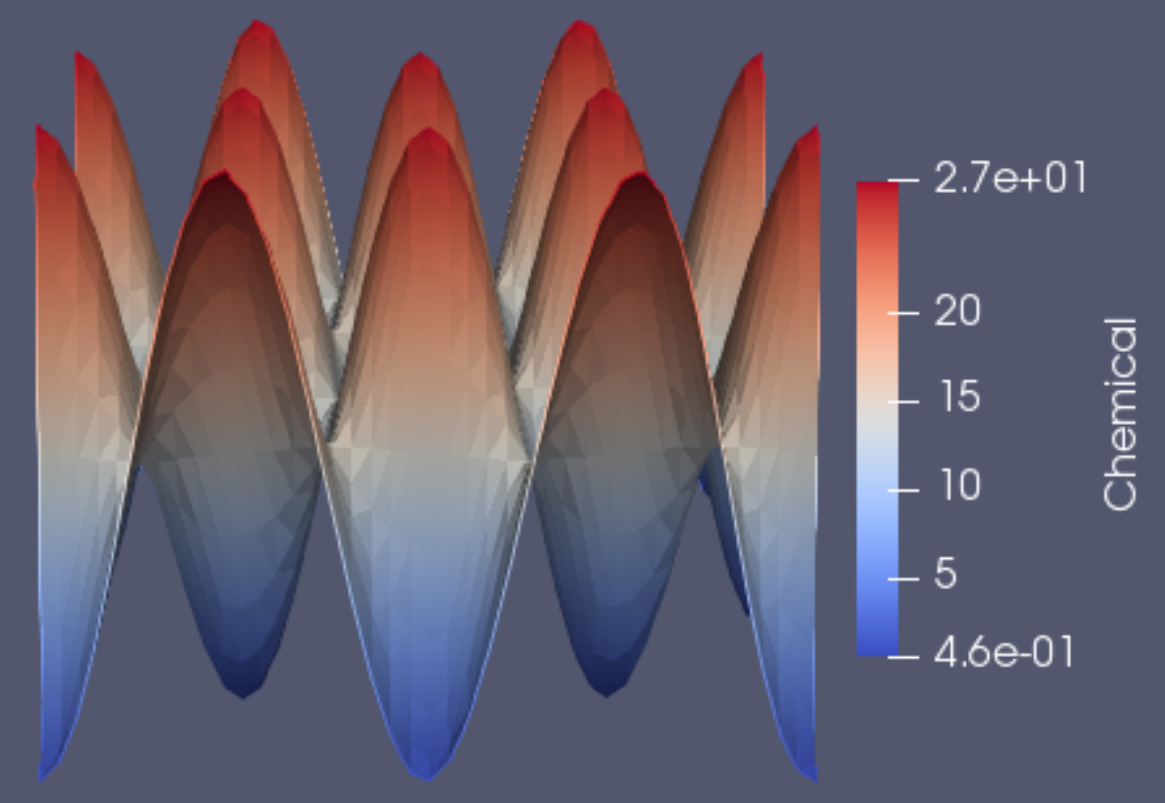}} 
\caption{Initial conditions.} \label{C5:fig:initcond52}
\end{figure}
We choose the space for $v$  generated by $\mathbb{P}_1$-continuous FE. Then, we obtain that: 
\begin{enumerate}
\item[(i)] All the schemes \textbf{UV$\varepsilon$}, \textbf{US$\varepsilon$}, $\textbf{UV}$ and $\textbf{US0}$ satisfy the energy decreasing in time property for the exact energy $\mathcal{E}_e(u,v)$ (see Figure \ref{C5:fig:Euv14}), that is, 
$$
 \mathcal{E}_e(u^n,v^n)\le \mathcal{E}_e(u^{n-1},v^{n-1}) \ \ \forall n.
$$
\item[(ii)] The schemes \textbf{US0} and \textbf{US$\varepsilon$} satisfy the discrete energy inequality $RE_e(u^n,v^n)\leq 0$, for $RE_e(u^n,v^n)$ defined in  (\ref{C5:ns01-bb1}), independently of the choice of $\varepsilon$; while the schemes $\textbf{UV}$ and \textbf{UV$\varepsilon$} have $RE(u^n,v^n)>0$ for some $t_n\geq 0$. However, it is observed that the scheme \textbf{UV$\varepsilon$} introduces lower numerical source than the scheme \textbf{UV}, and lower numerical dissipation than the schemes \textbf{US0} and \textbf{US$\varepsilon$} (see Figure \ref{C5:fig:REuv14}).
\end{enumerate}
 \begin{figure}[h]
  \begin{center}
  {\includegraphics[height=0.5\linewidth]{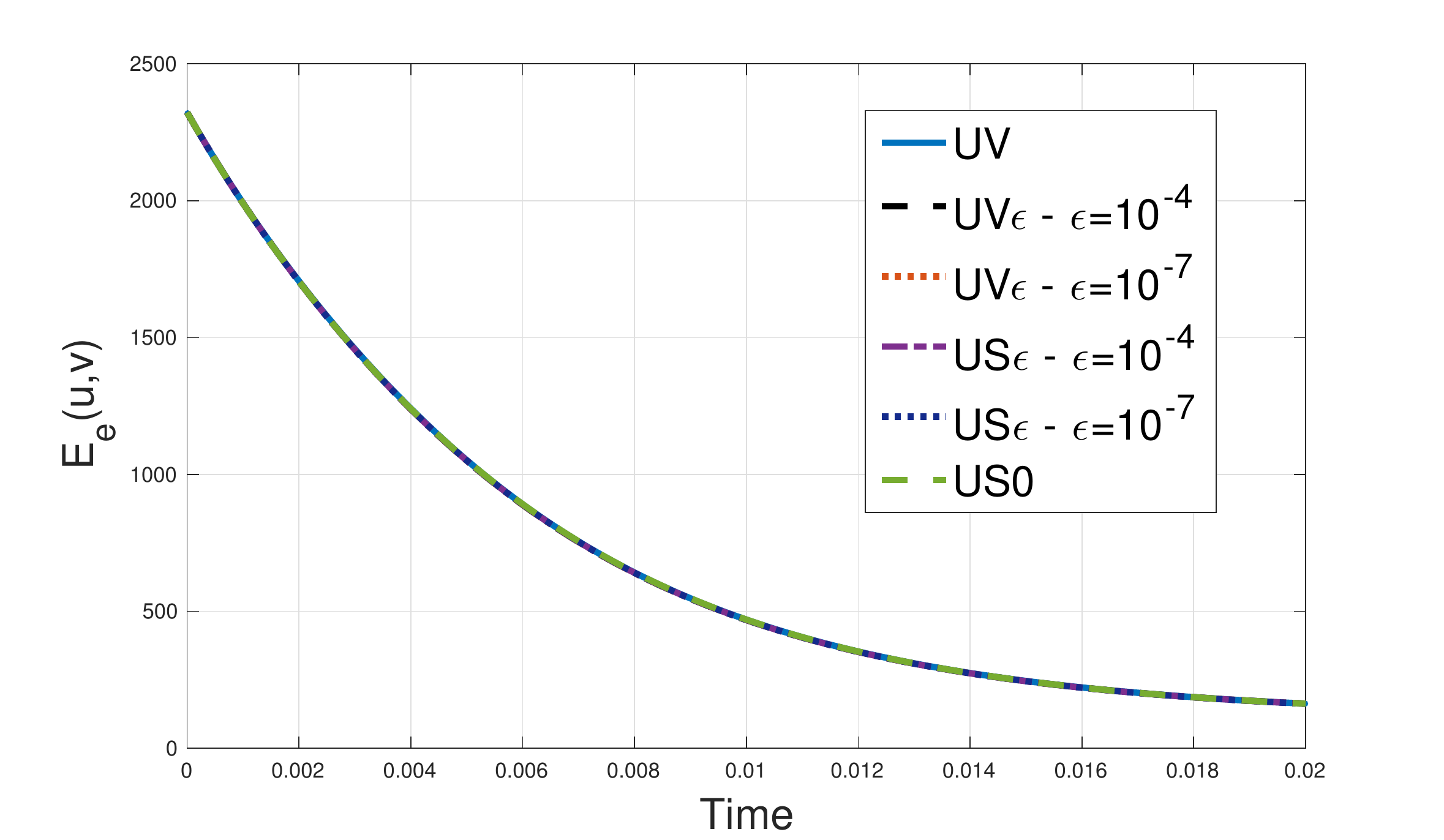}}
  \caption{$\mathcal{E}_e(u^n,v^n)$ of the schemes $\textbf{UV}$, $\textbf{US0}$,  \textbf{UV$\varepsilon$} and  \textbf{US$\varepsilon$} (for $\varepsilon=10^{-4},10^{-7}$).
  \label{C5:fig:Euv14}}
  \end{center}
\end{figure}
 \begin{figure}[h]
  \begin{center}
  {\includegraphics[height=0.5\linewidth]{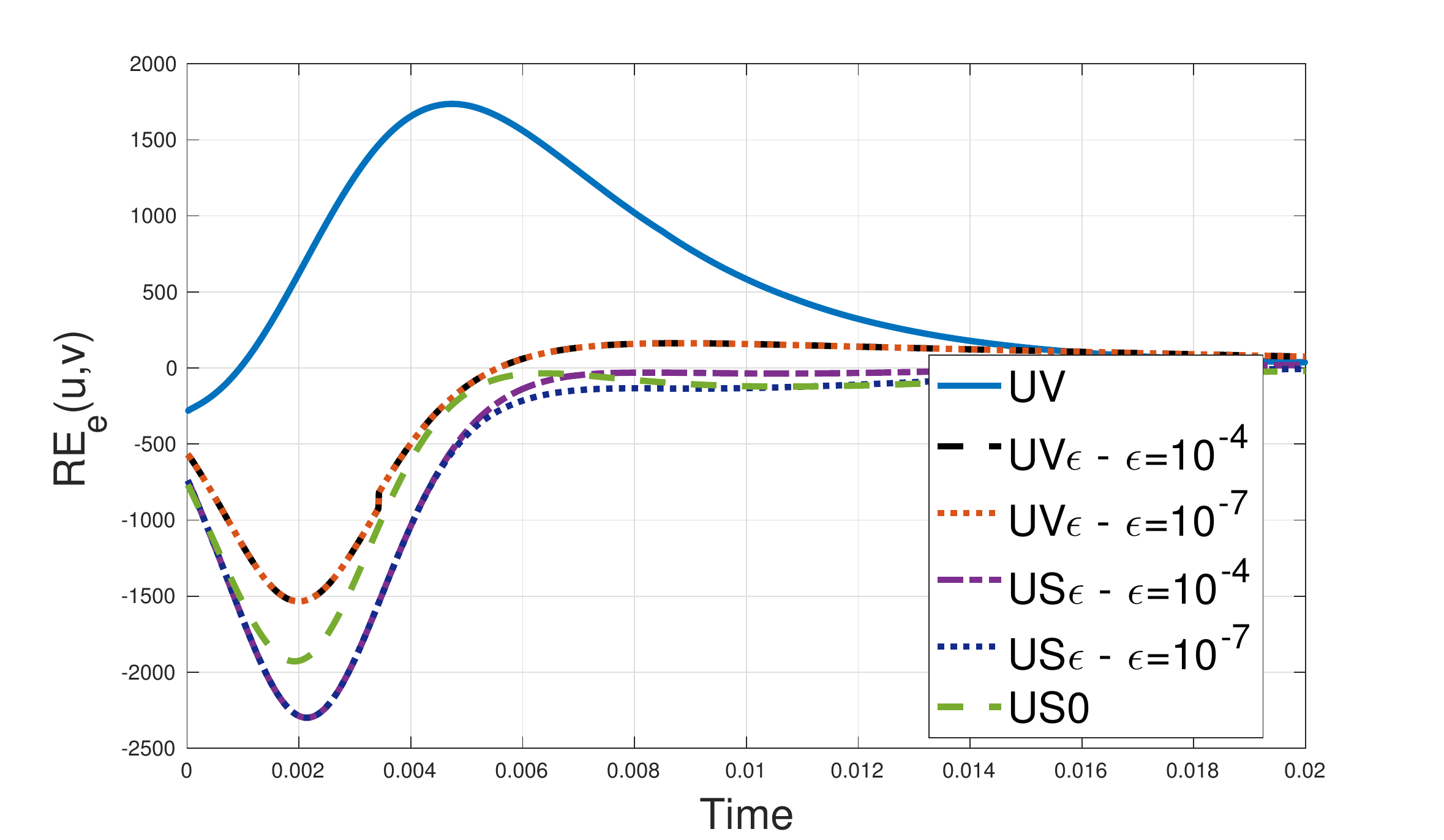}}
  \caption{$RE_e(u^n,v^n)$ of the schemes $\textbf{UV}$, $\textbf{US0}$,  \textbf{UV$\varepsilon$} and \textbf{US$\varepsilon$} (for $\varepsilon=10^{-4},10^{-7}$).
  \label{C5:fig:REuv14}}
  \end{center}
\end{figure}

\section{Conclusions}\label{C5:S6C5:C}
In this paper we have developed three new mass-conservative and unconditionally energy-stable fully discrete FE schemes for the chemorepulsion production model (\ref{C5:modelf00}), namely \textbf{UV$\varepsilon$}, \textbf{US$\varepsilon$} and \textbf{US0}.
From the theoretical point of view we have obtained: 
\begin{enumerate}
\item[(i)] The solvability of the nume\-ri\-cal schemes.
\item[(ii)] The schemes  \textbf{UV$\varepsilon$} and \textbf{US$\varepsilon$} are unconditionally energy-stables with respect to the modified energies $\mathcal{E}^h_\varepsilon(u,v)$ (given in (\ref{C5:ENuvreg})) and $\mathcal{E}^h_\varepsilon(u,{\boldsymbol\sigma})$ (given in (\ref{C5:ENusreg})) respectively, under the right-angles constraint ({\bf H}); while the scheme \textbf{US0} is unconditionally energy-stable with respect to the modified energy $\mathcal{E}^h(u,{\boldsymbol\sigma})$ given in (\ref{C5:ENus}), without this restriction ({\bf H}) on the mesh. 
\item[(iii)] It is not clear how to prove the energy-stability of the nonlinear scheme \textbf{UV} (see Remark \ref{C5:RBE5}).
\item[(iv)] In the schemes  \textbf{UV$\varepsilon$} and \textbf{US$\varepsilon$} there is a control for $\Pi^h (u^n_{\varepsilon-})$ in $L^2$-norm, which tends to $0$ as $\varepsilon\rightarrow 0$. This allows to conclude the nonnegativity of the solution $u^n_\varepsilon$ in the limit as $\varepsilon\rightarrow 0$. 
\end{enumerate}
On the other hand, from the numerical simulations, we can conclude:
\begin{enumerate}
\item[(i)] The four schemes have decreasing in time energy $\mathcal{E}_e(u,v)$, independently of the choice of $\varepsilon$.   
\item[(ii)] The schemes \textbf{US0} and \textbf{US$\varepsilon$} satisfy the discrete energy inequality $RE_e(u^n,v^n)\leq 0$, for $RE_e(u^n,v^n)$ defined in  (\ref{C5:ns01-bb1}), independently of the choice of $\varepsilon$; while the schemes $\textbf{UV}$ and \textbf{UV$\varepsilon$} have $RE(u^n,v^n)>0$ for some $t_n\geq 0$. However, it was observed that the scheme \textbf{UV$\varepsilon$} introduces lower numerical source than the scheme \textbf{UV}, and lower numerical dissipation than the schemes \textbf{US0} and \textbf{US$\varepsilon$}.
\item[(iii)] Finally, it was observed numerically that for the schemes \textbf{UV$\varepsilon$} and \textbf{US$\varepsilon$}, $\underset{\overline{\Omega}\times[0,T]}{\min}\ u^n_{\varepsilon} \rightarrow 0$ as $\varepsilon\rightarrow 0$.
\end{enumerate}

\section*{Acknowledgements}
The authors have been partially supported by MINECO grant MTM2015-69875-P
(Ministerio de Econom\'{\i}a y Competitividad, Spain) with the participation of FEDER.
The third author have also been supported by Vicerrector\'ia de Investigaci\'on y Extensi\'on of Universidad Industrial de Santander.

\end{document}